\documentclass[11pt]{amsart}

\usepackage[a4paper,margin=1in]{geometry}
\usepackage{amsmath,amssymb,amsfonts,amsthm,amsrefs,bbm}
\usepackage{mathrsfs}
\usepackage{hyperref}
\hypersetup{colorlinks=true,linkcolor=blue,citecolor=blue,urlcolor=blue}
\usepackage{enumitem}
\usepackage{mathtools}
\usepackage{amscd}
\usepackage{tikz-cd}
\mathtoolsset{showonlyrefs}

\newtheorem{theorem}{Theorem}[section]
\newtheorem{lemma}[theorem]{Lemma}
\newtheorem{proposition}[theorem]{Proposition}
\newtheorem{corollary}[theorem]{Corollary}
\theoremstyle{definition}
\newtheorem{definition}[theorem]{Definition}
\theoremstyle{remark}
\newtheorem{remark}[theorem]{Remark}
\newtheorem{observation}[theorem]{Observation}
\newtheorem{conventions}[theorem]{Conventions}

\newtheorem{example}[theorem]{Example}

\newcommand{\T}{\mathbb{T}}
\newcommand{\Z}{\mathbb{Z}}
\newcommand{\R}{\mathbb{R}}
\newcommand{\C}{\mathbb C}
\newcommand{\N}{\mathbb N}
\newcommand{\Q}{\mathbb Q}
\newcommand{\K}{\ensuremath{\mathcal{K}}}
\newcommand{\GL}{\mathrm{GL}}

\newcommand{\SO}{\mathrm{SO}}

\newcommand{\Aut}{\mathrm{Aut}}

\newcommand{\rank}{\mathrm{rank}}
\newcommand{\aue}{\sim_{\mathrm{a.u.e.}}}
\newcommand{\Tr}{\mathrm{Tr}}

\newcommand{\id}{\mathrm{Id}}
\newcommand{\Ell}{\mathrm{Ell}}
\newcommand{\Ped}{\mathrm{Ped}}
\newcommand{\Aff}{\mathrm{Aff}}
\newcommand{\LAff}{\mathrm{LAff}}
\newcommand{\ex}{\mathrm{pro}}
\newcommand{\ev}{\mathrm{ev}}
\newcommand{\dom}{\mathrm{dom}}
\newcommand{\End}{\mathrm{End}}

\title[Noncommutative protori and inductive spectral triples]{Noncommutative protori and inductive spectral triples}

\author{Remus Floricel}
\address{University of Regina, Department of Mathematics and Statistics, Regina, SK, Canada}
\email{Remus.Floricel@uregina.ca} 
\author{Patrick Melanson}
\address{University of Regina, Department of Mathematics and Statistics, Regina, SK, Canada}
\email{pdm456@uregina.ca} 
\date{\today}


\thanks {The research of R.F. was partially funded by an NSERC Discovery Grant.}

\begin{document}

\begin{abstract}
We study inductive limits of higher-dimensional noncommutative tori, which we
call noncommutative protori. We compute the Elliott invariants for broad classes
of unital and nonunital systems, including toric maps, Morita-corner embeddings,
and dimension-changing and proper embeddings. For the resulting simple limits we
determine explicitly the ordered $K$-groups, trace cone, scale, and projection
scale, yielding concrete classification criteria. We also construct compatible
spectral triples and locally compact spectral triples on these limits via
Fourier- and Morita-compatible Dirac structures.
\end{abstract}
\maketitle
{\small\tableofcontents}

\section{Introduction and Background}

Noncommutative tori \cite{Rieffel1981,RieffelNC} are among the fundamental
examples in noncommutative geometry \cite{ConnesBook} and the classification theory of
$C^*$-algebras. They are sufficiently tractable to admit explicit
computations of $K$-theory, traces, Morita equivalence, and Dirac-type operators, while still rich enough to exhibit
genuinely noncommutative phenomena \cite{Elliott-Evans, ConnesRieffelYM, CDV1, ConnesTretkoff, FathizadehKhalkhaliScalar, RieffelNC, RieffelSchwarz}.

In this paper, we study inductive limits of such
algebras. More precisely, we consider systems of the form
\[
A_{\ex}\;\cong\;\varinjlim (B_n,\phi_n),\qquad B_n:=M_{r_n}(A_{\Theta_n})\;\cong\;M_{r_n}(\C)\otimes A_{\Theta_n},
\]
where each $A_{\Theta_n}$ is a simple higher-dimensional noncommutative torus
and the connecting maps $\phi_n$ are nonzero $*$-homomorphisms. We call such limits noncommutative protori, or protoral $C^*$-
algebras. The terminology is motivated by the classical fact that protori \cite{HM1, HM2}, that is, compact connected abelian groups, are inverse limits of ordinary tori; after passing to continuous function algebras, inverse limits of spaces become direct limits of commutative $C^*$-algebras.

The two-dimensional unital case has already appeared in an important form in
the work of Latr\'emoli\`ere and Packer on noncommutative solenoids
\cite{LatPackSol}.  Their noncommutative solenoids are realized as inductive
limits of irrational rotation algebras, and their $K$-theory and classification
are described explicitly in terms of the solenoid parameter.  Subsequent work
has also studied spectral triples on noncommutative solenoids and their metric
aspects
\cite{FarsiLandryLarsenPackerSolenoids,FarsiLatremolierePackerStandardSolenoids}.
One purpose of the present paper is to place these examples within a broader
higher-dimensional framework.  From our point of view, the solenoid examples
arise as the special case in which the building blocks are two-dimensional
rotation algebras and the connecting maps are unital toric maps induced by
scalar integer matrices.

The main novelty of the present work lies in allowing a substantially broader class of connecting maps. In addition to unital toric maps arising from integer lattice transformations, we also consider maps that factor through full Morita corners, maps that change dimension, and same-dimensional embeddings whose image is a proper unital subalgebra of a corner. Consequently, the class studied here encompasses both unital and nonunital inductive systems. The nonunital case is particularly significant, as the trace scale and projection scale become part of the invariant, capturing information that is not visible from the abstract ordered $K$-groups alone.

The first part of the paper studies homomorphisms between simple
noncommutative tori.  We use classification results for homomorphisms between
simple tracial rank zero $C^*$-algebras \cite{LinHom, LinRange}  to formulate $K$-theoretic and
trace-theoretic criteria for the existence and uniqueness, up to approximate
unitary equivalence, of unital and nonunital maps between noncommutative tori
and their matrix amplifications (see Theorems~\ref{thm:unital-classification} and \ref{thm:nonunital-classification}).  We then identify concrete representatives
among these maps.  In particular, we show in Proposition~\ref{prop:toric-explicit} that, if $M\in M_{n\times m}(\Z)$ and $\Theta\equiv M^t\Psi M \pmod{M_m(\Z)_{\mathrm{skew}}},$
then there is a monomial $*$-homomorphism $\varphi_{M,z}\colon A_\Theta\longrightarrow A_\Psi$ such that 
$\varphi_{M,z}(U_j)=z_jV^{Me_j}.$
On $K$-theory this map is given by the exterior powers
$\Lambda^{\mathrm{even}}(M)$ and $\Lambda^{\mathrm{odd}}(M)$ (see Proposition~\ref{prop:Ktoric-explicit}).  We also
describe how Morita equivalences and full corners produce nonunital connecting
maps (Theorem~\ref{thm:case2-explicit}, Theorem~\ref{thm:case3-explicit}, and Corollary~\ref{cor:case3-from-case4}). These particular $*$-homomorphisms, grouped into what we call Case~(1), Case~(2), Case~(3), and Case~(4), provide the basic supply of inductive systems used throughout the paper.

The second part computes the Elliott invariants of these inductive limits. We first show in Proposition~\ref{prop:protorus-regularity} that $A_{\ex}=\varinjlim(B_n,\phi_n)$ is simple, separable, nuclear, satisfies the Universal Coefficient Theorem (UCT), and
has real rank zero and stable rank one.  Moreover,
$K_i(A_{\ex}) \cong\varinjlim \bigl(K_i(B_n),(\phi_n)_{*i}\bigr),$ and the positive cone on $K_0(A_{\ex})$ is the corresponding direct-limit cone. Since each
building block has a unique trace, the cone of densely defined lower
semicontinuous traces on $A_{\ex}$ is one-dimensional. The
scale function is given by $\Sigma_{A_{\ex}}(\lambda\tau)=\lambda\sup_n c_n,$
and the projection scale is $\Sigma(K_0(A_{\ex}))=\bigcup_{n\geq1}
[\,0,(\iota_n)_*([1_{B_n}])\,]$ (see Theorem~\ref{thm:protorus-invariant}). 
These formulas make the Elliott invariant completely explicit for the systems
considered here.

This invariant computation is one of the central points of the paper.  It
turns the classification of noncommutative protori into a concrete calculation
with ordered direct limits, trace ranges, trace scales, and projection scales.
In the unital case, the scale is simply the order unit.  In the stable case,
the trace scale is infinite and the projection scale is the full positive cone.
In intermediate nonunital cases, however, the projection scale can be a proper
subset of the positive cone, and it can distinguish limits that would otherwise
look similar at the level of $K$-groups and trace pairings.  This is why the
nonunital scale data are treated explicitly throughout the paper.

We then work out a series of examples.  The unital toric $N$-solenoid recovers
the noncommutative solenoid picture of Latr\'emoli\`ere and Packer as a special
case (see Remark~\ref{rem:latpack}).  Pure-corner systems give stable noncommutative protori (Example~\ref{ex:stable-corner}). Dimension-changing
systems produce limits whose $K$-groups are exterior algebras on countably
generated lattices (Example~\ref{ex:dimension-changing}).  Same-dimensional noncorner systems yield nonunital
examples with projection scales governed by strict trace inequalities (Example~\ref{ex:AX7-noncorner}).  We also
construct unital systems in which the connecting maps are designed to change
$K_1$ while leaving $K_0$ fixed (Example~\ref{ex:exotic-K1-engine}), and higher-dimensional examples in which
infinitesimal $K_0$-classes are killed in the limit (Example~\ref{ex:infinitesimal-killing}).  These examples illustrate
that inductive limits of noncommutative tori can display behavior that is not
visible at any single finite stage.

The final part of the paper studies spectral triples on noncommutative protori.
In the unital toric case, the connecting maps act directly on Fourier labels.
We first isolate a flexible single-stage construction: a selfadjoint
matrix-valued Fourier multiplier
$F\colon \mathbb Z^m\to \End(S)$ whose spectrum escapes to infinity and whose
translation increments are bounded gives a spectral triple on $A_\Theta$
(Proposition~\ref{prop:single-torus-fourier-multiplier}).  This recovers the
usual flat Dirac operators and length-function triples, and it is stable under
compatible bounded perturbations and right-conformal deformations.  For an
inductive system, Fourier-compatible data $(F_n,S_n,J_n)$, satisfying the
intertwining relation $F_{n+1}(M_nx)J_n=J_nF_n(x)$, produce an inductive system of spectral triples
(Proposition~\ref{prop:unital-fourier-multiplier-triples}), to which we apply
the inductive-limit construction from \cite{FloricelGhorbanpour}.  We then give
a convenient direct-limit criterion, formulated on the label group
$\varinjlim(\mathbb Z^{m_n},M_n)$, ensuring compact resolvent and bounded
commutators for the limit operator (Corollary~\ref{cor:limit-multiplier-criterion}).
The unital $N$-solenoid examples show both a natural flat compatible operator
which fails compact resolvent (Example~\ref{ex:solenoid-spectral}) and a length-type modification which restores
compact resolvent (Example~\ref{ex:compact-resolvent-solenoid-dirac}); compatible inner fluctuations (Example~\ref{ex:solenoid-inner-fluctuation}) and right-conformal
deformations (Example~\ref{ex:solenoid-right-conformal}) are also constructed.

For nonunital systems the situation is different.  The connecting maps factor
through full corners and are no longer unital, so the natural limiting objects
are locally compact spectral triples rather than ordinary unital spectral
triples.  We formulate Morita-compatible Dirac data using smooth full
projections, smooth corner isomorphisms, trace-rescaled GNS isometries, and
stagewise operators intertwined through the toric, Morita, and corner pieces (see Definition~\ref{def:morita-compatible-dirac-data}).
Theorem~\ref{thm:protorus-spectral-general} proves that the resulting
inductive-limit operator has bounded commutators exactly when the finite-stage
commutators are uniformly bounded, and it reduces local compactness to the
compactness of the cut-down resolvents $\pi(\iota_n(1_{B_n}))(1+D_{\ex}^2)^{-1/2}$.  We verify these criteria in
three model families: the pure-corner stable protorus (Example~\ref{ex:stable-corner-trace-gns-spectral}), the dimension-changing
weighted-length model (Example~\ref{ex:dimension-changing-weighted-length-spectral}), and the same-dimensional noncorner flag-filtration
model (Example~\ref{ex:AX7-flag-filtration-spectral}).  These examples explain why additional weights or filtrations are needed:
the limit operator generally cannot have compact resolvent, but local
compactness is recovered after cutting down by finite-stage units.

The importance of these constructions is twofold.  On the $C^*$-algebraic
side, they provide a computable class of simple classifiable inductive limits
built from noncommutative tori, extending the known solenoid examples to a much
larger range of higher-dimensional and nonunital systems.  On the geometric
side, they show that compatible Dirac-type data can be transported through the
same inductive systems, giving spectral and locally compact spectral triples on
the resulting noncommutative protori.  Thus the paper connects explicit
Elliott-invariant classification with noncommutative metric geometry in a
single inductive-limit framework.

We conclude this introductory section by briefly recalling the definition and basic structure of simple higher-dimensional noncommutative tori. Additional properties and notation will be introduced later, as needed.

Let $m\geq 2$, and let $\Theta\in M_m(\R)$ be skew-symmetric. The
noncommutative $m$-torus $A_\Theta$ \cite{RieffelNC} is the universal
unital $C^*$-algebra generated by unitaries
$U_1,\dots,U_m$ satisfying
\begin{equation}\label{eq:nctorus-rel}
U_kU_j=e^{2\pi i\Theta_{j,k}}\,U_jU_k,
\qquad 1\leq j,k\leq m.
\end{equation}

It is sometimes convenient to use the twisted-group picture. For
$x=(x_1,\dots,x_m)\in\Z^m$, set $U^x:=U_1^{x_1}\cdots U_m^{x_m}.$
Then the relations \eqref{eq:nctorus-rel} imply
\begin{equation}\label{eq:ordered-commutator}
U^yU^x
=
\exp\bigl(2\pi i\langle x,\Theta y\rangle\bigr)\,U^xU^y,
\qquad x,y\in\Z^m.
\end{equation}
Equivalently, $A_\Theta\cong C^\ast(\Z^m,\omega_\Theta),$
where, for the convention
$\lambda_x\lambda_y=\omega_\Theta(x,y)\lambda_{x+y}$, one may take the
normalized $2$-cocycle
$\omega_\Theta(x,y)
=
\exp\left(
2\pi i\sum_{1\leq j<k\leq m}\Theta_{j,k}x_ky_j
\right).$
This cocycle is cohomologous to the alternating cocycle
$\sigma_\Theta(x,y)=\exp\bigl(-\pi i\langle x,\Theta y\rangle\bigr).$

We say that $\Theta$ is nondegenerate \cite{Slawny} if the only
$x\in\Z^m$ satisfying $\exp(2\pi i\langle x,\Theta y\rangle)=1$
for all $y\in\Z^m$
is $x=0$.

\begin{theorem}[Phillips]\label{thm:phillips}
If $\Theta$ is nondegenerate, then $A_\Theta$ is a simple AT algebra with
real rank zero and stable rank one. It has a unique tracial state, denoted
$\tau_\Theta$. In particular, $A_\Theta$ is separable, nuclear, satisfies
the UCT, and has tracial rank zero. Moreover,
\begin{equation}\label{eq:K-exterior}
K_0(A_\Theta)\cong \Lambda^{\mathrm{even}}\Z^m,
\qquad
K_1(A_\Theta)\cong \Lambda^{\mathrm{odd}}\Z^m,
\end{equation}
so both $K$-groups are free abelian of rank $2^{m-1}$.
\end{theorem}

\begin{proof}
For nondegenerate $\Theta$, the algebra $A_\Theta$ is simple and has at
most one tracial state by \cite[Theorem~3.7 and Lemma~3.1]{Slawny}. The
canonical trace on the dense Fourier $*$-subalgebra gives existence. The AT
and real rank zero assertions are \cite[Theorem~3.8]{PhillipsAT}. Stable rank
one and tracial rank zero follow from the structure of simple AT algebras with
real rank zero. Since $A_\Theta$ is a twisted group $C^*$-algebra of
$\Z^m$, it is separable, nuclear, and satisfies the UCT. Finally, the
$K$-theory identifications in \eqref{eq:K-exterior} follow from repeated use
of the Pimsner--Voiculescu exact sequence, as recalled in
\cite{PhillipsAT,RieffelProj}.
\end{proof}

For nondegenerate $\Theta$, the order on $K_0(A_\Theta)$ is determined by
the unique trace. Write \[\rho_\Theta:=(\tau_\Theta)_*\colon K_0(A_\Theta)\to\R\]
for the trace pairing. Elliott computed $\rho_\Theta(K_0(A_\Theta))$
explicitly in terms of Pfaffians of even-dimensional submatrices of
$\Theta$ \cite{Ell84}. For our purposes, the key point is that the positive
cone is trace-determined \cite[Theorem~6.1]{RieffelProj}:
\begin{equation}\label{eq:order-by-trace}
K_0(A_\Theta)^+
=
\{x\in K_0(A_\Theta):\rho_\Theta(x)>0\}\cup\{0\}.
\end{equation}

\section{Homomorphisms of simple noncommutative tori}

In this section we separate two issues:
\begin{enumerate}[label=(\arabic*)]
\item the classification of all homomorphisms between simple noncommutative tori up to approximate unitary equivalence, and
\item the construction of concrete monomial representatives coming from integer matrices.
\end{enumerate}
The first point is an application of Phillips's structure theorem for simple higher-dimensional noncommutative tori together with Lin's classification of unital monomorphisms from AH-algebras into simple tracial-rank-zero algebras \cite{LinRange}. The second point recovers the toric maps that are useful for explicit constructions.

\subsection{The general classification}

We begin by recalling the construction of the Dadarlat--Loring $KL$-group $KL(A,B)$; see
\cite{DadarlatLoring}. 
Let $A$ and $B$ be separable $C^*$-algebras, with
$A$ satisfying the UCT. The UCT yields the short exact sequence
\[
0\longrightarrow
\operatorname{Ext}^1_{\mathbb Z}\bigl(K_*(A),K_{*+1}(B)\bigr)
\longrightarrow
KK(A,B)
\longrightarrow
\operatorname{Hom}_{\mathbb Z}
\bigl(K_*(A),K_*(B)\bigr)
\longrightarrow 0.
\]
The group $KL(A,B)$ is defined by
\[
KL(A,B):=
KK(A,B)/
\operatorname{Pext}\bigl(K_*(A),K_{*+1}(B)\bigr),
\]
where $\operatorname{Pext}$ denotes the subgroup of pure extension classes in
the UCT kernel. Thus every class $\alpha\in KL(A,B)$ induces a graded homomorphism
\[\alpha_*=(\alpha_{*0},\alpha_{*1})\colon K_*(A)\to K_*(B).\]
Conversely, for noncommutative tori, we have the following result.
\begin{lemma}
\label{lem:kl}
Let $\Theta$ and $\Psi$ be nondegenerate skew-symmetric matrices. Then every graded group homomorphism $\kappa\colon K_*(A_\Theta)\to K_*(A_\Psi)$
determines a unique class in $KL(A_\Theta,A_\Psi)$. The same holds with
$A_\Psi$ replaced by any nonzero corner $pM_k(A_\Psi)p$.
\end{lemma}

\begin{proof}
By Theorem~\ref{thm:phillips}, the algebras $A_\Theta$ and $A_\Psi$ are
separable, nuclear, and satisfy the UCT. Moreover, $K_*(A_\Theta)=
K_0(A_\Theta)\oplus K_1(A_\Theta)
\cong
\Lambda^{\mathrm{even}}\mathbb Z^m
\oplus
\Lambda^{\mathrm{odd}}\mathbb Z^m$ as a $\mathbb Z/2$-graded group. In particular, $K_*(A_\Theta)$ is free
abelian. Hence $\operatorname{Ext}^1_{\mathbb Z}
\bigl(K_*(A_\Theta),K_{*+1}(A_\Psi)\bigr)=0,$
and the UCT identifies $KK(A_\Theta,A_\Psi)
\cong
\operatorname{Hom}_{\mathbb Z}
\bigl(K_*(A_\Theta),K_*(A_\Psi)\bigr).$
Thus every graded homomorphism $\kappa$ determines a unique $KK$-class, and
hence a unique $KL$-class.

Now let $pM_k(A_\Psi)p$ be a nonzero corner. Since $A_\Psi$ is simple,
$M_k(A_\Psi)$ is simple, and every nonzero projection
$p\in M_k(A_\Psi)$ is full. Therefore $pM_k(A_\Psi)p$ is Morita equivalent
to $A_\Psi$. In particular, it is separable, nuclear, satisfies the UCT, and
has the same $K$-theory up to the usual Morita identification. Applying the
same UCT argument with $pM_k(A_\Psi)p$ in place of $A_\Psi$ proves the
corner case.
\end{proof}

The following result classifies unital $*$-homomorphisms between simple
higher-dimensional noncommutative tori up to approximate unitary equivalence. To this end, recall that if $A$ and $B$ are unital $C^*$-algebras, then two $*$-homomorphisms
$\phi,\psi\colon A\to B$ are said to be approximately unitarily equivalent,
written $\phi\aue\psi$, if there exists a sequence of unitaries
$u_n\in B$ such that $\lim_{n\to\infty}\|u_n\phi(a)u_n^*-\psi(a)\|=0$
for all $a\in A.$

\begin{theorem}
\label{thm:unital-classification}
Let $\Theta\in M_m(\R)$ and $\Psi\in M_n(\R)$ be nondegenerate
skew-symmetric matrices. Let $\kappa=(\kappa_0,\kappa_1)\colon K_*(A_\Theta)\to K_*(A_\Psi)$
be a graded group homomorphism. Then the following are equivalent:
\begin{enumerate}[label=\textup{(\roman*)}]
\item There exists a unital $*$-homomorphism
$\phi\colon A_\Theta\to A_\Psi$ with
$\phi_{*i}=\kappa_i$ for $i=0,1$.
\item The $K_0$-map satisfies
\begin{equation}\label{eq:unital-trace-compatibility}
\kappa_0([1_{A_\Theta}])=[1_{A_\Psi}]
\qquad\text{and}\qquad
\rho_\Psi\circ\kappa_0=\rho_\Theta.
\end{equation}
\end{enumerate}
Moreover, if $\phi,\psi\colon A_\Theta\to A_\Psi$ are unital
$*$-homomorphisms, then $\phi\aue\psi$
if and only if $\phi_{*0}=\psi_{*0}$
and $\phi_{*1}=\psi_{*1}.$
Consequently, $\mathrm{Hom}_1(A_\Theta,A_\Psi)/\!\aue$
is naturally in bijection with the set of graded homomorphisms
$\kappa$ satisfying \eqref{eq:unital-trace-compatibility}.
\end{theorem}

\begin{proof}

Assume \textup{(i)}. Since $A_{\Theta}$ and $A_{\Psi}$ have unique tracial states, $\tau_{\Psi}\circ\phi$ is a tracial state on $A_{\Theta}$ and hence equals $\tau_{\Theta}$. Passing to $K_0$ gives $\rho_{\Psi}\circ\phi_{\ast 0}=\rho_{\Theta}.$
Since $\phi$ is unital, $\phi_{\ast 0}([1_{A_{\Theta}}])=[1_{A_{\Psi}}]$. Thus \textup{(ii)} holds.

Now assume \textup{(ii)}. By Lemma~\ref{lem:kl}, the graded homomorphism $\kappa$ determines a unique class, still denoted $\kappa$, in $KL(A_{\Theta},A_{\Psi})$. The trace identity in \eqref{eq:unital-trace-compatibility}, and \eqref{eq:order-by-trace}, imply that $\kappa_0$ is strictly positive: if $x\in K_0(A_{\Theta})^+\setminus\{0\}$, then $\rho_{\Theta}(x)>0$, hence $\rho_{\Psi}(\kappa_0(x))=\rho_{\Theta}(x)>0$. Therefore $\kappa_0(x)\in K_0(A_{\Psi})^+\setminus\{0\}$ by \eqref{eq:order-by-trace}.

Define a unital strictly positive linear map $\gamma\colon (A_{\Theta})_{\mathrm{sa}}\to \operatorname{Aff}(T(A_{\Psi}))\cong \R$
by $\gamma(a)(\tau_{\Psi}) := \tau_{\Theta}(a).$ Since $A_{\Theta}$ is simple and $\tau_{\Theta}$ is faithful, $\gamma$ is strictly positive: if $0\neq a\in (A_\Theta)_+$, then $\gamma(a)(\tau_\Psi)=\tau_\Theta(a)>0.$ Moreover, $\gamma$ is compatible with $\kappa_0$, because $\rho_\Psi\circ\kappa_0=\rho_\Theta.$ Thus Lin's existence theorem for unital monomorphisms from AH algebras into simple $C^\ast$-algebras of tracial rank zero applies, and yields a unital monomorphism $\phi\colon A_{\Theta}\to A_{\Psi}$
with $[\phi]=\kappa$ in $KL(A_{\Theta},A_{\Psi})$ and $\tau_{\Psi}\circ\phi=\tau_{\Theta}$ \cite[Theorem~5.2]{LinRange}. 
Since $K_*(A_\Theta)$ is free, the $KL$-class determines the ordinary
graded $K$-theory map, and hence $\phi_{\ast i}=\kappa_i$ for $i=0,1$. This proves \textup{(ii)}$\Rightarrow$\textup{(i)}.

For uniqueness, let $\phi,\psi\colon A_{\Theta}\to A_{\Psi}$ be unital and assume that they induce the same maps on $K_0$ and $K_1$. By Lemma~\ref{lem:kl}, they determine the same class in $KL(A_{\Theta},A_{\Psi})$. Since both maps are unital and the traces are unique, both satisfy $\tau_{\Psi}\circ\phi=\tau_{\Theta}=\tau_{\Psi}\circ\psi.$
Therefore Lin's uniqueness theorem gives $\phi\sim_{\mathrm{au}}\psi$ \cite[Theorem~3.1]{LinRange}. The converse implication is immediate because approximate unitary equivalence preserves the induced maps on $K$-theory.\end{proof}

We now classify all maps into matrix amplifications, which automatically includes the case of
nonunital maps into $A_{\Psi}$ itself by taking $k=1$. For $k\geq 1$ we identify $K_i(M_k(A_{\Psi}))\cong K_i(A_{\Psi})$
via the standard Morita equivalence, so that $[1_{M_k(A_\Psi)}]=k[1_{A_\Psi}]$
in $K_0(A_\Psi)$, and we write $\tau_{\Psi}^{(k)}:=\Tr_k\otimes\tau_\Psi$ for the canonical
unnormalized trace on $M_k(A_\Psi)$.

\begin{theorem}
\label{thm:nonunital-classification}
Let $\Theta\in M_m(\R)$ and $\Psi\in M_n(\R)$ be nondegenerate skew-symmetric matrices, and let $k\geq 1$.
Let $\kappa=(\kappa_0,\kappa_1)\colon K_*(A_\Theta)\to K_*(A_\Psi)$
be a graded group homomorphism. Then the following are equivalent:
\begin{enumerate}[label=\textup{(\roman*)}]
\item There exists a $*$-homomorphism $\phi\colon A_\Theta\to M_k(A_\Psi)$
such that $\phi_{*i}=\kappa_i$ for $i=0,1$.

\item Either $\kappa_0=0$ and $\kappa_1=0$, corresponding to the zero
homomorphism, or, setting $\eta:=\kappa_0([1_{A_\Theta}])$,
one has
$0<\eta\leq k[1_{A_\Psi}]$ in the ordered group $K_0(A_\Psi)$, and $\rho_\Psi\circ\kappa_0=\rho_\Psi(\eta)\,\rho_\Theta.$ 
\end{enumerate}

Moreover, if $\phi,\psi\colon A_\Theta\to M_k(A_\Psi)$ are nonzero $*$-homomorphisms, then $\phi\aue\psi$
if and only if $
\phi_{*0}=\psi_{*0}$ and
$\phi_{*1}=\psi_{*1}.$
Consequently, $\mathrm{Hom}(A_\Theta,M_k(A_\Psi))/\!\aue$
is the disjoint union of the zero class and the classes parametrized by graded
homomorphisms satisfying the nonzero alternative in \textup{(ii)}.
\end{theorem}

\begin{proof}
Assume first that \textup{(i)} holds, and let $\phi\colon A_{\Theta}\to M_k(A_{\Psi})$ be such a map.
If $\phi=0$, then $\kappa_0=\kappa_1=0$ and we are done. So suppose $\phi\neq 0$.
Because $A_{\Theta}$ is simple, $\phi$ is injective. Set $p:=\phi(1_{A_{\Theta}})\in M_k(A_{\Psi})$.
Then $p$ is a nonzero projection, and $\phi$ may be regarded as a unital monomorphism $\phi\colon A_{\Theta}\to pM_k(A_{\Psi})p.$
Moreover, $\eta:=\kappa_0([1_{A_\Theta}])=[p]$. Since $0\neq p\leq 1_{M_k(A_\Psi)}$, we have $0<\eta\leq [1_{M_k(A_\Psi)}]=k[1_{A_\Psi}]$
in $K_0(A_\Psi)$.

The algebra $M_k(A_\Psi)$ is simple, has tracial rank zero, stable rank one,
and has a unique tracial state. Since $p\neq 0$, the corner
$pM_k(A_\Psi)p$ is full and simple. Tracial rank zero passes to unital corners,
and the normalized restriction of the unique trace on $M_k(A_\Psi)$ is the
unique tracial state on $pM_k(A_\Psi)p$:
\[
\tau_p(x):=\frac{\tau_{\Psi}^{(k)}(x)}{\tau_{\Psi}^{(k)}(p)},\qquad x\in pM_k(A_{\Psi})p
\]
Since $\tau_p\circ\phi$ is a tracial state on $A_{\Theta}$, uniqueness of the trace on $A_{\Theta}$ gives $\tau_p\circ\phi = \tau_{\Theta}.$
Let $t:=\tau_\Psi^{(k)}(p)=\rho_\Psi([p])=\rho_\Psi(\eta)>0.$
Multiplying the identity $\tau_p\circ\phi=\tau_\Theta$ by $t$ gives  $\tau_{\Psi}^{(k)}\circ\phi = t\,\tau_{\Theta}.$

Under the standard identification $K_0(M_k(A_{\Psi}))\cong K_0(A_\Psi)$, the map on
$K_0$ induced by the unnormalized trace $\tau_\Psi^{(k)}$ is precisely
$\rho_\Psi$. Hence, passing to $K_0$ gives $\rho_\Psi\circ\kappa_0=t\,\rho_\Theta=\rho_\Psi(\eta)\,\rho_\Theta.$
Thus \textup{(ii)} holds.

Conversely, assume \textup{(ii)}. If $\kappa_0=0$ and $\kappa_1=0$, then the zero homomorphism $A_\Theta\to M_k(A_\Psi)$
realizes $\kappa$.

Now assume that $\eta:=\kappa_0([1_{A_\Theta}])$
satisfies $0<\eta\leq [1_{M_k(A_{\Psi})}] =k[1_{A_\Psi}]$
and that $\rho_\Psi\circ\kappa_0=\rho_\Psi(\eta)\rho_\Theta.$ Set $t:=\rho_\Psi(\eta).$
By the order description of $K_0(A_\Psi)$, the inequality $0<\eta$ implies $t>0$.

Since $\eta\in K_0(M_k(A_\Psi))^+$, choose a projection $q\in M_\ell(M_k(A_\Psi))$
with $[q]=\eta$. Since $[1_{M_k(A_\Psi)}]-\eta\in K_0(M_k(A_\Psi))^+$, choose a projection
$r\in M_s(M_k(A_\Psi))$ with $[r]=[1_{M_k(A_\Psi)}]-\eta$. Thus $[q]+[r]=[1_{M_k(A_\Psi)}]$
in $K_0(M_k(A_\Psi))$. Since $M_k(A_\Psi)$ has stable rank one, projections over $M_k(A_\Psi)$
satisfy cancellation.  Hence equality of $K_0$-classes of projections implies
Murray--von Neumann equivalence after identifying projections in matrix
amplifications.  Applying this to $[q]+[r]=[1_{M_k(A_\Psi)}],$
we obtain that $q\oplus r$ is Murray--von Neumann equivalent, in the stabilization,
to $1_{M_k(A_\Psi)}$. Therefore $q$ is Murray--von Neumann subequivalent to
$1_{M_k(A_\Psi)}$. Thus $q$ is equivalent to a projection $p\leq 1_{M_k(A_\Psi)}.$
Since \(p\leq 1_{M_k(A_\Psi)}\), we may regard \(p\) as a projection in
\(M_k(A_\Psi)\), and \([p]=\eta\). Moreover, since $\eta>0$, the projection $p$ is nonzero and, since $M_k(A_\Psi)$
is simple, every nonzero projection in $M_k(A_\Psi)$, in particular $p$, is full.

Let $\iota\colon pM_k(A_{\Psi})p\hookrightarrow M_k(A_{\Psi})$
be the corner inclusion. Since $p$ is full, $\iota$ induces isomorphisms $\iota_{*i}\colon K_i(pM_k(A_{\Psi})p )\xrightarrow{\cong}K_i(M_k(A_{\Psi}))\cong K_i(A_\Psi)$ for $i=0,1.$ Define
\[
\kappa_i^p:=(\iota_{*i})^{-1}\circ\kappa_i
\colon K_i(A_\Theta)\to K_i(pM_k(A_{\Psi})p).
\]
Then $\kappa_0^p([1_{A_\Theta}])=[1_{pM_k(A_{\Psi})p}]=[p].$

Notice that the unique tracial state on $pM_k(A_{\Psi})p$ is \[\tau_{pM_k(A_{\Psi})p}(x)
=
\frac{\tau_\Psi^{(k)}(x)}{t},
\qquad x\in pM_k(A_{\Psi})p,\]
and let
$\rho_{pM_k(A_{\Psi})p}:=(\tau_{pM_k(A_{\Psi})p})_*\colon K_0(pM_k(A_{\Psi})p)\to\R.$
For every $x\in K_0(A_\Theta)$, we have $\rho_{pM_k(A_{\Psi})p}(\kappa_0^p(x))=\frac{1}{t}\rho_\Psi(\kappa_0(x))
=\frac{1}{t}\,t\,\rho_\Theta(x)
=\rho_\Theta(x).$
Thus
\begin{equation}\label{eq:corner-unital-trace-compatibility}
\rho_{pM_k(A_{\Psi})p}\circ\kappa_0^p=\rho_\Theta.
\end{equation}

The ordered $K_0$-group of $pM_k(A_{\Psi})p$ is transported from that of $M_k(A_{\Psi})$ by the full-corner Morita equivalence. Since the positive cone of $K_0(M_k(A_{\Psi}))$ is determined by the trace, the same is true for $K_0(pM_k(A_{\Psi})p)$:
\[
K_0(pM_k(A_{\Psi})p)^+
=
\{y\in K_0(pM_k(A_{\Psi})p):\rho_{pM_k(A_{\Psi})p}(y)>0\}\cup\{0\}.
\]
Therefore \eqref{eq:corner-unital-trace-compatibility} implies that $\kappa_0^p$ is positive and sends nonzero positive classes to nonzero positive classes. 

By Lemma~\ref{lem:kl}, the graded homomorphism $\kappa^p=(\kappa_0^p,\kappa_1^p)$
determines a unique class in $KL(A_\Theta,pM_k(A_{\Psi})p).$
Define the trace map $\lambda\colon T(pM_k(A_{\Psi})p)\to T(A_\Theta)$
by $\lambda(\tau_{pM_k(A_{\Psi})p})=\tau_\Theta.$ Since both trace spaces are singletons, this is the unique affine map between
them. It is strictly positive because $\tau_\Theta$ is faithful. 
The compatibility between $\lambda$ and $\kappa^p$ is exactly
\eqref{eq:corner-unital-trace-compatibility}. Hence Lin's existence theorem \cite[Theorem~5.2]{LinRange} applies and gives a unital monomorphism
\[
\phi_p\colon A_\Theta\to pM_k(A_{\Psi})p
\]
with $(\phi_p)_{*i}=\kappa_i^p$ for $i=0,1$,
and $\tau_{pM_k(A_{\Psi})p}\circ\phi_p=\tau_\Theta.$ Finally, define
\[
\phi:=\iota\circ\phi_p\colon A_\Theta\to M_k(A_\Psi).
\]
Then $\phi_{*i}
=
\iota_{*i}\circ(\phi_p)_{*i}
=
\iota_{*i}\circ\kappa_i^p
=
\kappa_i$
for $i=0,1$. This proves \textup{(i)}.

It remains to prove the uniqueness statement. Let $\phi,\psi\colon A_\Theta\to M_k(A_\Psi)$
be nonzero $*$-homomorphisms with $\phi_{*0}=\psi_{*0}$ and $\phi_{*1}=\psi_{*1}.$ Set $p:=\phi(1_{A_\Theta}),$ 
$q:=\psi(1_{A_\Theta}).$
Then $[p]=\phi_{*0}([1_{A_\Theta}])=\psi_{*0}([1_{A_\Theta}])=[q]$
in $K_0(M_k(A_\Psi))$. Since $M_k(A_\Psi)$ has stable rank one, cancellation holds for projections.
Thus $[p]=[q]$ implies that $p$ and $q$ are Murray--von Neumann
equivalent. Applying the same argument to $1-p$ and $1-q$, we obtain a
unitary $w\in M_k(A_\Psi)$ such that $wpw^*=q.$
After replacing $\psi$ by $\operatorname{Ad}(w^*)\circ\psi$, we may assume
that $p=q$. Thus both $\phi$ and $\psi$ are unital monomorphisms $A_\Theta\to pM_k(A_\Psi)p.$

The inclusion $pM_k(A_\Psi)p\hookrightarrow M_k(A_\Psi)$ induces an isomorphism on $K$-theory, and $\phi$ and $\psi$ induce the same maps on $K_0(pM_k(A_\Psi)p)$ and $K_1(pM_k(A_\Psi)p)$. By Lemma~\ref{lem:kl}, they determine the same class in $KL(A_\Theta,pM_k(A_\Psi)p).$
The corner $pM_k(A_\Psi)p$ has a unique tracial state, so $\tau_{pM_k(A_\Psi)p}\circ\phi=\tau_\Theta=\tau_{pM_k(A_\Psi)p}\circ\psi.$
Lin's uniqueness theorem therefore gives approximate unitary equivalence of $\phi$ and $\psi$ inside the corner $pM_k(A_\Psi)p$. If $v_n$ are the implementing unitaries in $pM_k(A_\Psi)p$, then $v_n+(1-p)$
are unitaries in $M_k(A_\Psi)$, so the same approximate unitary equivalence holds inside $M_k(A_\Psi)$. Undoing the initial unitary conjugacy, we obtain $\phi\aue\psi$
as maps into $M_k(A_\Psi)$.

The converse is immediate, since approximate unitary equivalence preserves the induced maps on $K$-theory.
\end{proof}

\subsection{Toric \texorpdfstring{$\ast$}{*}--homomorphisms from integer lattice maps}\label{toric}

We now turn to concrete lattice-induced maps. They provide explicit representatives for special subclasses of the homomorphisms classified above. Besides the unital monomial maps of Case~(1), discussed in Proposition~\ref{prop:toric-explicit} and the Morita-corner models of Case~(2), treated in Proposition~\ref{thm:case2-explicit}, we also record two further geometric families that will be used later: embeddings with dimension change, described in Theorem~\ref{thm:case3-explicit}, and same-dimensional embeddings whose range is a proper unital subalgebra of a corner, described in Corollary~\ref{cor:case3-from-case4}. All four families fit within the complete approximate-unitary-equivalence classification established in Theorems~\ref{thm:unital-classification} and \ref{thm:nonunital-classification}.

\begin{proposition}[Case (1): monomial (toric) homomorphisms]\label{prop:toric-explicit}
Let $\Theta\in M_m(\R)$ and $\Psi\in M_n(\R)$ be skew-symmetric matrices.
Let $M\in M_{n\times m}(\Z)$ and let $z=(z_1,\dots,z_m)\in\T^m$.
Assume
\begin{equation}\label{eq:congruence-explicit}
\Theta \equiv M^t\Psi M \pmod{M_m(\Z)_{\mathrm{skew}}},
\end{equation}
meaning that $\Theta- M^t\Psi M$ is an integer-valued skew-symmetric matrix. Then there is a unital $*$-homomorphism $\varphi_{M,z}\colon A_{\Theta}\longrightarrow A_{\Psi}$
defined on generators by
\begin{equation}\label{eq:phiM-explicit}
\varphi_{M,z}(U_j)=z_j\,V^{Me_j}=z_j\,V_1^{M_{1j}}\cdots V_n^{M_{nj}},
\qquad 1\le j\le m,
\end{equation}
where $V_1,\dots,V_n$ are the canonical generators of $A_{\Psi}$ and $e_j$ is the $j$th standard basis vector of $\Z^m$. Conversely, any $*$-homomorphism $\varphi\colon A_{\Theta}\to A_{\Psi}$ satisfying \eqref{eq:phiM-explicit} for some $M\in M_{n\times m}(\Z)$ and $z\in\T^m$ forces \eqref{eq:congruence-explicit}.
\end{proposition}

\begin{proof}
Set $W_j:=z_jV^{Me_j}$.
It suffices to compute the commutation relations among the ordered monomials $V^{Me_j}$.
Applying \eqref{eq:ordered-commutator} in $A_{\Psi}$ with $x=Me_j$ and $y=Me_k$ gives
\[
V^{Me_k}V^{Me_j}
=
\exp\bigl(2\pi i\langle Me_j,\Psi Me_k\rangle\bigr)\,V^{Me_j}V^{Me_k}
=
\exp\bigl(2\pi i(M^t\Psi M)_{j,k}\bigr)\,V^{Me_j}V^{Me_k}.
\]
Condition \eqref{eq:congruence-explicit} is exactly the statement that $\exp(2\pi i\Theta_{j,k})=\exp\bigl(2\pi i(M^t\Psi M)_{j,k}\bigr)$
for all $j,k$. Therefore the unitaries $W_1,\dots,W_m$ satisfy the defining relations \eqref{eq:nctorus-rel} of $A_{\Theta}$, and the universal property yields a unique unital $*$-homomorphism $\varphi_{M,z}$ with $\varphi_{M,z}(U_j)=W_j$.

Conversely, if \eqref{eq:phiM-explicit} holds, then the images satisfy $\varphi(U_k)\varphi(U_j)=\exp\bigl(2\pi i(M^t\Psi M)_{j,k}\bigr)\,\varphi(U_j)\varphi(U_k),$
while the defining relations in $A_{\Theta}$ force $\varphi(U_k)\varphi(U_j)
=
\exp(2\pi i\Theta_{j,k})\,\varphi(U_j)\varphi(U_k).$
Hence
\[
\exp(2\pi i\Theta_{j,k})
=
\exp\bigl(2\pi i(M^t\Psi M)_{j,k}\bigr)
\qquad\text{for all }j,k,
\]
which is equivalent to \eqref{eq:congruence-explicit}.
\end{proof}

In preparation for the lemma below, we recall the following standard terminology; see
\cite{BlanchardTensorCX,WilliamsCrossedProducts}. A $C([0,1])$-algebra
is a $C^*$-algebra $\mathcal A$ equipped with a unital $*$-homomorphism $C([0,1])\longrightarrow Z(M(\mathcal A)),$
where $Z(M(\mathcal A))$ denotes the center of the multiplier algebra of
$\mathcal A$. Thus elements of $C([0,1])$ act on $\mathcal A$ by central
multipliers. We write $s\in C([0,1])$, $s(a)=a,$
for the identity function, viewed as a central self-adjoint multiplier. For
$a\in[0,1]$, let $C_a([0,1])=\{f\in C([0,1]):f(a)=0\}.$
The fiber of $\mathcal A$ at $a$ is $\mathcal A_a:=\mathcal A/\overline{C_a([0,1])\mathcal A},$
and the quotient map is denoted by $\ev_a\colon \mathcal A\to \mathcal A_a.$

\begin{lemma}
\label{lem:toric-PV-naturality}
Let $H\in M_r(\R)_{\mathrm{skew}}$ be a skew-symmetric matrix. Let $\mathcal A_H$ denote the universal
$C([0,1])$-algebra generated by unitaries $\mathcal U_1,\ldots,\mathcal U_r$
satisfying the relations
\begin{equation}\label{eq:com1}
\mathcal U_\ell \mathcal U_j
=
\exp(2\pi i\,sH_{j,\ell})\,\mathcal U_j\mathcal U_\ell,
\qquad 1\leq j,\ell\leq r.
\end{equation} where $s\in C([0,1])$ is the identity function $s(a)=a$. 
For $a\in[0,1]$, let $\ev_a^H\colon \mathcal A_H\longrightarrow A_{aH}$
be the evaluation map. Then $(\ev_a^H)_*$ is an isomorphism on $K$-theory for
every $a\in[0,1]$. Hence the maps
\[
\mu_H
:=
(\ev_1^H)_*\circ\bigl((\ev_0^H)_*\bigr)^{-1}
\colon
K_*(C(\T^r))\longrightarrow K_*(A_H)
\]
identifies $K_*(A_H)$, as a graded abelian group, with $\Lambda^{\mathrm{even}}\Z^r\oplus \Lambda^{\mathrm{odd}}\Z^r$
under the standard identification $K_*(C(\T^r))\cong \Lambda^*\Z^r$.

Moreover, let $G\in M_N(\R)_{\mathrm{skew}}$ and let
$L\in M_{N\times r}(\Z)$ satisfy $H=L^tGL.$
Let $\mathcal V_1,\ldots,\mathcal V_N$ be the universal generators of
$\mathcal A_G$. Then the assignment
\[
\mathcal U_j\longmapsto \mathcal V^{Le_j}
=
\mathcal V_1^{L_{1j}}\cdots \mathcal V_N^{L_{Nj}},
\qquad 1\leq j\leq r,
\]
defines a $C([0,1])$-linear $*$-homomorphism $\Phi_L\colon \mathcal A_H\longrightarrow \mathcal A_G.$

If $f_L\colon \T^N\to \T^r$ is defined as
\[
f_L(t_1,\ldots,t_N)
=
\left(\prod_{i=1}^N t_i^{L_{ij}}\right)_{j=1}^r,
\]
then the identity $
(\varphi_{L,1})_*\circ \mu_H=\mu_G\circ (f_L^*)_*$ holds on $K$-theory.
\end{lemma}

\begin{proof}We notice that in our notation, $\ev_a^H\colon \mathcal A_H\longrightarrow A_{aH}$, is simply the quotient map to the fiber $(\mathcal A_H)_a$ at $a$, obtained by identifying $(\mathcal A_H)_a$ and $A_{aH}$: after evaluating $s$ at $a$, the images of the generators satisfy exactly
the defining relations of $A_{aH}$. Hence the universal property of
$A_{aH}$ gives a surjective map $A_{aH}\to(\mathcal A_H)_a$, while the
quotient map $\mathcal A_H\to A_{aH}$ induced by evaluation gives the inverse
map. Thus $(\mathcal A_H)_a\cong A_{aH}$.

We first prove that the evaluation maps induce $K$-theory isomorphisms. 
This is the usual homotopy-invariance part of the iterated Pimsner--Voiculescu
computation, but we spell out the argument.

For $q=1,\ldots,r$, let $\mathcal A_H^{(q)}$ be the
$C([0,1])$-algebra generated by $\mathcal U_1,\ldots,\mathcal U_q$ with
the corresponding relations coming from the upper-left $q\times q$ corner of
$H$. Thus $\mathcal A_H^{(1)}\cong C([0,1]\times\T).$
For $q\geq2$, we have $\mathcal A_H^{(q)}
\cong
\mathcal A_H^{(q-1)}\rtimes_{\alpha_q}\Z,$
where the crossed-product convention is that the implementing unitary $u$
satisfies $uau^*=\alpha_q(a)$, and $\alpha_q(\mathcal U_j)
=
\exp(2\pi i\,sH_{j,q})\,\mathcal U_j,$ for all $1\leq j<q.$  Evaluation at $a\in[0,1]$ intertwines this crossed-product decomposition with
the corresponding crossed-product decomposition $A_{aH}^{(q)}
\cong
A_{aH}^{(q-1)}\rtimes_{\alpha_q^{(a)}}\Z $ of $A_{aH}^{(q)}$, where $\alpha_q^{(a)}(U_j)
=
\exp(2\pi i\,aH_{j,q})U_j.$

Set $B=\mathcal A_H^{(q-1)},$ $C=A_{aH}^{(q-1)},$
$\beta=\alpha_q,$ $\gamma=\alpha_q^{(a)},$
and let $\varepsilon=\ev_a^{H,q-1}\colon B\to C.$
Then $\varepsilon\circ\beta=\gamma\circ\varepsilon,$
so $\varepsilon$ induces a $*$-homomorphism $\widetilde{\varepsilon}\colon
B\rtimes_\beta\Z
\longrightarrow
C\rtimes_\gamma\Z.$
By naturality of the Pimsner--Voiculescu six-term exact sequence
\cite{PimsnerVoiculescu}, we obtain the following commutative diagram, written
in unrolled six-term form:
\[
\begin{tikzcd}[column sep=small, row sep=large]
K_0(B)
  \arrow[r,"1-\beta_*"]
  \arrow[d,"\varepsilon_*"']
&
K_0(B)
  \arrow[r,"(i_B)_*"]
  \arrow[d,"\varepsilon_*"']
&
K_0(B\rtimes_\beta\Z)
  \arrow[r,"\partial_0^B"]
  \arrow[d,"\widetilde{\varepsilon}_*"']
&
K_1(B)
  \arrow[r,"1-\beta_*"]
  \arrow[d,"\varepsilon_*"']
&
K_1(B)
  \arrow[r,"(i_B)_*"]
  \arrow[d,"\varepsilon_*"']
&
K_1(B\rtimes_\beta\Z)
  \arrow[r,"\partial_1^B"]
  \arrow[d,"\widetilde{\varepsilon}_*"']
&
K_0(B)
  \arrow[d,"\varepsilon_*"]
\\
K_0(C)
  \arrow[r,"1-\gamma_*"]
&
K_0(C)
  \arrow[r,"(i_C)_*"]
&
K_0(C\rtimes_\gamma\Z)
  \arrow[r,"\partial_0^C"]
&
K_1(C)
  \arrow[r,"1-\gamma_*"]
&
K_1(C)
  \arrow[r,"(i_C)_*"]
&
K_1(C\rtimes_\gamma\Z)
  \arrow[r,"\partial_1^C"]
&
K_0(C).
\end{tikzcd}
\]
For $q=1$, the evaluation map $C([0,1]\times\T)\longrightarrow C(\T)$
is a homotopy equivalence and hence induces an isomorphism on $K$-theory.
Assume the result known for $q-1$. Then $\varepsilon_*$ is an isomorphism
on $K_0$ and $K_1$. The diagram above and the five lemma imply that $\widetilde{\varepsilon}_*
\colon
K_*(B\rtimes_\beta\Z)
\longrightarrow
K_*(C\rtimes_\gamma\Z)$
is an isomorphism. In other words,
\[
(\ev_a^{H,q})_*
\colon
K_*(\mathcal A_H^{(q)})
\longrightarrow
K_*(A_{aH}^{(q)})
\]
is an isomorphism. Induction gives the claim for $q=r$.

At $a=0$, we have $A_{0H}\cong C^*(\Z^r)\cong C(\T^r)$.
Thus $\mu_H=(\ev_1^H)_*\circ\bigl((\ev_0^H)_*\bigr)^{-1}$
transports the standard identification $K_*(C(\T^r))\cong \Lambda^*\Z^r$
to $K_*(A_H)$. This is the usual deformation/Pimsner--Voiculescu
identification of the $K$-theory of a noncommutative torus.

Now assume $H=L^tGL$. We verify that $\Phi_L$ is well-defined. In
$\mathcal A_G$, one has $\mathcal V^{Le_\ell}\mathcal V^{Le_j}=
\exp\bigl(2\pi i\,s\langle Le_j,GLe_\ell\rangle\bigr)
\mathcal V^{Le_j}\mathcal V^{Le_\ell}.$
Since $\langle Le_j,GLe_\ell\rangle=(L^tGL)_{j,\ell}=H_{j,\ell},$
the unitaries $\mathcal V^{Le_1},\ldots,\mathcal V^{Le_r}$ satisfy exactly the
defining relations of $\mathcal A_H$. Thus the universal property gives the
$C([0,1])$-linear $*$-homomorphism $\Phi_L\colon \mathcal A_H\to \mathcal A_G.$

For each $a\in[0,1]$, evaluation gives a commutative diagram
\[
\begin{CD}
\mathcal A_H @>{\Phi_L}>> \mathcal A_G\\
@V{\ev_a^H}VV @VV{\ev_a^G}V\\
A_{aH} @>{\varphi_{L,1}^{(a)}}>> A_{aG},
\end{CD}
\]
where $\varphi_{L,1}^{(a)}$ is the monomial homomorphism determined by $L$ at
parameter $a$. For $a=0$, this is exactly the commutative pullback $f_L^*\colon C(\T^r)\to C(\T^N).$
For $a=1$, it is the noncommutative toric homomorphism
$\varphi_{L,1}\colon A_H\to A_G.$
Therefore $(\ev_1^G)_*\circ(\Phi_L)_*=(\varphi_{L,1})_*\circ(\ev_1^H)_*$
and $(\ev_0^G)_*\circ(\Phi_L)_*=(f_L^*)_*\circ(\ev_0^H)_*.$
Since the evaluation maps induce $K$-theory isomorphisms, we obtain
\[
\begin{aligned}
(\varphi_{L,1})_*\circ\mu_H
&=
(\varphi_{L,1})_*
\circ
(\ev_1^H)_*
\circ
\bigl((\ev_0^H)_*\bigr)^{-1}
\\
&=
(\ev_1^G)_*
\circ
(\Phi_L)_*
\circ
\bigl((\ev_0^H)_*\bigr)^{-1}
\\
&=
(\ev_1^G)_*
\circ
\bigl((\ev_0^G)_*\bigr)^{-1}
\circ
(f_L^*)_*
\\
&=
\mu_G\circ(f_L^*)_*.
\end{aligned}
\]
This proves the naturality identity. \end{proof}

\begin{conventions}\label{conv:toric-K-identification}
Let $\varphi_{M,z}\colon A_\Theta\to A_\Psi$
be as in Proposition~\ref{prop:toric-explicit}, and assume $\Theta\equiv M^t\Psi M
\pmod{M_m(\Z)_{\mathrm{skew}}}.$
Put $H:=M^t\Psi M.$ Since $\Theta-H$ has integer entries, the defining commutation constants for
$A_\Theta$ and $A_H$ are the same. Hence there is a canonical
$*$-isomorphism $\chi_{\Theta,H}\colon A_\Theta\longrightarrow A_H$ satisfying 
$\chi_{\Theta,H}(U_j^\Theta)=U_j^H,$ for every $1\leq j\leq m.$

For the source algebra $A_\Theta$, we use the exterior-algebra identification
obtained by first applying $\chi_{\Theta,H}$ and then using the
deformation--Pimsner--Voiculescu identification $\mu_H$ from
Lemma~\ref{lem:toric-PV-naturality}. Equivalently, we identify
$\Lambda^*\Z^m$ with $K_*(A_\Theta)$ by the map
\[
\nu_{\Theta,M,\Psi}
:=
(\chi_{\Theta,H})_*^{-1}\circ\mu_H
\colon
\Lambda^*\Z^m\cong K_*(C(\T^m))
\longrightarrow K_*(A_\Theta).
\]
For the target algebra $A_\Psi$, we use the usual identification
\[
\mu_\Psi\colon \Lambda^*\Z^n\cong K_*(C(\T^n))
\longrightarrow K_*(A_\Psi).
\]
Thus, when we say that $(\varphi_{M,z})_*$ is $\Lambda^*(M)$, we mean that $\mu_\Psi^{-1}\circ(\varphi_{M,z})_*\circ\nu_{\Theta,M,\Psi}=\Lambda^*(M).$
\end{conventions}

\begin{proposition}[$K$-theory of monomial maps]
\label{prop:Ktoric-explicit}
Let $\varphi_{M,z}\colon A_\Theta\to A_\Psi$
be as in Proposition~\ref{prop:toric-explicit}. Under the exterior-algebra
identifications described in
Convention~\ref{conv:toric-K-identification}, the induced maps are $(\varphi_{M,z})_{*0}=\Lambda^{\mathrm{even}}(M)\colon\Lambda^{\mathrm{even}}\Z^m\to \Lambda^{\mathrm{even}}\Z^n,$
and $(\varphi_{M,z})_{*1}=\Lambda^{\mathrm{odd}}(M)\colon\Lambda^{\mathrm{odd}}\Z^m\to \Lambda^{\mathrm{odd}}\Z^n.$
Equivalently, $\mu_\Psi^{-1}\circ(\varphi_{M,z})_*\circ\nu_{\Theta,M,\Psi}=\Lambda^*(M)$
as maps $\Lambda^*\Z^m\longrightarrow \Lambda^*\Z^n.$
\end{proposition}

\begin{proof}
We first remove the phase vector $z$. Let $\beta_z\in\operatorname{Aut}(A_\Theta)$
be the gauge automorphism defined by $\beta_z(U_j)=z_jU_j,$ for every $1\leq j\leq m.$
Since $\T^m$ is path connected, $\beta_z$ is homotopic to the identity
automorphism. Hence $(\beta_z)_*=\id$
on $K$-theory. Moreover, $\varphi_{M,z}=\varphi_{M,1}\circ\beta_z.$
Thus it suffices to compute the induced map for $\varphi_{M,1}$.

Put $H:=M^t\Psi M.$ By Convention~\ref{conv:toric-K-identification}, we identify $A_\Theta$ with
$A_H$ by the canonical isomorphism $\chi_{\Theta,H}\colon A_\Theta\to A_H,$
$\chi_{\Theta,H}(U_j^\Theta)=U_j^H.$ Under this identification, $\varphi_{M,1}$ becomes the toric homomorphism $\widetilde\varphi_{M,1}\colon A_H\to A_\Psi,$ $\widetilde\varphi_{M,1}(U_j^H)=V^{Me_j},$ associated to the exact identity $H=M^t\Psi M.$ That is, $\varphi_{M,1}=\widetilde\varphi_{M,1}\circ \chi_{\Theta,H}.$

We now apply Lemma~\ref{lem:toric-PV-naturality} with $L=M$, $G=\Psi,$ and $H=M^t\Psi M$. The lemma gives the commutative $K$-theory diagram
\[
\begin{CD}
K_*(C(\T^m)) @>{(f_M^*)_*}>> K_*(C(\T^n))\\
@V{\mu_H}VV @VV{\mu_\Psi}V\\
K_*(A_H) @>{(\widetilde\varphi_{M,1})_*}>> K_*(A_\Psi),
\end{CD}
\]
where $f_M\colon \T^n\to\T^m$
is given by
\[
f_M(t_1,\dots,t_n)
=
\left(\prod_{i=1}^n t_i^{M_{ij}}\right)_{j=1}^m.
\]
Equivalently, $(\widetilde\varphi_{M,1})_*\circ\mu_H=\mu_\Psi\circ(f_M^*)_*.$
Since $\nu_{\Theta,M,\Psi}=(\chi_{\Theta,H})_*^{-1}\circ\mu_H$
and $\varphi_{M,1}=\widetilde\varphi_{M,1}\circ\chi_{\Theta,H},$
we obtain
\[
\begin{aligned}
\mu_\Psi^{-1}
\circ(\varphi_{M,1})_*
\circ\nu_{\Theta,M,\Psi}
&=
\mu_\Psi^{-1}
\circ(\widetilde\varphi_{M,1})_*
\circ(\chi_{\Theta,H})_*
\circ(\chi_{\Theta,H})_*^{-1}
\circ\mu_H
\\
&=
\mu_\Psi^{-1}
\circ(\widetilde\varphi_{M,1})_*
\circ\mu_H
\\
&=
(f_M^*)_*.
\end{aligned}
\]
Thus, in the chosen exterior-algebra coordinates, the induced map is exactly
$(f_M^*)_*$.

It remains to compute $(f_M^*)_*$ on ordinary torus $K$-theory. Let
$u_1,\ldots,u_m$ be the coordinate functions on $\T^m$, and let
$v_1,\ldots,v_n$ be the coordinate functions on $\T^n$. Then $f_M^*(u_j)
=
v_1^{M_{1j}}\cdots v_n^{M_{nj}}.$
Let also $e_j=[u_j]\in K^1(\T^m)$ and $f_i=[v_i]\in K^1(\T^n).$
Under the standard identification $K^*(\T^m)\cong \Lambda^*\Z^m,$
the classes $e_1,\ldots,e_m$ are the degree-one exterior generators. Similarly,
$f_1,\ldots,f_n$ are the degree-one generators of $K^*(\T^n)\cong \Lambda^*\Z^n.$ Then
\[(f_M^*)_*(e_j)=
[f_M^*(u_j)]=
[v_1^{M_{1j}}\cdots v_n^{M_{nj}}]
=
\sum_{i=1}^n M_{ij}[v_i]=
\sum_{i=1}^n M_{ij}f_i.
\]
Thus the degree-one part of $(f_M^*)_*$ is precisely the homomorphism $M\colon \Z^m\to\Z^n.$
Since \(f_M^*\) is a graded ring homomorphism and $K^*(\T^m)\cong \Lambda^*\Z^m$
is the exterior algebra on the degree-one generators, the induced map on all
of $K^*(\T^m)$ is $\Lambda^*(M)\colon \Lambda^*\Z^m\to\Lambda^*\Z^n.$
Consequently,
$(\varphi_{M,1})_{*0}=\Lambda^{\mathrm{even}}(M)$ and
$(\varphi_{M,1})_{*1}=\Lambda^{\mathrm{odd}}(M)$
in the chosen exterior-algebra coordinates. Since $(\beta_z)_*=\id$, the same formulas hold for $\varphi_{M,z}$.
\end{proof}

\begin{remark}
\label{rem:MR21-prop37}
Mathai and Rosenberg prove an existence criterion for unital
$*$-homomorphisms between irrational rotation algebras
\cite[Theorem~2.1]{MathaiRosenberg}. Proposition~\ref{prop:toric-explicit}
recovers the explicit monomial construction appearing in the sufficiency
direction of their theorem.

Let $t\in(0,1)$, and write $J(t):=
\begin{pmatrix}
0&t\\
-t&0
\end{pmatrix}.$
For every $M\in M_2(\Z)$, one has $M^tJ(\theta)M=\det(M)\,J(\theta).$
Hence, if $\Theta=c\theta+d$, with
$c,d\in\Z$, $c\neq0$,
and if $M\in M_2(\Z)$ satisfies $\det(M)=c$, then $J(\Theta)\equiv M^tJ(\theta)M
\pmod{M_2(\Z)_{\mathrm{skew}}}.$
Therefore Proposition~\ref{prop:toric-explicit} gives a unital toric
$*$-homomorphism $A_{J(\Theta)}\longrightarrow A_{J(\theta)}.$
In particular, taking $M=
\begin{pmatrix}
c&0\\
0&1
\end{pmatrix}$
gives the monomial map
\[
U_1\longmapsto V_1^c,
\qquad
U_2\longmapsto V_2.
\]
Equivalently, after translating between the generator convention used here and
the convention $UV=e^{2\pi i\theta}VU$ used in
\cite{MathaiRosenberg}, this is the Mathai--Rosenberg map
\[
U\longmapsto u^c,
\qquad
V\longmapsto v.
\]

Conversely, the necessity of the condition $\Theta\in \Z+\theta\Z$ for an
arbitrary unital $* $-homomorphism follows from
Theorem~\ref{thm:unital-classification}; equivalently, it follows from the
trace-range description
\[
\rho_\Theta(K_0(A_\Theta))=\Z+\Theta\Z,
\qquad
\rho_\theta(K_0(A_\theta))=\Z+\theta\Z,
\qquad
\rho_\Theta([1])=\rho_\theta([1])=1.
\]
Indeed, if $A_\Theta\to A_\theta$ is unital, trace compatibility gives $\rho_\Theta=\rho_\theta\circ\phi_{*0}.$
Hence $\rho_\Theta(K_0(A_\Theta))
\subseteq
\rho_\theta(K_0(A_\theta)),$
and in particular $\Theta\in \Z+\theta\Z$. Thus
Proposition~\ref{prop:toric-explicit}, together with
Theorem~\ref{thm:unital-classification}, recovers
\cite[Theorem~2.1]{MathaiRosenberg}.

Finally, if $|c|=1$, then $M\in \GL(2,\Z)$, and the associated toric map is
an isomorphism. If $|c|\neq1$, then Proposition~\ref{prop:Ktoric-explicit}
shows that the induced map on $K_1\cong\Z^2$
is $M$, which is not an isomorphism. Since a unital homomorphism from the
simple algebra $A_\Theta$ is injective, surjectivity would force it to be an
isomorphism. Therefore the toric map cannot be onto $A_\theta$.
\end{remark}
\begin{proposition}
\label{prop:dimension-rigidity-explicit}
Let $A_\Theta$ and $A_\Psi$ be simple noncommutative tori of dimensions
$m$ and $n$, respectively. If there is a nonzero projection
$p\in A_\Psi$ such that $pA_\Psi p\cong A_\Theta,$
then $m=n$.
\end{proposition}

\begin{proof}
Since \(A_\Psi\) is simple and \(p\neq0\), the projection \(p\) is full. Hence
\(pA_\Psi p\) is a full corner of \(A_\Psi\), and so \(pA_\Psi p\) is strongly
Morita equivalent to \(A_\Psi\). In particular, $K_i(pA_\Psi p)\cong K_i(A_\Psi)$,
for $i=0,1$. 

If $pA_\Psi p\cong A_\Theta$, then $K_i(A_\Theta)\cong K_i(A_\Psi)$, for $i=0,1$, as abelian groups. By \eqref{eq:K-exterior},
$\rank K_0(A_\Theta)=\rank K_1(A_\Theta)=2^{m-1},$
and similarly $\rank K_0(A_\Psi)=\rank K_1(A_\Psi)=2^{n-1}.$
Thus $m=n$.
\end{proof}
In the proposition below, we use the notation and results from \cite{RieffelSchwarz, LiStrongMorita}.
\begin{proposition}[Case (2): corner models via strong Morita equivalence]
\label{thm:case2-explicit}
Let $d\geq 2$, and let $\Psi\in M_d(\R)$ be skew-symmetric. Let
\[
g=
\begin{pmatrix}
A & B\\
C & D
\end{pmatrix}
\in \SO(d,d\mid\Z),
\qquad
A,B,C,D\in M_d(\Z),
\]
and assume that $C\Psi+D$ is invertible. Put $\Psi':=g\cdot\Psi:=(A\Psi+B)(C\Psi+D)^{-1}.$
Then $A_\Psi$ and $A_{\Psi'}$ are strongly Morita equivalent. Consequently,
there exist $k\geq 1$ and a full projection $p\in M_k(A_\Psi)$
such that $A_{\Psi'}\cong pM_k(A_\Psi)p.$ Thus, once such a corner identification is fixed, every monomial map
$\varphi_{M,z}\colon A_\Theta\longrightarrow A_{\Psi'}$
produces a $*$-homomorphism
\[
A_\Theta\longrightarrow M_k(A_\Psi)
\]
by composition with the corner identification and the inclusion $pM_k(A_\Psi)p\hookrightarrow M_k(A_\Psi).$
If $p\neq 1_{M_k(A_\Psi)}$, then this composite is nonunital.
\end{proposition}
\begin{proof}
The strong Morita equivalence statement is \cite[Theorem~1.1]{LiStrongMorita}.
Thus $A_{\Psi'}$ and $A_\Psi$ are strongly Morita equivalent.

Let $E$ be an $A_{\Psi'}$-$A_\Psi$ imprimitivity bimodule. Since
$A_{\Psi'}$ is unital, the algebra of compact adjointable operators
$\mathcal K_{A_\Psi}(E)$ is unital. Hence $E$ is finitely generated
projective as a right Hilbert $A_\Psi$-module. Therefore there are
$k\geq 1$ and a projection $p\in M_k(A_\Psi)$
such that $E\cong pA_\Psi^k$
as right Hilbert $A_\Psi$-modules. It follows that
\[
A_{\Psi'}
\cong
\mathcal K_{A_\Psi}(E)
\cong
\mathcal K_{A_\Psi}(pA_\Psi^k)
\cong
pM_k(A_\Psi)p.
\]
Because $E$ is an imprimitivity bimodule, it is full as a right
$A_\Psi$-module; equivalently, the projection $p$ is full in
$M_k(A_\Psi)$.

Now let $\varphi_{M,z}\colon A_\Theta\to A_{\Psi'}$
be a monomial map. Choosing an isomorphism $\chi\colon A_{\Psi'}\xrightarrow{\cong}pM_k(A_\Psi)p$
as above, and composing with the corner inclusion, gives a
$*$-homomorphism
\[
A_\Theta
\xrightarrow{\ \varphi_{M,z}\ }
A_{\Psi'}
\xrightarrow{\ \chi\ }
pM_k(A_\Psi)p
\hookrightarrow
M_k(A_\Psi).
\]
Since $\varphi_{M,z}$ and $\chi$ are unital onto their codomains, this
composite sends $1_{A_\Theta}\mapsto
p.$
Therefore, if $p\neq 1_{M_k(A_\Psi)}$, the composite is nonunital as a map
into $M_k(A_\Psi)$.
\end{proof}

\begin{remark}
For $d=2$, the $\SO(2,2\mid\Z)$-orbit picture on skew $2\times2$
matrices reduces, at the level of the scalar rotation parameter, to the usual
fractional-linear action
\[
\theta\longmapsto \frac{a\theta+b}{c\theta+d}
\]
of $\GL(2,\Z)$. This is the classical Morita-equivalence picture for
irrational rotation algebras and is closely related to the nonunital part of
the Mathai--Rosenberg classification \cite[Theorem~2.7]{MathaiRosenberg}.

For $d\geq3$, being in the same $\SO(d,d\mid\Z)$-orbit is a sufficient
condition for strong Morita equivalence of $C^*$-algebraic noncommutative
tori, but it is not the most general $C^*$-algebraic criterion. The latter is
given in terms of ordered $K_0$-groups and centers; see
\cite[Theorem~1.1]{EllLiII}.
\end{remark}

\begin{theorem}[Case (3): embeddings with dimension change]
\label{thm:case3-explicit}
Let $m,n\geq 2$, and let
$\Theta\in M_m(\R)$ and $\Psi\in M_n(\R)$ be skew-symmetric matrices.
Assume that $\Theta$ is nondegenerate. Then the following are equivalent.

\begin{enumerate}[label=\textup{(\roman*)}]
\item There exist $k\geq 1$, a full projection
$p\in M_k(A_\Psi)$, a skew-symmetric matrix $\Psi'\in M_n(\R)$, an isomorphism $\alpha\colon A_{\Psi'}\xrightarrow{\cong}pM_k(A_\Psi)p,$
an integer matrix $M\in M_{n\times m}(\Z)$, and a phase vector
$z\in\T^m$, such that the composite
\[
\phi\colon
A_\Theta
\xrightarrow{\ \varphi_{M,z}\ }
A_{\Psi'}
\xrightarrow{\ \alpha\ }
pM_k(A_\Psi)p
\xrightarrow{\ \iota_p\ }
M_k(A_\Psi)
\]
is an injective $*$-homomorphism.

\item There exist a skew-symmetric matrix $\Psi'\in M_n(\R)$, an integer
matrix $M\in M_{n\times m}(\Z)$ with $\operatorname{rank}(M)=m$, and an
integer skew-symmetric matrix $K\in M_m(\Z)_{\mathrm{skew}}$, such that
$A_{\Psi'}$ is strongly Morita equivalent to $A_\Psi$ and $\Theta=M^t\Psi'M+K.$
\end{enumerate}

In this situation one necessarily has $m\leq n$. Moreover, if $\beta_i:=(\iota_p)_{*i}\circ\alpha_{*i}\colon K_i(A_{\Psi'})\longrightarrow K_i(M_k(A_\Psi)),$ for $i=0,1,$
then, under the exterior-algebra identifications used in
Proposition~\ref{prop:Ktoric-explicit}, $\phi_{*0}=\beta_0\circ\Lambda^{\mathrm{even}}(M)$ and $\phi_{*1}
=\beta_1\circ\Lambda^{\mathrm{odd}}(M).$ If, in addition, $\Psi$ is nondegenerate, then, under the standard Morita
identification $K_0(M_k(A_\Psi))\cong K_0(A_\Psi)$, one has $\rho_\Psi\circ\phi_{*0}
=
\rho_\Psi\bigl(\phi_{*0}([1_{A_\Theta}])\bigr)\,\rho_\Theta.$
Equivalently, $\tau_\Psi^{(k)}\circ\phi=t\,\tau_\Theta,$
where $\tau_\Psi^{(k)}:=\operatorname{Tr}_k\otimes\tau_\Psi$
is the unnormalized matrix trace and $t=\tau_\Psi^{(k)}(p)=\rho_\Psi\bigl(\phi_{*0}([1_{A_\Theta}])\bigr)>0.$
Whenever $p\neq 1_{M_k(A_\Psi)}$, the resulting map is genuinely nonunital.
\end{theorem}

\begin{proof}
Assume \textup{(ii)}. Since $\Theta=M^t\Psi'M+K$
with $K\in M_m(\Z)_{\mathrm{skew}}$, we have $\Theta\equiv M^t\Psi'M
\pmod{M_m(\Z)_{\mathrm{skew}}}.$
Hence Proposition~\ref{prop:toric-explicit} yields, for every
$z\in\T^m$, a unital monomial $*$-homomorphism $\varphi_{M,z}\colon A_\Theta\to A_{\Psi'}.$
Because $\Theta$ is nondegenerate, $A_\Theta$ is simple, and the unital map $\varphi_{M,z}$ is nonzero; hence $\varphi_{M,z}$ is injective.

Next, we proceed as in the proof of Proposition~\ref{thm:case2-explicit}. Since $A_{\Psi'}$ and $A_\Psi$ are strongly Morita equivalent, there is an
$A_{\Psi'}$-$A_\Psi$ imprimitivity bimodule $E$. Because $A_{\Psi'}$ is
unital, $E$ is finitely generated projective as a right Hilbert
$A_\Psi$-module. Hence there exist $k\geq1$ and a projection $p\in M_k(A_\Psi)$
such that $E\cong pA_\Psi^k.$
It follows that $A_{\Psi'}
\cong
\mathcal K_{A_\Psi}(E)
\cong
pM_k(A_\Psi)p.$
Moreover, $p$ is full because $E$ is an imprimitivity bimodule. 

Fix such an
isomorphism $\alpha\colon A_{\Psi'}\xrightarrow{\cong}pM_k(A_\Psi)p,$
and let $\iota_p\colon pM_k(A_\Psi)p\hookrightarrow M_k(A_\Psi)$
be the corner inclusion. Then $\phi:=\iota_p\circ\alpha\circ\varphi_{M,z}$
has the required form. Since $\varphi_{M,z}$, $\alpha$, and $\iota_p$
are injective, $\phi$ is injective. This proves \textup{(i)}.

We next compute the induced maps on \(K\)-theory. By
Proposition~\ref{prop:Ktoric-explicit}, $(\varphi_{M,z})_{*0}=\Lambda^{\mathrm{even}}(M)$ and
$(\varphi_{M,z})_{*1}=
\Lambda^{\mathrm{odd}}(M).$
By functoriality, $\phi_{*i}=(\iota_p)_{*i}\circ\alpha_{*i}\circ(\varphi_{M,z})_{*i},$
for $i=0,1$.
Thus, with $\beta_i:=(\iota_p)_{*i}\circ\alpha_{*i},$
we obtain $\phi_{*0}=\beta_0\circ\Lambda^{\mathrm{even}}(M)$ and
$\phi_{*1}=\beta_1\circ\Lambda^{\mathrm{odd}}(M).$

Assume now, in addition, that $\Psi$ is nondegenerate. Then $A_\Psi$ is
simple and has unique tracial state $\tau_\Psi$. We write $\tau_\Psi^{(k)}:=\operatorname{Tr}_k\otimes\tau_\Psi$
for the unnormalized trace on $M_k(A_\Psi)$. Thus the unique tracial state on
 $M_k(A_\Psi)$ is \(k^{-1}\tau_\Psi^{(k)}$. The corner $pM_k(A_\Psi)p$ has unique tracial state
\[
\tau_p(x)
=
\frac{\tau_\Psi^{(k)}(x)}{\tau_\Psi^{(k)}(p)},
\qquad
x\in pM_k(A_\Psi)p.
\]
The map $\alpha\circ\varphi_{M,z}\colon A_\Theta\to pM_k(A_\Psi)p$
is unital, and both domain and codomain have unique tracial states. Hence $\tau_p\circ\alpha\circ\varphi_{M,z}=\tau_\Theta.$
Multiplying by $t:=\tau_\Psi^{(k)}(p)>0$
gives $\tau_\Psi^{(k)}\circ\phi=t\,\tau_\Theta.$
Passing to $K_0$, and using the standard identification
$K_0(M_k(A_\Psi))\cong K_0(A_\Psi)$, gives $\rho_\Psi\circ\phi_{*0}=t\,\rho_\Theta.$
Since $\phi(1_{A_\Theta})=p,$
we also have $t=
\tau_\Psi^{(k)}(p)
=\rho_\Psi\bigl(\phi_{*0}([1_{A_\Theta}])\bigr).$
Therefore $\rho_\Psi\circ\phi_{*0}=\rho_\Psi\bigl(\phi_{*0}([1_{A_\Theta}])\bigr)\,\rho_\Theta.$

Finally, since $\varphi_{M,z}$ is unital and $\alpha$ is unital onto the
corner $pM_k(A_\Psi)p$, we have $\phi(1_{A_\Theta})=p.$
Thus, if $p\neq 1_{M_k(A_\Psi)}$, the map $\phi$ is nonunital as a map
into $M_k(A_\Psi)$.

Conversely, assume \textup{(i)}. Since $\alpha$ is an isomorphism onto a full
corner of $M_k(A_\Psi)$, the algebra $A_{\Psi'}$ is strongly Morita
equivalent to $M_k(A_\Psi)$, hence to $A_\Psi$. Moreover,
Proposition~\ref{prop:toric-explicit}, applied to the toric part $\varphi_{M,z}\colon A_\Theta\to A_{\Psi'},$
shows that $\Theta\equiv M^t\Psi'M
\pmod{M_m(\Z)_{\mathrm{skew}}}.$
Hence there exists $K\in M_m(\Z)_{\mathrm{skew}}$
such that $\Theta=M^t\Psi'M+K.$

It remains to prove that $\operatorname{rank}(M)=m$. Suppose, to the
contrary, that $\operatorname{rank}(M)<m$. Since $M$ has integer entries,
there exists a nonzero vector $x\in\Z^m$
such that $Mx=0.$
Then, for every $y\in\Z^m$, $x^t\Theta y=x^t(M^t\Psi'M)y+x^tKy=(Mx)^t\Psi'My+x^tKy=x^tKy.$
Since $K$ has integer entries, $x^tKy\in\Z$
for all $y\in\Z^m.$
Therefore $\exp(2\pi i\,x^t\Theta y)=1$
for all $y\in\Z^m.$
This contradicts the nondegeneracy of $\Theta$. Hence $\ker(M)\cap\Z^m=\{0\},$
and, since $M$ is an integer matrix, this forces $\operatorname{rank}(M)=m.$
In particular $m\leq n$. This proves \textup{(ii)} and completes the proof.
\end{proof}

In the special case  $m=n=d$, Theorem~\ref{thm:case3-explicit} reduces to the following result, which provides a more explicit formulation in this setting:
\begin{corollary}[Case (4) as the same-dimensional noncorner subcase of Case~\textup{(3)}]
\label{cor:case3-from-case4}
Let $d\geq 2$, and let $\Theta,\Psi\in M_d(\R)$ be skew-symmetric matrices.
Assume that $\Theta$ is nondegenerate. Then condition \textup{(i)} of
Theorem~\ref{thm:case3-explicit} holds if and only if there exist a
skew-symmetric matrix $\Psi'\in M_d(\R)$, an integer matrix
$M\in M_d(\Z)$ with $\det(M)\neq 0,$
and an integer skew-symmetric matrix
$K\in M_d(\Z)_{\mathrm{skew}}$
such that $A_{\Psi'}$ is strongly Morita equivalent to $A_\Psi$ and $
\Theta=M^t\Psi' M+K.$

For any factorization
\[
\phi=
\iota_p\circ \alpha\circ \varphi_{M,z}
\colon
A_\Theta
\xrightarrow{\ \varphi_{M,z}\ }
A_{\Psi'}
\xrightarrow{\ \alpha\ }
pM_k(A_\Psi)p
\xrightarrow{\ \iota_p\ }
M_k(A_\Psi)
\]
arising from such data, the range of $\alpha\circ \varphi_{M,z}
\colon A_\Theta\to pM_k(A_\Psi)p$
is a unital $C^*$-subalgebra of the corner $pM_k(A_\Psi)p$ with unit $p$.
Moreover, $\alpha\bigl(\varphi_{M,z}(A_\Theta)\bigr)
\subsetneq pM_k(A_\Psi)p$
if and only if $|\det(M)|\neq 1.$
Equivalently, for this factorization, the range is the whole corner exactly
when $M\in \GL_d(\Z)$.
\end{corollary}

\begin{proof}
The first assertion is exactly the special case $m=n=d$ of
Theorem~\ref{thm:case3-explicit}. 
Now fix such a factorization $\phi=\iota_p\circ \alpha\circ \varphi_{M,z}.$
Since $\varphi_{M,z}$ is unital and $\alpha$ is a unital isomorphism from
$A_{\Psi'}$ onto the corner $pM_k(A_\Psi)p$, the algebra $\alpha\bigl(\varphi_{M,z}(A_\Theta)\bigr)$
is a unital $C^*$-subalgebra of $pM_k(A_\Psi)p$ with unit $p$.

We first prove that the range is proper when $|\det(M)|\neq 1$. Suppose, to the
contrary, that $\alpha\bigl(\varphi_{M,z}(A_\Theta)\bigr)=pM_k(A_\Psi)p.$
Applying $\alpha^{-1}$ gives $\varphi_{M,z}(A_\Theta)=A_{\Psi'}.$
Thus $\varphi_{M,z}$ is surjective. Since $A_\Theta$ is simple and
$\varphi_{M,z}$ is unital, $\varphi_{M,z}$ is also injective. Hence
$\varphi_{M,z}$ would be an isomorphism. By Proposition~\ref{prop:Ktoric-explicit}, under the exterior-algebra
identifications, the induced map on $K_1$ is
$(\varphi_{M,z})_{*1}
=\Lambda^{\mathrm{odd}}(M)\colon\Lambda^{\mathrm{odd}}\Z^d\to \Lambda^{\mathrm{odd}}\Z^d.$
This map preserves exterior degree, and its exterior degree-one component is
precisely $M\colon \Z^d\to \Z^d.$
If $|\det(M)|\neq 1$, then $M$ is not surjective as a homomorphism
$\Z^d\to\Z^d$. Therefore $\Lambda^{\mathrm{odd}}(M)$ is not surjective.
Consequently $(\varphi_{M,z})_{*1}$ cannot be an isomorphism, contradicting
the fact that $\varphi_{M,z}$ is an isomorphism. Hence $\alpha\bigl(\varphi_{M,z}(A_\Theta)\bigr)\subsetneq pM_k(A_\Psi)p.$

Conversely, assume that $|\det(M)|=1$. Then $M\in\GL_d(\Z)$. We show that
$\varphi_{M,z}$ is surjective. Let $V_1,\ldots,V_d$ be the canonical generators
of $A_{\Psi'}$. Since $M\in\GL_d(\Z)$, the columns $Me_1,\ldots,Me_d$
form a $\Z$-basis of $\Z^d$. The range of $\varphi_{M,z}$ contains the unitaries $z_jV^{Me_j},$ for $1\leq j\leq d.$
Multiplying by scalars, it therefore contains each $V^{Me_j}$. Since the
vectors $Me_j$ form a $\Z$-basis of $\Z^d$, every standard basis vector $e_i$ of $\Z^d$ is an
integer linear combination of them. Hence, using products and adjoints of the
unitaries $V^{Me_j}$, the range contains a nonzero scalar multiple of each
canonical generator $V_i$. Therefore the range contains all $V_i$, and so $\varphi_{M,z}(A_\Theta)=A_{\Psi'}.$
Thus $\alpha\bigl(\varphi_{M,z}(A_\Theta)\bigr)=pM_k(A_\Psi)p.$

Combining the two directions, the range inside the chosen corner is proper
exactly when $|\det(M)|\neq 1$, and it is the whole corner exactly when
$M\in\GL_d(\Z)$.
\end{proof}

\section{Elliott invariants of noncommutative protori}\label{sec:protorus-invariants}

Consider now inductive systems of $C^*$-algebras $B_1 \xrightarrow{\ \phi_1\ } B_2 \xrightarrow{\ \phi_2\ } B_3 \xrightarrow{\ \phi_3\ } \cdots,$ with
$B_n = M_{r_n}(A_{\Theta_n}),$
where each $\Theta_n$ is nondegenerate, each $r_n \geq 1$, and each $\phi_n$ is nonzero. Let 
$A_{\ex}=\varinjlim \bigl(B_n, \,\phi_n\bigr)$ be the corresponding noncommutative protori.  For each $n$, let $\tau_n:=\tau_n^{\Theta_n}=\frac{1}{r_n}\Tr_{r_n}\otimes\tau_{\Theta_n}$ be the normalized tracial state on $B_n$. Let
\[
\rho_n := (\tau_n)_* \colon K_0(B_n) \to \mathbb{R},
\qquad
f_{i,n} := (\phi_n)_{*i} \colon K_i(B_n) \to K_i(B_{n+1}), \quad i = 0,1,
\]
and let $\iota_n \colon B_n \to A_{\ex}$ denote the canonical maps. Because of the uniqueness of the trace, for each $n$ there is a unique scalar $t_n \in (0,1]$ such that $\rho_{n+1} \circ f_{0,n} = t_n\, \rho_n.$
We then define $c_1 := 1$, and $c_n := (t_1 t_2 \cdots t_{n-1})^{-1}$ for all $n \geq 2$.

\begin{proposition}[Regularity properties of simple protori]
\label{prop:protorus-regularity}
The following hold.

\begin{enumerate}[label={\rm(\roman*)}]
\item Each $\phi_n$ is injective and full, and $A_{\ex}$ is simple.
\item $A_{\ex}$ is separable, nuclear, hence amenable, and satisfies the UCT.
\item $A_{\ex}$ is $\mathcal Z$-stable, and has real rank zero and stable rank one .
\end{enumerate}
\end{proposition}

\begin{proof}

(i) Each $B_n=M_{r_n}(A_{\Theta_n})$ is simple because $A_{\Theta_n}$ is simple. 
Since $\phi_n\neq 0$ and the domain of $\phi_n$ is simple, $\phi_n$ is injective. Its range is full because the ideal generated by $\phi_n(B_n)$ contains the nonzero projection $\phi_n(1_{B_n})$, and $B_{n+1}$ is simple. It follows from the standard simplicity criterion for inductive limits with injective full connecting maps that $A_{\ex}$ is simple.

(ii) Each building block $B_n$ is separable, nuclear, simple, and in the UCT class (see Theorem~\ref{thm:phillips}); matrix algebras
preserve these properties. Since countable inductive limits preserve separability,
nuclearity, and membership in the bootstrap/UCT class, $A_{\ex}$ is separable, nuclear,
and satisfies the UCT.

(iii) For $\mathcal Z$-stability, note that each $B_n$ is a simple unital AT algebra,
hence a simple unital AH algebra of finite topological dimension, as AT algebras
are AH algebras with one-dimensional building blocks. 
By \cite[Corollary~3.1]{TomsWinterZASH}, each $B_n$ is $\mathcal Z$-stable.
Since $\mathcal Z$ is strongly self-absorbing and $\mathcal{Z}$-stability passes to inductive
limits of separable $C^*$-algebras by \cite[Corollary~3.4]{TomsWinterSSA}, it follows that $A_{\ex}\cong A_{\ex}\otimes \mathcal Z.$
Finally, each $B_n$ has real rank zero, and real rank zero is preserved under
inductive limits \cite{BlackadarDadarlatRordamRR}. Hence $A_{\ex}$ has real rank zero. Similarly, since each \(B_n\) has stable rank one and stable rank one is preserved under inductive limits, \(A_{\ex}\) has stable rank one as well.
\end{proof}

The use of real-rank-zero structure and ordered $K$-theory in the
classification of these limits is in the spirit of Elliott's real-rank-zero
classification program \cite{ElliottRR0}. For the simple $C^*$-algebras relevant here, we use the unified nonunital
Elliott invariant in the form described in
\cite[Definition~5.8]{GongLinNiuReview}.

Let $A$ be a separable simple \(C^*\)-algebra with
$\widetilde T(A)\neq\{0\}$, where $\widetilde T(A)$ denotes the cone of
densely defined positive lower semicontinuous traces on $A$, equipped with
the topology of pointwise convergence on the Pedersen ideal $\Ped(A)$.
The scale function is the element $\Sigma_A\in \LAff_+(\widetilde T(A))$
defined pointwise by $\Sigma_A(\tau):=
\sup\{\tau(a):a\in \Ped(A)_+,\ \|a\|\leq 1\},$ for all $\tau\in\widetilde T(A).$ Equivalently, if $e_A\in A_+$ is a strictly positive element, then
$\Sigma_A(\tau)=d_\tau(e_A),$ for $\tau\in\widetilde T(A)$,
and this function is independent of the choice of $e_A$.

The trace pairing is denoted $\rho_A\colon K_0(A)\to \Aff(\widetilde T(A)).$
For finite separable amenable simple $\mathcal Z$-stable $C^*$-algebras,
the positive cone is determined by the trace pairing: $K_0(A)^+=\{x\in K_0(A):\rho_A(x)>0\}\cup\{0\},$
where $\rho_A(x)>0$ means that $\rho_A(x)$ is strictly positive on
$\widetilde T(A)\setminus\{0\}$.

The distinguished set $\Sigma_\rho(K_0(A))$ is defined from the scale function
as follows. In the unital case,
\[
\Sigma_\rho(K_0(A))
=
\{x\in K_0(A)^+:\rho_A(x)<\Sigma_A\}\cup\{[1_A]\}.
\]
In the nonunital case,
\[
\Sigma_\rho(K_0(A))
=
\{x\in K_0(A)^+:\rho_A(x)<\Sigma_A\}.
\]
Here the inequality $\rho_A(x)<\Sigma_A$ is understood in the strict order used for scaled
ordered group pairings: $\Sigma_A-\rho_A(x)\in
\Aff_+(\widetilde T(A))$ and $\Sigma_A-\rho_A(x)\neq0$.
Equivalently, \(\rho_A(x)(\tau)<\Sigma_A(\tau)\) for all nonzero
\(\tau\in\widetilde T(A)\), in the cases considered here.

The unified Elliott invariant is
\[
\Ell(A)
:=
\bigl((K_0(A),\Sigma_\rho(K_0(A)),\widetilde T(A),\Sigma_A,\rho_A),
K_1(A)\bigr).
\]
For finite separable amenable simple $\mathcal Z$-stable algebras,
\cite[Theorem~5.11]{GongLinNiuReview} gives $\Sigma_\rho(K_0(A))=\Sigma(K_0(A)),$
where $\Sigma(K_0(A)):=\{[p]\in K_0(A):p\in A\text{ is a projection}\}$
is the projection scale of $A$. Thus, in the present finite setting, one may
equivalently write
\[
\Ell(A)
=
\bigl((K_0(A),\Sigma(K_0(A)),\widetilde T(A),\Sigma_A,\rho_A),
K_1(A)\bigr).
\]

If $A$ is unital, then $\Sigma(K_0(A))$ has maximum element $[1_A]$, and
the invariant can be written in the familiar unital form
\[
\Ell(A)
=
\bigl(K_0(A),K_0(A)^+,[1_A],T(A),\rho_A,K_1(A)\bigr).
\]
For nonunital finite separable amenable simple $\mathcal Z$-stable
$C^*$-algebras, \cite[Theorem~5.13]{GongLinNiuReview} shows that the
projection scale is determined by the scale function. Hence, in the nonunital
case, the invariant may be written in the reduced form
\begin{equation}\label{eq:ell}
\Ell(A)
=
\bigl((K_0(A),\widetilde T(A),\Sigma_A,\rho_A),K_1(A)\bigr).
\end{equation}
If \(A\) is stable, then the scale function carries no additional information,
and the invariant reduces further to
\[
\Ell(A)
=
\bigl(K_0(A),\widetilde T(A),\rho_A,K_1(A)\bigr).
\]

By the isomorphism theorem of Gong and Lin
\cite[Theorem~14.9]{GongLinNonunitalIV}, equivalently
\cite[Theorem~5.13]{GongLinNiuReview}, the above Elliott invariant is a
complete isomorphism invariant for finite separable amenable simple
$\mathcal Z$-stable $C^*$-algebras satisfying the UCT:
\[
A\cong B
\quad\Longleftrightarrow\quad
\Ell(A)\cong\Ell(B).
\]
In the nonunital case, one may equivalently use the reduced invariant
\eqref{eq:ell}.

The invariant $\Ell(A_{\ex})$ of our $C^*$-algebra $A_{\ex}$ is described below.

\begin{theorem}[Elliott invariant of a protoral system]
\label{thm:protorus-invariant}
Let $A_{\ex}=\varinjlim(B_n,\phi_n)$
be as above. 
\begin{enumerate}[label={\rm(\roman*)}]
\item For \(i=0,1\), one has $K_i(A_{\ex})\cong \varinjlim(K_i(B_n),f_{i,n}).$ The positive cone on \(K_0(A_{\ex})\) is the direct-limit cone.

\item The cone \(\widetilde T(A_{\ex})\) of densely defined lower semicontinuous
traces on \(A_{\ex}\) is one-dimensional. More precisely, there is a densely
defined lower semicontinuous trace \(\tau\) on \(A_{\ex}\), unique up to
multiplication by a positive scalar, such that $\tau\circ\iota_n=c_n\tau_n$
for all $n\geq 1$.
Equivalently, if $\rho_\tau:=\tau_*\colon K_0(A_{\ex})\to\mathbb R$
then $\rho_\tau\circ(\iota_n)_*=c_n\rho_n$ for all $n\geq 1$.
The full trace pairing $\rho_{A_{\ex}}\colon K_0(A_{\ex})\to \Aff(\widetilde T(A_{\ex}))$
is given by $\rho_{A_{\ex}}(x)(\lambda\tau)=\lambda\,\rho_\tau(x)$ for all
$x\in K_0(A_{\ex})$ and $\lambda\in[0,\infty).$

\item The scale function is $\Sigma_{A_{\ex}}(\lambda\tau)
=
\lambda\,\sup_{n\geq 1}c_n,$ for every $
\lambda\in[0,\infty),$
where the value \(+\infty\) is allowed, with the convention $0\cdot\infty=0$.

\item The projection scale is $\Sigma(K_0(A_{\ex}))
=
\bigcup_{n\geq 1}
[0,(\iota_n)_*([1_{B_n}])]
\subseteq K_0(A_{\ex})^+.$

\item If every $\phi_n$ is unital, then $A_{\ex}$ is unital,
$t_n=1$ for all $n$, and $\tau$ is the unique tracial state on
$A_{\ex}$.

\item If \(A_{\ex}\) is nonunital and \(c_n\to\infty\), then $\Sigma(K_0(A_{\ex}))=K_0(A_{\ex})^+.$
\end{enumerate}
\end{theorem}

\begin{proof}
Set $e_n:=\iota_n(1_{B_n})\in A_{\ex}.$ Then $(e_n)_{n\geq1}$ is an increasing approximate unit of projections
for $A_{\ex}$.
\smallskip

\noindent
\textup{(i)}
Continuity of $K$-theory for inductive limits gives $K_i(A_{\ex})\cong \varinjlim(K_i(B_n),f_{i,n})$ for 
$i=0,1.$ It remains to identify the positive cone. Let $V(A)$ denote the
Murray--von Neumann semigroup of projections in matrix algebras over a
$C^*$-algebra $A$. The functor $V$ is continuous for inductive limits, so $V(A_{\ex})\cong\varinjlim V(B_n).$
Each $B_n$ has real rank zero and stable rank one. Hence the natural map $V(B_n)\longrightarrow K_0(B_n)^+$
is an isomorphism of ordered semigroups. Passing to the direct limit gives that
$K_0(A_{\ex})^+$ is precisely the direct-limit cone.

\smallskip

\noindent
\textup{(ii)}
First we construct the trace. Since $\tau_{n+1}\circ\phi_n=t_n\tau_n,$ the traces $c_n\tau_n$
are compatible with the connecting maps: $c_{n+1}\tau_{n+1}\circ\phi_n
=c_{n+1}t_n\tau_n
=c_n\tau_n.$
Thus they define a positive trace $\tau_0$ on the algebraic inductive limit $\bigcup_{n\geq1}\iota_n(B_n),$ defined by
$\tau_0(\iota_n(a)):=c_n\tau_n(a),$ if
$a\in B_n.$
Let $\tau$ be its lower semicontinuous regularization on the $C^*$-completion
$A_{\ex}$:
\[
\tau(x):=\sup\{\tau_0(b):\, b\in (\bigcup_n \iota_n(B_n))_+, \ b\le x\}.
\]
Then $\tau$ is a densely defined lower semicontinuous trace on $A_{\ex}$,
and, since each $c_n\tau_n$ is bounded on $B_n$, the regularization agrees
with $c_n\tau_n$ on the finite stages, i.e.  $\tau\circ\iota_n=c_n\tau_n$ for all
$n\geq1$.
Indeed, if $x\in \iota_n(B_n)_+$ and $b\in\bigcup_m\iota_m(B_m)_+$ satisfies
$b\leq x$, then, after passing to a common later stage, the compatibility of
the traces gives $\tau_0(b)\leq \tau_0(x)$. Hence the regularization takes the
value $\tau_0(x)=c_n\tau_n(x)$ on $x$.

Passing to $K_0$ gives $\rho_\tau\circ(\iota_n)_*=c_n\rho_n,$
where $\rho_\tau=\tau_*$. This proves the displayed formula.

We now prove uniqueness of the trace ray. For $m\geq n$, write $\phi_{m,n}:=\phi_{m-1}\circ\cdots\circ\phi_n$ and
$\phi_{n,n}:=\id_{B_n}.$ Inside \(B_m\), put $p_{n,m}:=\phi_{m,n}(1_{B_n})$, and set $C_{n,m}:=p_{n,m}B_m p_{n,m}$ for simplicity.
Then $C_{n,m}$ is a nonzero full corner of the simple algebra $B_m$. Hence
$C_{n,m}$ is simple and has a unique tracial state, namely
\[
\omega_{n,m}(x)
=
\frac{\tau_m(x)}{\tau_m(p_{n,m})},
\qquad x\in C_{n,m}.
\]
The maps $C_{n,m}\longrightarrow C_{n,m+1}$
induced by $\phi_m$ are unital. Moreover, the traces $\omega_{n,m}$ are
compatible, because for $x\in C_{n,m}$,
\[
\frac{\tau_{m+1}(\phi_m(x))}
     {\tau_{m+1}(\phi_m(p_{n,m}))}
=
\frac{t_m\tau_m(x)}
     {t_m\tau_m(p_{n,m})}
=
\omega_{n,m}(x).
\]
Therefore $e_nA_{\ex}e_n
\cong
\varinjlim_{m\geq n}(C_{n,m},\phi_m|_{C_{n,m}})$
is a unital $C^*$-algebra with a unique tracial state. Denote this state by $\omega_n.$

Let $\sigma\in\widetilde T(A_{\ex})$ be a nonzero densely defined lower
semicontinuous trace. Since $\sigma$ is densely defined, it is finite on the
Pedersen ideal of $A_{\ex}$. In particular, it is finite on projections, so $\lambda_n:=\sigma(e_n)<\infty$ for every  $n\geq1$.
The restriction of $\sigma$ to the unital algebra $\iota_n(B_n)$, whose unit
is $e_n$, is therefore a bounded trace. Since $B_n$ has a unique tracial
state, there is a scalar $\lambda_n\geq0$ such that $\sigma\circ\iota_n=\lambda_n\tau_n.$
Compatibility with the connecting maps gives
\[
\lambda_n\tau_n
=
\sigma\circ\iota_n
=
\sigma\circ\iota_{n+1}\circ\phi_n
=
\lambda_{n+1}\tau_{n+1}\circ\phi_n
=
\lambda_{n+1}t_n\tau_n.
\]
Thus $\lambda_{n+1}={\lambda_n}{t_n}^{-1},$ and hence
$\lambda_n=\lambda_1c_n,$
for all $n\geq1$.

If $\lambda_1=0$, then $\lambda_n=0$ for every $n$. Since $e_nA_{\ex}e_n$
has unique tracial state, this implies $\sigma|_{e_nA_{\ex}e_n}=0$
for all $n\geq1$. For $a\in (A_{\ex})_+$, the positive elements $a^{1/2}e_na^{1/2}$
increase to $a$ and converge to
$a$ in norm. By traciality,
$\sigma(a^{1/2}e_na^{1/2})=\sigma(e_nae_n)=0.$
Lower semicontinuous traces satisfy monotone convergence, so $\sigma(a)=\sup_n\sigma(a^{1/2}e_na^{1/2})=0.$
Thus $\sigma=0$, contrary to our assumption. Hence $\lambda_1>0$.

Since $e_nA_{\ex}e_n$ has unique tracial state and $\sigma(e_n)=\lambda_n$, we have $\sigma|_{e_nA_{\ex}e_n}=\lambda_n\omega_n
=\lambda_1 c_n\omega_n.$
On the other hand, the trace $\tau$ constructed above satisfies $\tau(e_n)=c_n,$
and hence $\tau|_{e_nA_{\ex}e_n}=c_n\omega_n.$
Therefore $\sigma|_{e_nA_{\ex}e_n}=\lambda_1\,\tau|_{e_nA_{\ex}e_n}$ for every $n\geq1$.

Now let $a\in (A_{\ex})_+$. Again $a^{1/2}e_na^{1/2}\nearrow a,$
and by traciality,
\[
\sigma(a^{1/2}e_na^{1/2})
=
\sigma(e_nae_n)
=
\lambda_1\tau(e_nae_n)
=
\lambda_1\tau(a^{1/2}e_na^{1/2}).
\]
Using monotone convergence for both lower semicontinuous traces gives
\[
\sigma(a)
=
\sup_n\sigma(a^{1/2}e_na^{1/2})
=
\lambda_1\sup_n\tau(a^{1/2}e_na^{1/2})
=
\lambda_1\tau(a).
\]
Thus every nonzero densely defined lower semicontinuous trace is a positive
scalar multiple of $\tau$. Hence $\widetilde T(A_{\ex})=\R_+\tau.$ The final displayed formula in \textup{(ii)} follows from homogeneity:
for $x\in K_0(A_{\ex})$ and $\lambda\geq0$, $\rho_{A_{\ex}}(x)(\lambda\tau)=\lambda\,\tau_*(x).$

\smallskip

\noindent
\textup{(iii)}
Let $h\in (A_{\ex})_+$ be a strictly positive element. By the definition of the scale
function in the unified invariant, $\Sigma_{A_{\ex}}(\tau)=d_\tau(h).$
The support projection of $h$ in $A_{\ex}^{**}$ is $1_{A_{\ex}^{**}}$. Let
$\bar\tau$ denote the normal extension of $\tau$ to $A_{\ex}^{**}$. Then
$d_\tau(h)=\bar\tau(1_{A_{\ex}^{**}}).$
Since $(e_n)$ is an increasing approximate unit of projections, $e_n\nearrow
1_{A_{\ex}^{**}}$ strongly. Therefore
\[
\bar\tau(1_{A_{\ex}^{**}})
=
\sup_{n\geq1}\tau(e_n)
=
\sup_{n\geq1}c_n.
\]
Thus $\Sigma_{A_{\ex}}(\tau)=\sup_{n\geq1}c_n.$
By homogeneity of the scale function, $
\Sigma_{A_{\ex}}(\lambda\tau)=\lambda\,\sup_{n\geq1}c_n,$ for $
\lambda\in[0,\infty).$

\smallskip

\noindent
\textup{(iv)}
We prove the projection-scale formula $\Sigma(K_0(A_{\ex}))=\bigcup_{n\geq1}[0,[e_n]].$
Here and below $[e_n]$ means $(\iota_n)_*([1_{B_n}])$.

First let $p\in A_{\ex}$ be a projection. Since $e_n\to 1$ strictly, we have $e_npe_n\to p$ in norm. For $n$
sufficiently large, $e_npe_n$ is close enough to $p$ that the standard
projection perturbation lemma gives a projection $q\in e_nA_{\ex}e_n$, equivalently
$q\leq e_n$, which is Murray--von Neumann equivalent to $p$. Hence
$
[p]=[q]\leq [e_n].$
Therefore every projection class in $A_{\ex}$ belongs to
$
\bigcup_{n\geq1}[0,[e_n]].$

Conversely, suppose $x\in K_0(A_{\ex})^+$ and
$0\leq x\leq [e_n]$ for some $n$. Since $A_{\ex}$ is simple and $e_n\neq0$, the corner $e_nA_{\ex}e_n$
is full. Hence the inclusion $e_nA_{\ex}e_n\hookrightarrow A_{\ex}$
induces an order isomorphism on $K_0$. Let $x'\in K_0(e_nA_{\ex}e_n)^+$ be the
preimage of $x$. Then $0\leq x'\leq [1_{e_nA_{\ex}e_n}].$

The algebra $e_nA_{\ex}e_n$ has real rank zero and stable rank one, being a unital
corner of $A_{\ex}$. Therefore $x'$ is represented by a projection
in $e_nA_{\ex}e_n$.
Indeed, as $0\leq x'\leq [1_{e_nA_{\ex}e_n}]$, choose projections $q,r$ in matrix
algebras over $e_nA_{\ex}e_n$ with $[q]=x'$ and 
$[r]=[1_{e_nA_{\ex}e_n}]-x'.$
Then $[q]+[r]=[1_{e_nA_{\ex}e_n}].$ Since $e_nA_{\ex}e_n$ has stable rank one, projections over
$e_nA_{\ex}e_n$ satisfy cancellation; equivalently, the natural map $V(e_nA_{\ex}e_n)\to K_0(e_nA_{\ex}e_n)^+$
is injective. Hence the equality $[q]+[r]=[1_{e_nA_{\ex}e_n}]$
in $K_0$ implies that $q\oplus r$ is Murray--von Neumann equivalent to
$1_{e_nA_{\ex}e_n}$ in a matrix algebra. Therefore $q$ is Murray--von
Neumann equivalent to a subprojection of $1_{e_nA_{\ex}e_n}$, so $x'$ is
represented by an actual projection in $e_nA_{\ex}e_n$.

Therefore there exists a projection $p\in e_nA_{\ex}e_n\subseteq A_{\ex}$
such that
$[p]=x.$ Thus every element of $[0,[e_n]]$ is a projection class in $A_{\ex}$. This proves
$
\Sigma(K_0(A_{\ex}))
=\bigcup_{n\geq1}[0,[e_n]]
=\bigcup_{n\geq1}[0,(\iota_n)_*([1_{B_n}])].$

\smallskip

\noindent
\textup{(v)}
Assume every $\phi_n$ is unital. Then
$\phi_n(1_{B_n})=1_{B_{n+1}},$
so $e_n=e_{n+1}$
for all $n$. Hence $A_{\ex}$ is unital with unit $e_1$. Moreover, $t_n
=\tau_{n+1}(\phi_n(1_{B_n}))
=\tau_{n+1}(1_{B_{n+1}})=1.$
Thus $c_n=1$ for all $n$. The trace $\tau$ satisfies
$\tau(e_1)=1,$
so $\tau$ is a tracial state. Since the whole lower semicontinuous trace cone
is $\R_+\tau$, this tracial state is unique.

\smallskip

\noindent
\textup{(vi)}
Assume $A_{\ex}$ is nonunital and $c_n\to\infty$. Let
$x\in K_0(A_{\ex})^+.$
If $x=0$, then $x\in\Sigma(K_0(A_{\ex}))$.
Now assume $x\neq0$. By \textup{(i)}, there exist $n$ and $y\in K_0(B_n)^+\setminus\{0\}$ such
that $(\iota_n)_*(y)=x.$
Let $m\geq n$, and write
$f_{0,m,n}:=f_{0,m-1}\circ\cdots\circ f_{0,n}$
for the induced map from \(K_0(B_n)\) to \(K_0(B_m)\). Put
$y_m:=f_{0,m,n}(y)\in K_0(B_m)^+.$ But 
$\rho_\tau(x)=c_m\rho_m(y_m)$ by \textup{(ii)}
so $\rho_m(y_m)=c_m^{-1}\rho_\tau(x).$
Since $c_m\to\infty$, choose $m\geq n$ so large that
$c_m>\rho_\tau(x).$
Then $\rho_m(y_m)<1=\rho_m([1_{B_m}]).$
Hence $\rho_m([1_{B_m}]-y_m)>0.$ But
the order on \(K_0(B_m)\) is determined by the unique trace, so
$[1_{B_m}]-y_m\in K_0(B_m)^+.$
Thus $0\leq y_m\leq [1_{B_m}].$
Applying $(\iota_m)_*$, we get
$0\leq x\leq (\iota_m)_*([1_{B_m}])=[e_m].$
By \textup{(iv)}, $x\in\Sigma(K_0(A_{\ex})).$
Therefore $\Sigma(K_0(A_{\ex}))=K_0(A_{\ex})^+$, as claimed. The proof is complete.
\end{proof}

In the setting of Theorem~\ref{thm:protorus-invariant}, 
$B_n=M_{r_n}(A_{\Theta_n})$
and 
let $\mu_{i,n}\colon K_i(B_n)\xrightarrow{\cong}K_i(A_{\Theta_n})$
denote the standard Morita identifications, for $i=0,1$. 
The following formulas describe the
induced maps on $K$-theory and the trace-scaling constants. 

\begin{proposition}[The four classes of connecting maps on invariants]
\label{prop:four-cases-invariants}
\leavevmode

\smallskip

\noindent
 
\emph{(1) [Unital toric maps].}
Suppose $r_{n+1}=r_n$ and $\phi_n=\id_{M_{r_n}}\otimes \varphi_{M_n,z_n}
:
M_{r_n}(A_{\Theta_n})\longrightarrow M_{r_n}(A_{\Theta_{n+1}}),$
where $\varphi_{M_n,z_n}$ is the monomial map of
Proposition~\ref{prop:toric-explicit}. Then, under the Morita identifications
$\mu_{i,n}$ and $\mu_{i,n+1}$, $\mu_{0,n+1}\circ f_{0,n}\circ \mu_{0,n}^{-1}=\Lambda^{\mathrm{even}}(M_n)$
and $\mu_{1,n+1}\circ f_{1,n}\circ \mu_{1,n}^{-1}=\Lambda^{\mathrm{odd}}(M_n).$
Moreover, $t_n=1.$

More generally, if a unital diagonal amplification with \(s_n\) identical
toric blocks is inserted, so that \(r_{n+1}=s_nr_n\), then the right-hand sides
are multiplied by \(s_n\), while still \(t_n=1\).
\smallskip

\noindent
 
\emph{(2) [Pure corner maps].}
Suppose $\phi_n=\iota_{p_n}\circ \alpha_n:
B_n\xrightarrow{\ \alpha_n\ }p_nB_{n+1}p_n
\xrightarrow{\ \iota_{p_n}\ }B_{n+1},$
where \(p_n\in B_{n+1}\) is a nonzero full projection and
\(\alpha_n\) is a unital isomorphism onto the full corner \(p_nB_{n+1}p_n\).
Then $f_{i,n}=(\iota_{p_n})_*\circ(\alpha_n)_*$ for $i=0,1.$
The map $(\iota_{p_n})_*$ is the $K$-theory isomorphism induced by the
full-corner Morita equivalence. The trace-scaling constant is
$t_n=\tau_{n+1}(p_n)=\tau_{n+1}(\phi_n(1_{B_n})).$
\smallskip

\noindent

\emph{(3) [Dimension-changing corner maps].}
Suppose $\phi_n=\iota_{p_n}\circ\alpha_n\circ
(\id_{M_{r_n}}\otimes\varphi_{M_n,z_n}),$
where $\id_{M_{r_n}}\otimes\varphi_{M_n,z_n} : B_n\longrightarrow C_n:=M_{r_n}(A_{\Psi'_n})$
is induced by a full-column-rank matrix $M_n\in M_{d_{n+1}\times d_n}(\Z),$ $\Psi'_n\in M_{d_{n+1}}(\R)_{\mathrm{skew}}$,
and $\alpha_n:C_n\xrightarrow{\cong}p_nB_{n+1}p_n$
is a unital corner identification. Let $\nu_{i,n}\colon K_i(C_n)\xrightarrow{\cong}K_i(A_{\Psi'_n})$
be the standard Morita identification, and define
\[
\beta_{i,n}:=(\iota_{p_n})_*\circ(\alpha_n)_*\circ \nu_{i,n}^{-1}:K_i(A_{\Psi'_n})\longrightarrow K_i(B_{n+1}).
\]
Then $f_{0,n}=\beta_{0,n}\circ \Lambda^{\mathrm{even}}(M_n)\circ \mu_{0,n},$
and $f_{1,n}=\beta_{1,n}\circ \Lambda^{\mathrm{odd}}(M_n)\circ \mu_{1,n}.$
Moreover, $t_n=\tau_{n+1}(p_n)=\tau_{n+1}(\phi_n(1_{B_n})).$
\smallskip

\noindent

\emph{(4) [Same-dimensional proper toric maps].}
Suppose $\phi_n=\iota_{p_n}\circ\alpha_n\circ (\id_{M_{r_n}}\otimes\varphi_{M_n,z_n}),$
where now $M_n\in M_d(\Z)$ is such that
$|\det(M_n)|>1,$ and $\varphi_{M_n,z_n}:A_{\Theta_n}\to A_{\Psi'_n}$
is the corresponding same-dimensional toric map. Let
$C_n:=M_{r_n}(A_{\Psi'_n})$ and let $\alpha_n:C_n\xrightarrow{\cong}p_nB_{n+1}p_n$
be a unital corner identification. Define 
\[
\beta_{i,n}:=(\iota_{p_n})_*\circ(\alpha_n)_*\circ \nu_{i,n}^{-1}
:K_i(A_{\Psi'_n})\longrightarrow K_i(B_{n+1}),
\]
where $\nu_{i,n}:K_i(C_n)\xrightarrow{\cong}K_i(A_{\Psi'_n})$
is the standard Morita identification. Then $f_{0,n}
=\beta_{0,n}\circ \Lambda^{\mathrm{even}}(M_n)\circ \mu_{0,n},$
and $f_{1,n}=\beta_{1,n}\circ \Lambda^{\mathrm{odd}}(M_n)\circ \mu_{1,n}.$
Moreover, $t_n=\tau_{n+1}(p_n)=\tau_{n+1}(\phi_n(1_{B_n})).$
\end{proposition}
\begin{proof}
In Case~\textup{(1)}, the toric part is unital. By
Proposition~\ref{prop:Ktoric-explicit}, with the exterior-algebra convention of
Convention~\ref{conv:toric-K-identification}, the toric map contributes $\Lambda^{\mathrm{even}}(M_n)$
on $K_0$
and $\Lambda^{\mathrm{odd}}(M_n)$ on $K_1$.
Matrix amplification by \(M_{r_n}\) does not change these maps under the
standard Morita identifications. Hence the displayed formulas hold. Since the
map is unital and both the domain and codomain have unique normalized traces,
$\tau_{n+1}\circ\phi_n=\tau_n$,
so $t_n=1$. 

For the diagonal-amplified variant, write $\psi_n:=\id_{M_{r_n}}\otimes\varphi_{M_n,z_n}.$
The map is
\[
a\longmapsto
\operatorname{diag}(\psi_n(a),\ldots,\psi_n(a))\]
with $s_n$ diagonal blocks. On $K$-theory, block diagonal sum adds classes,
so the induced map is $s_n$ times the map induced by $\psi_n$. Thus the
right-hand sides are multiplied by $s_n$. On traces, the normalized trace on
$M_{s_nr_n}(A_{\Theta_{n+1}})$ contains the factor $1/s_n$, which cancels
the $s_n$ identical blocks. Hence the normalized trace is still preserved, and
again $t_n=1$.

In Case~\textup{(2)}, the equality $f_{i,n}=(\iota_{p_n})_*\circ(\alpha_n)_*$
is simply functoriality of $K$-theory. Since $\alpha_n$ is unital onto the
corner, for $a\in B_n$ we have $\phi_n(a)\in p_nB_{n+1}p_n.$
The unique normalized trace on the corner $p_nB_{n+1}p_n$ is $x\longmapsto \frac{\tau_{n+1}(x)}{\tau_{n+1}(p_n)}.$
Thus
\[
\frac{\tau_{n+1}(\phi_n(a))}{\tau_{n+1}(p_n)}=\tau_n(a),
\]
because $\alpha_n$ is a unital isomorphism between uniquely traced unital
simple algebras. Hence $\tau_{n+1}(\phi_n(a))=\tau_{n+1}(p_n)\tau_n(a),$ so $t_n=\tau_{n+1}(p_n)=\tau_{n+1}(\phi_n(1_{B_n})).$

Cases~\textup{(3)} and \textup{(4)} are obtained by the same two ingredients.
The toric part contributes $\Lambda^{\mathrm{even}}(M_n)$ on $K_0$
and $\Lambda^{\mathrm{odd}}(M_n)$ on $K_1$ by Proposition~\ref{prop:Ktoric-explicit}, again with the convention
of Convention~\ref{conv:toric-K-identification}.  The corner part contributes
$(\iota_{p_n})_*\circ(\alpha_n)_*.$
After inserting the Morita identification $\nu_{i,n}:K_i(M_{r_n}(A_{\Psi'_n}))\xrightarrow{\cong}K_i(A_{\Psi'_n}),$
this gives exactly $f_{0,n}=\beta_{0,n}\circ \Lambda^{\mathrm{even}}(M_n)\circ \mu_{0,n},$
and $f_{1,n}=\beta_{1,n}\circ \Lambda^{\mathrm{odd}}(M_n)\circ \mu_{1,n}.$

Since the toric part is unital and trace-preserving with respect to normalized
traces, and since $\alpha_n$ is unital onto the corner, the same corner-trace
calculation as in Case~\textup{(2)} gives $\tau_{n+1}\circ\phi_n=\tau_{n+1}(p_n)\tau_n.$
Therefore $t_n=\tau_{n+1}(p_n)=\tau_{n+1}(\phi_n(1_{B_n})).$

In Case~\textup{(4)}, the condition $|\det(M_n)|>1$ is precisely the
same-dimensional proper toric subalgebra condition:  if $|\det(M_n)|=1$, then
$M_n\in\GL_d(\Z)$, and the toric part is onto the intermediate torus. If $|\det(M_n)|>1$, then $M_n\colon\Z^d\to\Z^d$ is not surjective. Hence
$\Lambda^{\mathrm{odd}}(M_n)$ is not an isomorphism. Since $A_{\Theta_n}$
is simple, the unital toric map is injective. If it were also surjective, it
would be an isomorphism and would induce an isomorphism on $K_1$, a
contradiction. Therefore its range is a proper unital subalgebra of
$A_{\Psi'_n}$.
\end{proof}

\begin{remark}
\label{rem:latpack}
We now explain the relation with the noncommutative solenoids of
Latr\'emoli\`ere and Packer. Let $N>1$ and let $\alpha=(\alpha_n)_{n\ge0}\in \Xi_N$ in the sense of
\cite[Theorem~2.1]{LatPackSol}. For each $n\geq 0$, set
\[
\Theta_n:=J(\alpha_{2n})=
\begin{pmatrix}
0&\alpha_{2n}\\
-\alpha_{2n}&0
\end{pmatrix}.
\]
Theorem~3.6 of \cite{LatPackSol} realizes the noncommutative solenoid
$A_\alpha^{\mathscr S}$ as the inductive limit of the rotation algebras
$A_{\alpha_{2n}}$ under the embeddings $\varphi_n\colon A_{\alpha_{2n}}\longrightarrow A_{\alpha_{2n+2}}$,
\[U_{\alpha_{2n}}\longmapsto U^N_{\alpha_{2n+2}},\qquad V_{\alpha_{2n}}\longmapsto V^N_{\alpha_{2n+2}}.
\]
These are exactly the unital toric maps of Case~\textup{(1)} corresponding to the
matrix $M=NI_2$, with trivial phase vector $z=(1,1)$. Indeed, since
$\alpha\in\Xi_N$, for each $n$ one has $N\alpha_{n+1}-\alpha_n\in \Z,$
and iterating once gives $N^2\alpha_{2n+2}-\alpha_{2n}\in \Z.$
Therefore $M^t\Theta_{n+1}M=N^2J(\alpha_{2n+2})
\equiv
J(\alpha_{2n})=\Theta_n
\pmod{M_2(\Z)_{\mathrm{skew}}},$
so Proposition~\ref{prop:toric-explicit} applies. Thus the noncommutative
solenoids of \cite{LatPackSol} form a distinguished subclass of the Case~\textup{(1)}
systems considered in the present paper.

This observation places the irrational solenoids squarely inside our general
direct-limit framework. Indeed, \cite[Proposition~3.19]{LatPackSol} shows that
if $\alpha_0\notin\Q$ (equivalently, if some $\alpha_k$ is irrational), then
$A_\alpha^{\mathscr S}$ is a simple AT-algebra of real rank zero. Hence, for
this subclass, the Elliott invariant is complete. Moreover,
\cite[Theorem~3.7]{LatPackSol} computes the corresponding direct-limit invariant
explicitly: $K_1(A_\alpha^{\mathscr S})\cong \mathbb Z[1/N]^2$ and $K_0(A_\alpha^{\mathscr S})\cong K_\alpha,$
where
\[
K_\alpha=
\left\{
\left(z+\frac{pJ_k^\alpha}{N^k},\frac{p}{N^k}\right)
:\ z,p\in\Z,\ k\in\N
\right\}
\subseteq \mathbb Z[1/N]^2,
\qquad
J_k^\alpha:=N^k\alpha_k-\alpha_0,
\]
with distinguished order unit $[1]$ and trace map
\[
K_0(\tau)\colon
\left(z+\frac{pJ_k^\alpha}{N^k},\frac{p}{N^k}\right)
\longmapsto z+p\alpha_k.
\]

From our point of view, these are the invariant data of a Case~\textup{(1)}
system whose toric matrices are all $NI_2$. In the exterior-algebra
coordinates used in Proposition~\ref{prop:Ktoric-explicit}, the $K$-theory
maps are given by $\Lambda^*(NI_2)$. 
Thus the induced map on $K_1$ is
$NI_2$, giving $\varinjlim(\Z^2,NI_2)\cong \mathbb Z[1/N]^2.$
On $K_0$, the degree-zero class is fixed and the top exterior class is
multiplied by $N^2$. In the usual rotation-algebra
coordinates used in \cite{LatPackSol}, the integer congruence terms
$N^2\alpha_{2n+2}-\alpha_{2n}$ produce the shear terms encoded by
$J_k^\alpha=N^k\alpha_k-\alpha_0$ in the group $K_\alpha$.  Although the inductive-limit model uses the even
subsequence $(\alpha_{2n})$, the subgroup obtained from the even indices is
the same subgroup $K_\alpha$ written above. Indeed, if $k\geq0$ and $\ell\geq k$ is even, then
$N^{\ell-k}\alpha_\ell-\alpha_k\in\Z$.  Writing this integer as $s$, one has $J_\ell^\alpha=N^\ell\alpha_\ell-\alpha_0=J_k^\alpha+N^k s.$
Thus an element represented using the index $k$ can also be represented using
the even index $\ell$, after changing the integer $z$.  Hence the even
subsequence gives the same subgroup $K_\alpha$.

The full classification theorem \cite[Theorem~4.2]{LatPackSol} is slightly
broader than the present Elliott-invariant discussion, because it treats all
noncommutative solenoids, including regimes outside the simple AT real-rank-zero
case. For example, the periodic rational algebras described in
\cite[Theorem~3.20]{LatPackSol} are not simple, and the authors note immediately
after that theorem that the embeddings from \cite[Theorem~3.6]{LatPackSol} then
land in the centers of the rotation algebras. For the irrational subclass,
however, \cite[Theorem~4.2]{LatPackSol} may be viewed as an explicit translation
of equality of the Elliott invariant into the sequence language of $\Xi_N$: the
requirement that $N$ and $M$ have the same prime factors is already reflected in
the isomorphism type of $K_1(A_\alpha^{\mathscr S})\cong \mathbb Z[1/N]^2,$
while the common-subsequence/sign condition records the agreement of the ordered
$K_0$-extension together with its distinguished order unit and trace. In this
sense, \cite[Theorem~4.2]{LatPackSol} is the two-dimensional solenoid analogue
of the general principle developed here: once one is in the simple AT
real-rank-zero regime, explicit computations of the direct-limit Elliott
invariant lead to concrete isomorphism criteria.
\end{remark}

\begin{example}[A unital toric $N$-solenoid and its classification]
\label{ex:solenoid}
Fix an irrational number $\theta\in\R$ and an integer $N\geq2$. For
$n\geq1$, set $\theta_n:=\frac{\theta}{N^{2(n-1)}}$ and $A_n:=A_{\theta_n}$,
and let $U_n,V_n$ denote the canonical generators of $A_n$. Since $N^2\theta_{n+1}=\theta_n,$
we have $(NI_2)^tJ(\theta_{n+1})(NI_2)=J(\theta_n)$, where
$J(\alpha):=
\begin{pmatrix}
0&\alpha\\
-\alpha&0
\end{pmatrix}.$
Thus Proposition~\ref{prop:toric-explicit} gives a unital toric map $\phi_n\colon A_n\to A_{n+1}$
corresponding to the matrix $NI_2$, explicitly $\phi_n(U_n)=U_{n+1}^{N},$ and
$\phi_n(V_n)=V_{n+1}^{N}.$

Define
\[
A_{\ex}^{\theta,N}:=\varinjlim(A_n,\phi_n).
\]
Then $A_{\ex}^{\theta,N}$ is a simple unital noncommutative protorus. Let $R_N:=\Z[1/N].$
Since $R_N=\Z[1/N^2]$, Proposition~\ref{prop:four-cases-invariants} gives
\[
f_{0,n}
=
\Lambda^{\mathrm{even}}(NI_2)
=
\begin{pmatrix}
1&0\\
0&N^2
\end{pmatrix},
\qquad
f_{1,n}
=
\Lambda^{\mathrm{odd}}(NI_2)
=
NI_2.
\]
Therefore
$K_0(A_{\ex}^{\theta,N})
\cong
\varinjlim\left(\Z^2,
\begin{pmatrix}
1&0\\
0&N^2
\end{pmatrix}
\right)
\cong
\Z\oplus R_N,$
and $K_1(A_{\ex}^{\theta,N})
\cong
\varinjlim(\Z^2,NI_2)
\cong
R_N^2.$
Under the above identification of $K_0$, the order unit is $[1_{A_{\ex}^{\theta,N}}]=(1,0).$

All connecting maps are unital, so $t_n=1$ for every $n$. Hence
Theorem~\ref{thm:protorus-invariant} gives a unique tracial state on
$A_{\ex}^{\theta,N}$. At stage $n$, $\rho_n(K_0(A_n))
=
\Z+\theta_n\Z
=
\Z+\frac{\theta}{N^{2(n-1)}}\Z.$
Passing to the direct limit gives
\[
\rho_{\theta,N}(K_0(A_{\ex}^{\theta,N}))
=
\bigcup_{n\geq1}
\left(
\Z+\frac{\theta}{N^{2(n-1)}}\Z
\right)
=
\Z+\theta R_N,
\]
where $\rho_{\theta,N}:=(\tau_{\theta,N})_*
\colon K_0(A_{\ex}^{\theta,N})\to\R$
is the trace map induced by the unique tracial state \(\tau_{\theta,N}\).
Under the identification $K_0(A_{\ex}^{\theta,N})\cong \Z\oplus R_N,$
one has $\rho_{\theta,N}(a,q)=a+\theta q.$
Consequently
\[
K_0(A_{\ex}^{\theta,N})^+
=
\{(a,q)\in \Z\oplus R_N : a+\theta q>0\}\cup\{0\},
\]
and $[1_{A_{\ex}^{\theta,N}}]=(1,0).$

We now classify these algebras by their Elliott invariants. Let
\(\theta'\in\R\setminus\Q\) and \(M\geq2\), and construct
\(A_{\ex}^{\theta',M}\) in the same way. Let $\mathcal P(N):=\{p\text{ prime}:p\mid N\}.$
Then $A_{\ex}^{\theta,N}\cong A_{\ex}^{\theta',M}$
as unital \(C^*\)-algebras if and only if $\mathcal P(N)=\mathcal P(M)$
and $\frac{\theta}{\theta'}\in R_N^\times.$
Here
\[
R_N^\times
=
\left\{
\pm\prod_{p\in\mathcal P(N)}p^{k_p}
:
k_p\in\Z
\right\}.
\]
If $\theta,\theta'>0$, the second condition is equivalently $\theta=u\theta'$
for some positive unit $u\in R_N^\times$.

Indeed, suppose first that $A_{\ex}^{\theta,N}\cong A_{\ex}^{\theta',M}.$
Then their \(K_1\)-groups are isomorphic:
$R_N^2\cong R_M^2.$
For a prime $p$, the group $R_N^2$ is $p$-divisible if and only if
$p\mid N$. Since $p$-divisibility is preserved by group isomorphisms, we
must have $\mathcal P(N)=\mathcal P(M).$
Thus $R_N=R_M.$
Write this common group as $R$.

The isomorphism of unital Elliott invariants gives an order-unit preserving
group isomorphism $\alpha_0\colon \Z\oplus R\to \Z\oplus R$
such that $\alpha_0(1,0)=(1,0)$
and $\rho_{\theta',M}\circ \alpha_0=\rho_{\theta,N}.$
Since $\operatorname{Hom}(R,\Z)=0,$
every order-unit preserving automorphism of $\Z\oplus R$ has the form $\alpha_0(a,q)=(a,uq)$
for some unit \(u\in R^\times\). The trace-compatibility condition gives $a+\theta' u q=a+\theta q$
for all $a\in\Z$ and $q\in R$. Hence $\theta=\theta'u,$
or equivalently $\frac{\theta}{\theta'}=u\in R_N^\times.$

Conversely, suppose that $\mathcal P(N)=\mathcal P(M)$
and $\theta=u\theta'$
for some $u\in R_N^\times$. Then $R_N=R_M=:R$.
Define $\alpha_0\colon \Z\oplus R\to \Z\oplus R$ by $\alpha_0(a,q):=(a,uq).$
This is a group isomorphism, preserves the order unit $(1,0)$, and satisfies
$\rho_{\theta',M}(\alpha_0(a,q))=\rho_{\theta',M}(a,uq)=a+\theta'uq=a+\theta q=\rho_{\theta,N}(a,q).$
Therefore $\alpha_0$ preserves the positive cones. On \(K_1\), we may take
any group isomorphism $\alpha_1\colon R_N^2\to R_M^2,$
for example the identity after identifying \(R_N=R_M\). Thus the unital
Elliott invariants of \(A_{\ex}^{\theta,N}\) and \(A_{\ex}^{\theta',M}\) are
isomorphic. By the classification theorem for the simple unital class considered
above, $A_{\ex}^{\theta,N}\cong A_{\ex}^{\theta',M}.$

In particular, the algebra remembers the set of prime divisors of $N$, but
not the integer $N$ itself. For example, $A_{\ex}^{\theta,2}\cong A_{\ex}^{\theta,4}$, whereas
$A_{\ex}^{\theta,2}\not\cong A_{\ex}^{\theta',6}$ for every irrational
$\theta'$, since $\mathcal P(2)\neq\mathcal P(6)$.
\end{example}

\begin{example}[A pure-corner stable protorus and its classification]
\label{ex:stable-corner}
Fix a nondegenerate skew-symmetric matrix $\Theta\in M_d(\R)$. Put $m_n:=2^{n-1}$ and $B_n^\Theta:=M_{m_n}(A_\Theta)$,
for all $n\geq1$. Define $\phi_n^\Theta\colon B_n^\Theta\longrightarrow B_{n+1}^\Theta
=M_{2m_n}(A_\Theta)$
by the upper-left corner inclusion $\phi_n^\Theta(a)=
\begin{pmatrix}
a&0\\
0&0
\end{pmatrix}.$
Equivalently, if $p_n:=
\begin{pmatrix}
1_{B_n^\Theta}&0\\
0&0
\end{pmatrix}
\in B_{n+1}^\Theta,$
then \(\phi_n^\Theta\) is the composite $
B_n^\Theta
\xrightarrow{\ \cong\ }
p_nB_{n+1}^\Theta p_n
\xrightarrow{\ \iota_{p_n}\ }
B_{n+1}^\Theta.$
Thus this is a pure-corner system. Define
\[
A_{\ex}^{\Theta}
:=
\varinjlim(B_n^\Theta,\phi_n^\Theta).
\]

Then $A_{\ex}^{\Theta}\cong A_\Theta\otimes\K.$
Indeed, identify $M_{m_n}(A_\Theta)$ with the corner $(1_{A_\Theta}\otimes e_{m_n})
(A_\Theta\otimes\K)
(1_{A_\Theta}\otimes e_{m_n}),$
where $e_{m_n}$ is the rank-$m_n$ projection onto the first $m_n$ basis
vectors. Under this identification, the maps $\phi_n^\Theta$ are precisely
the inclusions of these corners. Since $e_{m_n}\nearrow 1$ strictly in $M(\K)=B(\ell^2)$, the union of these corners is dense in
$A_\Theta\otimes\K$. Hence the inductive limit is $A_\Theta\otimes\K$.

Under the standard Morita identifications $K_i(B_n^\Theta)\cong K_i(A_\Theta),$ 
the maps $(\phi_n^\Theta)_{*i}\colon K_i(B_n^\Theta)\to K_i(B_{n+1}^\Theta)$
are the identity maps. Therefore
$K_0(A_{\ex}^{\Theta})
\cong K_0(A_\Theta)
\cong \Lambda^{\mathrm{even}}\Z^d,$
and $K_1(A_{\ex}^{\Theta})
\cong K_1(A_\Theta)
\cong \Lambda^{\mathrm{odd}}\Z^d.$

The normalized trace on
$B_n^\Theta=M_{m_n}(A_\Theta)$
is $\tau_n^\Theta=\frac{1}{m_n}\Tr_{m_n}\otimes\tau_\Theta=2^{-(n-1)}\Tr_{2^{n-1}}\otimes\tau_\Theta,$
where $\Tr_{m_n}$ denotes the unnormalized matrix trace. For
$a\in B_n^\Theta$, we have
\[
\tau_{n+1}^\Theta(\phi_n^\Theta(a))
=
\frac{1}{2m_n}
(\Tr_{2m_n}\otimes\tau_\Theta)
\begin{pmatrix}
a&0\\
0&0
\end{pmatrix}
=
\frac12\,\tau_n^\Theta(a).
\]
Thus $t_n=1/2$ and $c_n=(t_1\cdots t_{n-1})^{-1}=2^{n-1}=m_n.$

By Theorem~\ref{thm:protorus-invariant}, the cone of densely defined lower
semicontinuous traces on \(A_{\ex}^{\Theta}\) is one-dimensional: $\widetilde T(A_{\ex}^{\Theta})=\R_+\tau^\Theta,$
where the distinguished generator $\tau^\Theta$ is characterized by
$\tau^\Theta\circ\iota_n=c_n\tau_n^\Theta=
\Tr_{m_n}\otimes\tau_\Theta.$
Under the identification $A_{\ex}^{\Theta}\cong A_\Theta\otimes\K,$
this trace is exactly $\tau_\Theta\otimes\Tr,$
where $\Tr$ is the usual unbounded trace on $\K$.

Let $\rho_{\Theta}^{\ex}:=(\tau^\Theta)_*
\colon K_0(A_{\ex}^{\Theta})\to\R.$
Since, under the Morita identification $K_0(B_n^\Theta)\cong K_0(A_\Theta),$
one has $
\rho_n^\Theta
=
\frac{1}{m_n}\rho_\Theta,$
the formula $
\rho_{\Theta}^{\ex}\circ(\iota_n)_*=c_n\rho_n^\Theta$
gives $\rho_{\Theta}^{\ex}\circ(\iota_n)_*=
m_n\cdot\frac{1}{m_n}\rho_\Theta=\rho_\Theta.$
Hence $\rho_{\Theta}^{\ex}(K_0(A_{\ex}^{\Theta}))
=
\rho_\Theta(K_0(A_\Theta)).$

The positive cone is transported by Morita equivalence. Therefore, under the
standard identification $K_0(A_{\ex}^{\Theta})\cong K_0(A_\Theta),$
one has $K_0(A_{\ex}^{\Theta})^+
=
K_0(A_\Theta)^+.$
Equivalently $K_0(A_{\ex}^{\Theta})^+
=
\{x\in K_0(A_{\ex}^{\Theta}):\rho_{\Theta}^{\ex}(x)>0\}
\cup\{0\}.$

Since $c_n=2^{n-1}\to\infty,$
Theorem~\ref{thm:protorus-invariant} gives the full projection scale $\Sigma(K_0(A_{\ex}^{\Theta}))
=
K_0(A_{\ex}^{\Theta})^+.$
Equivalently, every positive \(K_0\)-class is represented by a projection in
the stable algebra $A_{\ex}^{\Theta}\cong A_\Theta\otimes\K.$
The scale function is
\[
\Sigma_{A_{\ex}^{\Theta}}(\lambda\tau^\Theta)
=
\begin{cases}
0, & \lambda=0,\\
+\infty, & \lambda>0.
\end{cases}
\]

Using the reduced stable form of the Elliott invariant, we obtain
\[
\Ell(A_{\ex}^{\Theta})
=
\left(
K_0(A_\Theta),
\widetilde T(A_{\ex}^{\Theta})=\mathbb R_+\tau^\Theta,
\rho_{\Theta}^{\ex},
K_1(A_\Theta)
\right),
\]
with full projection scale and infinite scale function on every nonzero trace.

We now classify these stable pure-corner protori by their Elliott invariants.
Let $\Psi\in M_e(\R)$ be another nondegenerate skew-symmetric matrix, and form $A_{\ex}^{\Psi}$ in the same way. Then $A_{\ex}^{\Theta}\cong A_{\ex}^{\Psi}$
if and only if there exist graded group isomorphisms
\[
\Gamma_0\colon K_0(A_\Theta)\xrightarrow{\cong}K_0(A_\Psi),
\qquad
\Gamma_1\colon K_1(A_\Theta)\xrightarrow{\cong}K_1(A_\Psi),
\]
and a scalar $\lambda>0$ such that $\rho_\Psi\circ\Gamma_0=\lambda\,\rho_\Theta.$

Indeed, an isomorphism $A_{\ex}^{\Theta}\cong A_{\ex}^{\Psi}$
induces isomorphisms on $K_0$ and $K_1$. Since the trace cones are both
one-dimensional rays, the induced affine homeomorphism of trace cones sends
the distinguished generator $\tau^\Psi$ to $\lambda\tau^\Theta$ for some
$\lambda>0$. Compatibility of the Elliott invariant then gives $\rho_\Psi\circ\Gamma_0=\lambda\rho_\Theta.$

Conversely, suppose such $\Gamma_0,\Gamma_1$, and $\lambda$ exist. Since
the positive cones are trace-determined, the relation $\rho_\Psi\circ\Gamma_0=\lambda\rho_\Theta$
implies that $\Gamma_0$ is an order isomorphism. The projection scales are
both full,
\[
\Sigma(K_0(A_{\ex}^{\Theta}))=K_0(A_{\ex}^{\Theta})^+,
\qquad
\Sigma(K_0(A_{\ex}^{\Psi}))=K_0(A_{\ex}^{\Psi})^+,
\]
and the scale functions are both infinite on nonzero traces. Therefore the data $\Gamma_0,$ $\Gamma_1,$
$\tau^\Psi\mapsto \lambda\tau^\Theta$
give an isomorphism of the unified Elliott invariants. By the classification
theorem used in this section, $A_{\ex}^{\Theta}\cong A_{\ex}^{\Psi}.$

In particular, $A_{\ex}^{\Theta}\cong A_{\ex}^{\Psi}$ forces $\operatorname{rank}K_0(A_\Theta)=\operatorname{rank}K_0(A_\Psi),$
hence $2^{d-1}=2^{e-1},$
so $d=e$. Thus the dimension of the original noncommutative torus is remembered
by the stable pure-corner protorus.

Notice that in dimension $2$, this criterion becomes the usual stable classification of
irrational rotation algebras. Namely, for irrational parameters
$\theta,\eta$, one has $A_{\ex}^{J(\theta)}\cong A_{\ex}^{J(\eta)}$
if and only if there exist $\begin{pmatrix}
a&b\\
c&d
\end{pmatrix}
\in \GL_2(\Z)$
and $\lambda>0$ such that
$(1,\eta)
\begin{pmatrix}
a&b\\
c&d
\end{pmatrix}
=
\lambda(1,\theta).$
Equivalently, $
\theta=\frac{b+d\eta}{a+c\eta}$
with $a+c\eta>0$
after possibly multiplying the matrix by $-1$.
\end{example}

\begin{example}[A nonunital dimension-changing toric-corner model and its classification]
\label{ex:dimension-changing}
Let $\boldsymbol\theta=(\theta_j)_{j\geq1}$
be a sequence of irrational real numbers. 
For $n\geq1$, define $\Theta_n^{\boldsymbol\theta}:=J(\theta_1)\oplus\cdots\oplus J(\theta_n)
\in M_{2n}(\R), $ where $J(\alpha):=
\begin{pmatrix}
0&\alpha\\
-\alpha&0
\end{pmatrix}.$
Then each \(\Theta_n^{\boldsymbol\theta}\) is nondegenerate. Indeed, if $x\in\Z^{2n}$ satisfies $\exp(2\pi i\langle x,\Theta_n^{\boldsymbol\theta}y\rangle)=1$ for all $y\in\Z^{2n}$,
then testing against the two standard basis vectors in each $J(\theta_j)$-block
forces all coordinates of $x$ to be zero, since each $\theta_j$ is
irrational.

Put $B_n^{\boldsymbol\theta}:=M_{2^{n-1}}(A_{\Theta_n^{\boldsymbol\theta}}),$
and let $M_n\colon \Z^{2n}\hookrightarrow \Z^{2n+2}$
be the inclusion of the first $2n$ coordinates. Then
$M_n^t\Theta_{n+1}^{\boldsymbol\theta}M_n=\Theta_n^{\boldsymbol\theta}.$
Hence Proposition~\ref{prop:toric-explicit} gives a unital toric map $\varphi_n^{\boldsymbol\theta}\colon
A_{\Theta_n^{\boldsymbol\theta}}
\longrightarrow
A_{\Theta_{n+1}^{\boldsymbol\theta}}$
corresponding to $M_n$. Define $\phi_n^{\boldsymbol\theta}
\colon B_n^{\boldsymbol\theta}
\longrightarrow
B_{n+1}^{\boldsymbol\theta}$
by
\[
\phi_n^{\boldsymbol\theta}
=
\bigl(\text{upper-left corner inclusion}\bigr)
\circ
(\id_{M_{2^{n-1}}}\otimes \varphi_n^{\boldsymbol\theta}).
\]
Equivalently, after identifying $B_{n+1}^{\boldsymbol\theta}=M_2\bigl(M_{2^{n-1}}(A_{\Theta_{n+1}^{\boldsymbol\theta}})\bigr),$
one has
\[
\phi_n^{\boldsymbol\theta}(a)
=
\begin{pmatrix}
(\id_{M_{2^{n-1}}}\otimes\varphi_n^{\boldsymbol\theta})(a)&0\\
0&0
\end{pmatrix}.
\]
Thus $\phi_n^{\boldsymbol\theta}$ is a dimension-changing toric map followed
by a proper corner inclusion; it is a Case~\textup{(3)} map in the notation of
Proposition~\ref{prop:four-cases-invariants}. Define
\[
A_{\ex}^{\boldsymbol\theta}
:=
\varinjlim(B_n^{\boldsymbol\theta},\phi_n^{\boldsymbol\theta}).
\]
Then $A_{\ex}^{\boldsymbol\theta}$ is a simple nonunital protoral
$C^*$-algebra. Simplicity follows from
Proposition~\ref{prop:protorus-regularity}. To see that it is nonunital, let $e_n:=\iota_n(1_{B_n^{\boldsymbol\theta}})
\in A_{\ex}^{\boldsymbol\theta}.$ The projections $e_n$ form an increasing approximate unit. The trace computed
below satisfies $\tau^{\boldsymbol\theta}(e_n)=2^{n-1}$, so the projections
$e_n$ are not eventually constant. If $A_{\ex}^{\boldsymbol\theta}$ were
unital, then $e_n\to 1$ in norm; since the $e_n$ are projections, this
would force $e_n=1$ for all large $n$, a contradiction. Hence
$A_{\ex}^{\boldsymbol\theta}$ is nonunital.

We compute its $K$-theory. Under the standard Morita identifications $K_i(B_n^{\boldsymbol\theta})
\cong
K_i(A_{\Theta_n^{\boldsymbol\theta}}),$ the upper-left corner inclusion contributes the identity on $K$-theory, while
the toric part contributes $\Lambda^{\mathrm{even}}(M_n)$ on 
$K_0$ and 
$\Lambda^{\mathrm{odd}}(M_n)$ on $K_1$
by Proposition~\ref{prop:four-cases-invariants}. Therefore
\[K_0(A_{\ex}^{\boldsymbol\theta})
\cong
\varinjlim\bigl(\Lambda^{\mathrm{even}}\Z^{2n},
\Lambda^{\mathrm{even}}(M_n)\bigr)\quad\text{and }\quad K_1(A_{\ex}^{\boldsymbol\theta})
\cong
\varinjlim\bigl(\Lambda^{\mathrm{odd}}\Z^{2n},
\Lambda^{\mathrm{odd}}(M_n)\bigr).\]
Since $\varinjlim(\Z^{2n},M_n)
=
\Z^{(\infty)}
:=
\bigoplus_{j\geq1}\Z e_j,$
and exterior powers commute with filtered colimits, we obtain $K_0(A_{\ex}^{\boldsymbol\theta})
\cong
\Lambda^{\mathrm{even}}(\Z^{(\infty)})$ and $K_1(A_{\ex}^{\boldsymbol\theta})
\cong
\Lambda^{\mathrm{odd}}(\Z^{(\infty)}).$

Next consider traces. The normalized trace on $B_n^{\boldsymbol\theta}=M_{2^{n-1}}(A_{\Theta_n^{\boldsymbol\theta}})$
is $\tau_n^{\boldsymbol\theta}=
2^{-(n-1)}
\Tr_{2^{n-1}}\otimes\tau_{\Theta_n^{\boldsymbol\theta}},$
where $\Tr_{2^{n-1}}$ denotes the unnormalized matrix trace. Since
$\varphi_n^{\boldsymbol\theta}$ is unital and trace-preserving, and since the
upper-left corner has normalized trace $1/2$ inside
$B_{n+1}^{\boldsymbol\theta}$, we have $\tau_{n+1}^{\boldsymbol\theta}\circ\phi_n^{\boldsymbol\theta}
=
\frac12\,\tau_n^{\boldsymbol\theta}.$
Thus $t_n=1/2$ and $c_n=(t_1\cdots t_{n-1})^{-1}=2^{n-1}.$ 
By Theorem~\ref{thm:protorus-invariant}, $\widetilde T(A_{\ex}^{\boldsymbol\theta})
=
\R_+\tau^{\boldsymbol\theta},$
where the distinguished generator $\tau^{\boldsymbol\theta}$ is characterized
by $\tau^{\boldsymbol\theta}\circ\iota_n
=c_n\tau_n^{\boldsymbol\theta}=\Tr_{2^{n-1}}\otimes\tau_{\Theta_n^{\boldsymbol\theta}}.$
In particular, $\tau^{\boldsymbol\theta}(e_n)=c_n=2^{n-1}.$

Since $c_n\to\infty$, Theorem~\ref{thm:protorus-invariant} gives the full
projection scale $\Sigma(K_0(A_{\ex}^{\boldsymbol\theta}))=K_0(A_{\ex}^{\boldsymbol\theta})^+.$
The scale function is
\[
\Sigma_{A_{\ex}^{\boldsymbol\theta}}(\lambda\tau^{\boldsymbol\theta})
=
\begin{cases}
0, & \lambda=0,\\
+\infty, & \lambda>0.
\end{cases}
\]

We now compute the trace range. By Elliott's trace formula for
noncommutative tori
\cite[\S1.3, Theorems~2.2 and~3.1]{Ell84},
$\rho_{\Theta_n}(K_0(A_{\Theta_n}^{\boldsymbol\theta}))
=\sum_{I\in \mathrm{Minor}(2n)}\operatorname{pf}((\Theta_n^{\boldsymbol\theta})_I)\,\Z.$
For the block-diagonal matrix $\Theta_n^{\boldsymbol\theta}
=J(\theta_1)\oplus\cdots\oplus J(\theta_n),$ the only nonzero Pfaffian minors are the products
$\prod_{j\in F}\theta_j$ where $F\subseteq\{1,\ldots,n\}$,
with the empty product interpreted as \(1\). 
Therefore
\[
\rho_{\Theta_n^{\boldsymbol\theta}}
(K_0(A_{\Theta_n^{\boldsymbol\theta}}))
=
\left\langle
\prod_{j\in F}\theta_j:
F\subseteq\{1,\ldots,n\}
\right\rangle_{\Z}.
\]
Passing from $A_{\Theta_n^{\boldsymbol\theta}}$ to
$B_n^{\boldsymbol\theta}$ divides the trace range by $2^{n-1}$, because
$\tau_n^{\boldsymbol\theta}$ is the normalized matrix trace. But the
direct-limit trace satisfies $(\tau^{\boldsymbol\theta})_*\circ(\iota_n)_*=c_n(\tau_n^{\boldsymbol\theta})_*
=2^{n-1}(\tau_n^{\boldsymbol\theta})_*,$
so the factor $2^{n-1}$ exactly cancels the matrix normalization. Hence the
contribution from stage \(n\) is again
\[
\left\langle
\prod_{j\in F}\theta_j:
F\subseteq\{1,\ldots,n\}
\right\rangle_{\Z}.
\]
Taking the union over $n$, we get
\[
(\tau^{\boldsymbol\theta})_*
(K_0(A_{\ex}^{\boldsymbol\theta}))
=
\left\langle
\prod_{j\in F}\theta_j:
F\subset \N,\ |F|<\infty
\right\rangle_{\Z},
\]
with the empty product interpreted as  $1$.

Equivalently, under the identification $K_0(A_{\ex}^{\boldsymbol\theta})\cong\Lambda^{\mathrm{even}}(\Z^{(\infty)}),$
the trace map $\rho_{\boldsymbol\theta}^{\ex}:=(\tau^{\boldsymbol\theta})_*$
is determined as follows. Write the underlying free abelian group as
\[
\Z^{(\infty)}
=
\bigoplus_{j\geq1}(\Z x_j\oplus \Z y_j),
\]
where $x_j,y_j$ correspond to the $j$-th $J(\theta_j)$-block. Then $\rho_{\boldsymbol\theta}^{\ex}(1)=1,$
and, for $j_1<\cdots<j_k$, $\rho_{\boldsymbol\theta}^{\ex}
\bigl(
x_{j_1}\wedge y_{j_1}\wedge\cdots\wedge
x_{j_k}\wedge y_{j_k}
\bigr)
=
\prod_{\ell=1}^k\theta_{j_\ell}.$
The trace is zero on exterior basis monomials which are not, up to sign, wedges
of complete block pairs. Thus
\[
K_0(A_{\ex}^{\boldsymbol\theta})^+
=
\{x\in K_0(A_{\ex}^{\boldsymbol\theta}):
\rho_{\boldsymbol\theta}^{\ex}(x)>0\}\cup\{0\}.
\]

Finally, we record the Elliott-invariant classification of these examples. Let $\boldsymbol\eta=(\eta_j)_{j\geq1}$
be another sequence of irrational real numbers, and construct
$A_{\ex}^{\boldsymbol\eta}$ in the same way. Then $A_{\ex}^{\boldsymbol\theta}\cong A_{\ex}^{\boldsymbol\eta}$
if and only if there exist graded group isomorphisms
$\Gamma_0\colon
\Lambda^{\mathrm{even}}(\Z^{(\infty)})
\xrightarrow{\cong}
\Lambda^{\mathrm{even}}(\Z^{(\infty)}),$
and $\Gamma_1\colon
\Lambda^{\mathrm{odd}}(\Z^{(\infty)})
\xrightarrow{\cong}
\Lambda^{\mathrm{odd}}(\Z^{(\infty)}),$
together with a scalar $\lambda>0$, such that $\rho_{\boldsymbol\eta}^{\ex}\circ\Gamma_0=\lambda\,\rho_{\boldsymbol\theta}^{\ex}.$

Indeed, an isomorphism of $C^*$-algebras gives an isomorphism of the unified
Elliott invariants. Since both trace cones are one-dimensional, the induced
affine homeomorphism of trace cones sends
$\tau^{\boldsymbol\eta}$ to
$\lambda\tau^{\boldsymbol\theta}$
for some $\lambda>0$, and compatibility of the trace pairing gives the
displayed equation. Conversely, if such $\Gamma_0,\Gamma_1$, and $\lambda$
exist, then $\Gamma_0$ is an order isomorphism because the positive cones are
trace-determined. The projection scales are both full and the scale functions
are both infinite on every nonzero trace. Hence these data give an isomorphism
of the unified Elliott invariants, and the classification theorem used in this
section gives $A_{\ex}^{\boldsymbol\theta}\cong A_{\ex}^{\boldsymbol\eta}.$

In particular, an isomorphism $A_{\ex}^{\boldsymbol\theta}\cong A_{\ex}^{\boldsymbol\eta}$
forces the trace ranges to agree up to positive scalar:
\[
\left\langle
\prod_{j\in F}\eta_j:
F\subset\N,\ |F|<\infty
\right\rangle_{\Z}
=
\lambda
\left\langle
\prod_{j\in F}\theta_j:
F\subset\N,\ |F|<\infty
\right\rangle_{\Z}
\]
for some $\lambda>0$. Conversely, any block permutation of the sequence
$(\theta_j)$ gives an isomorphic protorus, because the corresponding
permutation of the pairs $(x_j,y_j)$ induces exterior-algebra isomorphisms
preserving the trace pairing.
\end{example}

\begin{example}[An explicit same-dimensional noncorner family and its classification]
\label{ex:AX7-noncorner}
Fix an integer $N\geq 2$ and an irrational number
$\theta_0\in(0,1)$. Define recursively $\theta_{n+1}:=\frac{\theta_n}{N+\theta_n},$ for 
$n\geq 0$, and put
$\eta_n:=\frac{\theta_n}{N}.$
Then $\eta_n=\frac{\theta_{n+1}}{1-\theta_{n+1}}.$ Each $\theta_n$ is irrational and lies in $(0,1)$.

Let $J(\alpha):=
\begin{pmatrix}
0&\alpha\\
-\alpha&0
\end{pmatrix}.$
For each $n\geq0$, consider the integer matrix $M_N:=
\begin{pmatrix}
1&0\\
0&N
\end{pmatrix}.$
Then $M_N^tJ(\eta_n)M_N=J(N\eta_n)=J(\theta_n).$
Hence Proposition~\ref{prop:toric-explicit} gives a unital toric map
$\psi_n\colon A_{\theta_n}\longrightarrow A_{\eta_n}$
determined on canonical generators by
\[
\psi_n(U_{\theta_n})=U_{\eta_n},
\qquad
\psi_n(V_{\theta_n})=V_{\eta_n}^{\,N}.
\]
Since $|\det(M_N)|=N>1,$
this toric step has proper range.

Next choose a projection $q_n\in A_{\theta_{n+1}}$
with $\tau_{\theta_{n+1}}(q_n)=1-\theta_{n+1}.$
Equivalently, under the standard identification $K_0(A_{\theta_{n+1}})\cong\Z^2$,
$\rho_{\theta_{n+1}}(a,b)=a+b\theta_{n+1},$
we may take $[q_n]=(1,-1),$
since $\rho_{\theta_{n+1}}(1,-1)=1-\theta_{n+1}>0.$

By the standard Rieffel corner realization for irrational rotation algebras,
using $\eta_n=\frac{\theta_{n+1}}{1-\theta_{n+1}},$
there is a unital isomorphism $\chi_n\colon A_{\eta_n}\xrightarrow{\cong}q_nA_{\theta_{n+1}}q_n.$
We choose $\chi_n$ so that, under the standard identifications
$K_0(A_{\eta_n})\cong\Z^2$ and $K_0(A_{\theta_{n+1}})\cong\Z^2,$ given by $\rho_{\eta_n}(a,b)=a+b\eta_n,$ respectively 
$\rho_{\theta_{n+1}}(a,b)=a+b\theta_{n+1},$
the induced map on $K_0$, after including the corner into
$A_{\theta_{n+1}}$, is
\[
(a,b)\longmapsto (a,-a+b).
\]
We also choose $\chi_n$ so that, under the standard Morita identification $K_1(q_nA_{\theta_{n+1}}q_n)\cong K_1(A_{\theta_{n+1}}),$
its induced map on $K_1$ is the identity. This choice is compatible with the Elliott classification theorem for unital
simple AT algebras of real rank zero, since the displayed $K_0$-map and the
chosen $K_1$-map give an isomorphism of Elliott invariants of
$A_{\eta_n}$ and $q_nA_{\theta_{n+1}}q_n$.

Define $\phi_n:=\bigl(q_nA_{\theta_{n+1}}q_n\hookrightarrow A_{\theta_{n+1}}\bigr)
\circ\chi_n\circ\psi_n
\colon
A_{\theta_n}\longrightarrow A_{\theta_{n+1}},$
and set
\[
A_{\ex}^{N,\theta_0}
:=
\varinjlim(A_{\theta_n},\phi_n).
\]
Then each $\phi_n$ is a same-dimensional nonunital noncorner embedding of
Case~\textup{(4)} in Proposition~\ref{prop:four-cases-invariants}. It is
nonunital because $\phi_n(1_{A_{\theta_n}})=q_n\neq 1_{A_{\theta_{n+1}}},$
and the range inside the corner is proper because the toric matrix has $|\det(M_N)|=N>1.$

We now compute the invariant. Identify $K_0(A_{\theta_n})\cong\Z^2$
by $\rho_{\theta_n}(a,b)=a+b\theta_n.$
The toric step $\psi_n$ induces $(\psi_n)_{*0}(a,b)=(a,Nb),$
because, on $K_0(A_{\theta_n})\cong\Lambda^{\mathrm{even}}\Z^2,$
the degree-two generator is multiplied by $\det(M_N)=N.$
The corner step contributes $(a,b)\longmapsto(a,-a+b)$
on $K_0$. Indeed, the normalized trace on the corner is
\[
x\longmapsto
\frac{\rho_{\theta_{n+1}}(x)}{1-\theta_{n+1}},
\]
and this agrees with $\rho_{\eta_n}$ since
$\rho_{\theta_{n+1}}(0,1)(1-\theta_{n+1})^{-1}
=
{\theta_{n+1}}(1-\theta_{n+1})^{-1}
=
\eta_n.$ Consequently $(\phi_n)_{*0}(a,b)=(a,-a+Nb).$

Set $s_0:=0$, $s_n:=\sum_{k=1}^nN^{-k},$ and $D_N:=
\begin{pmatrix}
1&0\\
-1&N
\end{pmatrix}.$
Thus
\[
K_0(A_{\ex}^{N,\theta_0})
\cong
\varinjlim
\left(
\Z^2\xrightarrow{\ D_N\ }\Z^2\xrightarrow{\ D_N\ }\Z^2
\xrightarrow{\ D_N\ }\cdots
\right).
\]
For each $n\geq0$, define
$F_n\colon\Z^2\longrightarrow \Z\oplus\Z[1/N]$
by $F_n(a,b):=
\left(a,\frac{b}{N^n}+a\,s_n\right).$
A direct computation gives $F_{n+1}\bigl(D_N(a,b)\bigr)=F_n(a,b).$ Therefore the maps $F_n$ induce a homomorphism $\varinjlim(\Z^2,D_N)\longrightarrow \Z\oplus\Z[1/N].$
This homomorphism is an isomorphism. Indeed, each $F_n$ is injective, and
every element $(m,q)\in\Z\oplus\Z[1/N]$ has the form $F_n(a,b)$ for all
sufficiently large $n$, by taking $a=m$ and  $b=N^n(q-ms_n)\in\Z.$ Hence $K_0(A_{\ex}^{N,\theta_0})\cong\Z\oplus\Z[1/N].$

For $K_1$, the toric step contributes $(a,b)\longmapsto(a,Nb),$
and our choice of $\chi_n$ makes the corner-identification step the identity
on $K_1$, after Morita identification. Hence
\[
K_1(A_{\ex}^{N,\theta_0})
\cong
\varinjlim
\left(
\Z^2,
\begin{pmatrix}
1&0\\
0&N
\end{pmatrix}
\right)
\cong
\Z\oplus\Z[1/N].
\]

We index this system from $n=0$. Let $e_n\in A_{\ex}^{N,\theta_0}$
denote the image of $1_{A_{\theta_n}}$. The trace-scaling constant is $t_n
=
\tau_{\theta_{n+1}}(\phi_n(1))
=
\tau_{\theta_{n+1}}(q_n)
=
1-\theta_{n+1}.$
Normalize the trace ray by setting $c_0:=1$, $c_{n+1}:=\frac{c_n}{t_n}.$
Using $\theta_{n+1}=\frac{\theta_n}{N+\theta_n},$
one checks inductively that $c_n=1+\theta_0s_n$ and $c_n\theta_n=\frac{\theta_0}{N^n}.$ For $(a,b)\in K_0(A_{\theta_n})\cong\Z^2$, the direct-limit trace is therefore
\[
c_n\rho_{\theta_n}(a,b)
=
c_n(a+b\theta_n)
=
a+\theta_0\left(\frac{b}{N^n}+a\,s_n\right).
\]
Thus, under the identification
$K_0(A_{\ex}^{N,\theta_0})\cong\Z\oplus\Z[1/N],$
the trace is $\rho_{N,\theta_0}(m,q)=m+\theta_0q.$
Since each $c_n$ is positive, one has
$(a,b)\in K_0(A_{\theta_n})^+
\,\Leftrightarrow\,
\rho_{\theta_n}(a,b)>0
\,\Leftrightarrow\,
\rho_{N,\theta_0}(F_n(a,b))>0.$
It follows that
\[
K_0(A_{\ex}^{N,\theta_0})^+
=
\{(m,q)\in\Z\oplus\Z[1/N]:m+\theta_0q>0\}\cup\{0\}.
\]

The trace cone is one-dimensional: $\widetilde T(A_{\ex}^{N,\theta_0})=\R_+\tau_{N,\theta_0},$
where $\tau_{N,\theta_0}$ is normalized by $\tau_{N,\theta_0}(e_0)=1.$
The class of the stage-$n$ unit is $[e_n]=F_n(1,0)=\left(1,\sum_{k=1}^nN^{-k}\right).$
By Theorem~\ref{thm:protorus-invariant}, the projection scale is
\[
\Sigma(K_0(A_{\ex}^{N,\theta_0}))
=
\bigcup_{n\geq0}[0,[e_n]].
\]
Moreover, $\rho_{N,\theta_0}([e_n])=1+\theta_0\sum_{k=1}^nN^{-k}
\nearrow
1+\frac{\theta_0}{N-1}.$
Since the order on $K_0(A_{\ex}^{N,\theta_0})$ is determined by
$\rho_{N,\theta_0}$, we obtain
\[
\Sigma(K_0(A_{\ex}^{N,\theta_0}))
=
\left\{
x\in K_0(A_{\ex}^{N,\theta_0})^+:
\rho_{N,\theta_0}(x)<1+\frac{\theta_0}{N-1}
\right\}.
\]
Equivalently, the scale function is $\Sigma_{A_{\ex}^{N,\theta_0}}(\lambda\tau_{N,\theta_0})=\lambda\left(1+\frac{\theta_0}{N-1}\right),$ for $\lambda\in[0,\infty).$
In particular, $A_{\ex}^{N,\theta_0}$ is nonunital: the projections
\((e_n)\) form an increasing approximate unit, but their \(K_0\)-classes $[e_n]=\left(1,\sum_{k=1}^nN^{-k}\right)$
are not eventually constant.

We now record the classification of these examples by the Elliott invariant.
Let $M\geq2$ and let $\theta_0'\in(0,1)\setminus\Q$, and construct
$A_{\ex}^{M,\theta_0'}$ in the same way. Put $R_N:=\Z[1/N],$
$R_M:=\Z[1/M],$ and $L_{N,\theta_0}:=1+\frac{\theta_0}{N-1},$
$L_{M,\theta_0'}:=1+\frac{\theta_0'}{M-1}.$
Also let $\mathcal P(N):=\{p\text{ prime}:p\mid N\}.$
Then
\[
A_{\ex}^{N,\theta_0}\cong A_{\ex}^{M,\theta_0'}
\]
if and only if the following conditions hold:
$\mathcal P(N)=\mathcal P(M),$
and, writing $R:=R_N=R_M,$
there exist $\varepsilon\in\{\pm1\}$, $v\in R$, $u\in R^\times$, and
$\lambda>0$,
such that
$\lambda=\varepsilon+\theta_0'v,$
$\theta_0'u=\lambda\theta_0$, and $L_{M,\theta_0'}=\lambda L_{N,\theta_0}.$

Indeed, suppose first that $A_{\ex}^{N,\theta_0}\cong A_{\ex}^{M,\theta_0'}.$
Then $K_1(A_{\ex}^{N,\theta_0})
\cong
K_1(A_{\ex}^{M,\theta_0'}),$
so $\Z\oplus R_N\cong \Z\oplus R_M.$
For a prime $p$, the subgroup $
\bigcap_{k\geq1}p^k(\Z\oplus R_N)$
is equal to $0\oplus R_N$ if $p\mid N$, and is zero if $p\nmid N$.
This is an isomorphism invariant of the abelian group. Hence $\mathcal P(N)=\mathcal P(M),$
and therefore $R_N=R_M=:R.$

The induced isomorphism on $K_0$ is a group automorphism $\alpha_0\colon \Z\oplus R\longrightarrow \Z\oplus R.
$ Since $\operatorname{Hom}(R,\Z)=0,$
every group automorphism of $\Z\oplus R$ has the form $\alpha_0(a,q)=(\varepsilon a,va+uq),$
where $\varepsilon\in\{\pm1\},$ $v\in R$, and $u\in R^\times$.

The trace cones are one-dimensional. Hence the affine homeomorphism of trace
cones sends $\tau_{M,\theta_0'}$ to $\lambda\tau_{N,\theta_0}$
for some $\lambda>0$. Compatibility of the trace pairing gives $\rho_{M,\theta_0'}(\alpha_0(a,q))=\lambda\rho_{N,\theta_0}(a,q)$
for every $(a,q)\in\Z\oplus R$. In other words, $\varepsilon a+\theta_0'(va+uq)=\lambda(a+\theta_0q).$
Comparing coefficients gives
$\lambda=\varepsilon+\theta_0'v$ and $
\theta_0'u=\lambda\theta_0.$
Compatibility of the scale functions gives $L_{M,\theta_0'}=\Sigma_{A_{\ex}^{M,\theta_0'}}(\tau_{M,\theta_0'})=
\Sigma_{A_{\ex}^{N,\theta_0}}(\lambda\tau_{N,\theta_0})=\lambda L_{N,\theta_0}.$

Conversely, suppose that the displayed conditions hold. Define $\alpha_0\colon \Z\oplus R\to\Z\oplus R$ by
$\alpha_0(a,q):=(\varepsilon a,va+uq).$
Then $\alpha_0$ is a group isomorphism and
$\rho_{M,\theta_0'}(\alpha_0(a,q))=\lambda\rho_{N,\theta_0}(a,q).$
Hence $\alpha_0$ is an order isomorphism. The equation
$L_{M,\theta_0'}=\lambda L_{N,\theta_0}$
says exactly that the scale functions are preserved by the trace-cone map $\tau_{M,\theta_0'}\longmapsto \lambda\tau_{N,\theta_0}.$
Equivalently, the projection scales are also preserved: $\alpha_0\bigl(\Sigma(K_0(A_{\ex}^{N,\theta_0}))\bigr)=\Sigma(K_0(A_{\ex}^{M,\theta_0'})).$

On $K_1$, choose any group isomorphism $\alpha_1\colon \Z\oplus R\longrightarrow \Z\oplus R.$
Thus the unified Elliott invariants of
$A_{\ex}^{N,\theta_0}$ and $A_{\ex}^{M,\theta_0'}$ are isomorphic. By the
classification theorem used in this section,
$A_{\ex}^{N,\theta_0}\cong A_{\ex}^{M,\theta_0'}.$
\end{example}

\begin{example}[A unital non-toric $K_1$-engine and its classification]
\label{ex:exotic-K1-engine}
Fix an irrational number $\theta\in(0,1)$, and let $\mathbf P=(P_n)_{n\geq1},$ $P_n\in M_2(\Z)$,
be any sequence of integer matrices. For every $n\geq1$, put
$B_n:=A_\theta$. We use the standard identifications $K_0(A_\theta)\cong\Z^2$, $\rho_\theta(a,b)=a+b\theta$, 
$[1_{A_\theta}]=(1,0)$,
and $K_1(A_\theta)\cong\Z^2.$
Define graded group homomorphisms $\kappa_{0,n}:=\id_{\Z^2}\colon K_0(A_\theta)\to K_0(A_\theta)$ and $\kappa_{1,n}:=P_n\colon K_1(A_\theta)\to K_1(A_\theta).$ Since $\kappa_{0,n}([1])=[1]$ and $\rho_\theta\circ\kappa_{0,n}=\rho_\theta$, Theorem~\ref{thm:unital-classification} gives a unital monomorphism $\phi_n\colon A_\theta\to A_\theta$
such that $(\phi_n)_{*0}=\kappa_{0,n}$ and 
$(\phi_n)_{*1}=P_n.$ Define
\[
A_{\ex}^{\theta,\mathbf P}
:=
\varinjlim(A_\theta,\phi_n).
\]
Then $A_{\ex}^{\theta,\mathbf P}$ is a simple unital noncommutative
protorus.

The invariant is immediate from Theorem~\ref{thm:protorus-invariant}. Since the
connecting maps are unital, the limit is unital and has a unique tracial state
$\tau_{\theta,\mathbf P}$. On $K_0$, all connecting maps are the identity,
so $K_0(A_{\ex}^{\theta,\mathbf P})\cong\Z^2$,
with order unit $[1_{A_{\ex}^{\theta,\mathbf P}}]=(1,0)$, 
positive cone $K_0(A_{\ex}^{\theta,\mathbf P})^+=\{(a,b)\in\Z^2:a+b\theta>0\}\cup\{0\}$,
and trace pairing $\rho_{\theta,\mathbf P}(a,b)
:=(\tau_{\theta,\mathbf P})_*(a,b)
=a+b\theta.$
On $K_1$, \[K_1(A_{\ex}^{\theta,\mathbf P})\cong
\varinjlim(\Z^2,P_n).\]

For example, if $P_n=0$ for every $n$, then $K_1(A_{\ex}^{\theta,\mathbf P})=0.$
If, for a fixed integer $N\geq2$, $P_n=
\begin{pmatrix}
N&0\\
0&0
\end{pmatrix}$
for every $n$, then $K_1(A_{\ex}^{\theta,\mathbf P})\cong \Z[1/N].$

We now record the classification of this family by the Elliott invariant. Let
$\eta\in(0,1)\setminus\Q$, let $\mathbf Q=(Q_n)_{n\geq1}$, $Q_n\in M_2(\Z)$,
and construct $A_{\ex}^{\eta,\mathbf Q}$ in the same way. Put $G_{\mathbf P}:=\varinjlim(\Z^2,P_n)$ and $G_{\mathbf Q}:=\varinjlim(\Z^2,Q_n).$
Then
\[
A_{\ex}^{\theta,\mathbf P}
\cong
A_{\ex}^{\eta,\mathbf Q}
\]
as unital $C^*$-algebras if and only if
$G_{\mathbf P}\cong G_{\mathbf Q}$
as abelian groups and there exist $r\in\Z$ and $\varepsilon\in\{\pm1\}$
such that $\theta=r+\varepsilon\eta.$
Since $\theta,\eta\in(0,1)$, this last condition is equivalently $\theta=\eta$ or $\theta=1-\eta$.

Indeed, an isomorphism $A_{\ex}^{\theta,\mathbf P}
\cong
A_{\ex}^{\eta,\mathbf Q}$
induces an isomorphism $G_{\mathbf P}\cong G_{\mathbf Q}$
on $K_1$. On $K_0$, it induces an order-unit preserving group automorphism $\alpha_0\colon\Z^2\to\Z^2$
satisfying $\rho_{\eta,\mathbf Q}\circ\alpha_0=\rho_{\theta,\mathbf P}.$
Every automorphism of $\Z^2$ preserving $(1,0)$ has the form $\alpha_0(a,b)=(a+rb,\varepsilon b),$ for some $r\in\Z$ and $\varepsilon\in\{\pm1\}.$
Trace compatibility gives $a+rb+\varepsilon\eta b=a+\theta b$
for all $a,b\in\Z$, hence $\theta=r+\varepsilon\eta.$

Conversely, if $G_{\mathbf P}\cong G_{\mathbf Q}$ and
$\theta=r+\varepsilon\eta$, then $\alpha_0(a,b):=(a+rb,\varepsilon b)$
is an order-unit preserving $K_0$-isomorphism compatible with traces, and any
chosen isomorphism $G_{\mathbf P}\to G_{\mathbf Q}$
gives the $K_1$-part. Thus the unital Elliott invariants are isomorphic, and
the classification theorem gives $A_{\ex}^{\theta,\mathbf P}
\cong
A_{\ex}^{\eta,\mathbf Q}.$

These systems are generally not produced by the four explicit toric-corner
classes. For instance, if $P_n=0$ for every $n$, then each connecting map
has zero $K_1$-map. This cannot occur for the toric-corner maps of
Cases~\textup{(1)}--\textup{(4)} in the simple full-rank setting: after the
Morita-identification part, those maps induce exterior-algebra maps coming from
full-rank integer matrices, and hence are injective after tensoring with
$\Q$ on $K_1$. Consequently, a direct limit built from such maps has
nonzero rational $K_1$-rank, whereas the present choice \(P_n=0\) gives $K_1(A_{\ex}^{\theta,\mathbf P})=0.$
\end{example}

\begin{example}[A four-dimensional infinitesimal-killing protorus and its classification]
\label{ex:infinitesimal-killing}
Choose irrational numbers $\alpha,\beta\in(0,1)$
such that $1,\alpha,\beta,\alpha\beta$
are linearly independent over $\Q$, and set $\Theta_{\alpha,\beta}:=J(\alpha)\oplus J(\beta)\in M_4(\R)$, where 
$J(t):=
\begin{pmatrix}
0&t\\
-t&0
\end{pmatrix}.$
Then $A_{\Theta_{\alpha,\beta}}$ is simple. Indeed, the irrationality of $\alpha$ and $\beta$ implies that the bicharacter associated to
$\Theta_{\alpha,\beta}$ is nondegenerate.

Under the exterior-algebra identification $K_0(A_{\Theta_{\alpha,\beta}})
\cong
\Lambda^{\mathrm{even}}\Z^4,$
write the standard even basis as
\[
1,\quad
e_{12},e_{13},e_{14},e_{23},e_{24},e_{34},\quad
e_{1234}.
\]
For the block-diagonal matrix $\Theta_{\alpha,\beta}=J(\alpha)\oplus J(\beta),$
Elliott's trace formula gives $\rho_{\Theta_{\alpha,\beta}}(1)=1,$
$\rho_{\Theta_{\alpha,\beta}}(e_{12})=\alpha,$ $\rho_{\Theta_{\alpha,\beta}}(e_{34})=\beta,$
$\rho_{\Theta_{\alpha,\beta}}(e_{1234})=\alpha\beta,$ and $\rho_{\Theta_{\alpha,\beta}}(e_{13})
=
\rho_{\Theta_{\alpha,\beta}}(e_{14})
=
\rho_{\Theta_{\alpha,\beta}}(e_{23})
=
\rho_{\Theta_{\alpha,\beta}}(e_{24})
=
0.$

Let $H_{\alpha,\beta}
:=
\langle 1,e_{12},e_{34},e_{1234}\rangle
\cong\Z^4$
and $L_{\alpha,\beta}
:=
\langle e_{13},e_{14},e_{23},e_{24}\rangle
\cong\Z^4.$
Then $K_0(A_{\Theta_{\alpha,\beta}})
=
H_{\alpha,\beta}\oplus L_{\alpha,\beta}$
and 
$L_{\alpha,\beta}\subseteq \ker\rho_{\Theta_{\alpha,\beta}}.$
Since $1,\alpha,\beta,\alpha\beta$ are linearly independent over $\Q$, the
trace is injective on $H_{\alpha,\beta}$, and $L_{\alpha,\beta}$ is exactly
the trace-zero summand in this decomposition.

Define $\kappa_0\colon K_0(A_{\Theta_{\alpha,\beta}})
\longrightarrow
K_0(A_{\Theta_{\alpha,\beta}})$
by $\kappa_0|_{H_{\alpha,\beta}}=\id_{H_{\alpha,\beta}}$ and
$\kappa_0|_{L_{\alpha,\beta}}=0.$
Define also $\kappa_1:=0\colon
K_1(A_{\Theta_{\alpha,\beta}})
\longrightarrow
K_1(A_{\Theta_{\alpha,\beta}}).$
Then $
\kappa_0([1])=[1]$ and $\rho_{\Theta_{\alpha,\beta}}\circ\kappa_0
=
\rho_{\Theta_{\alpha,\beta}}.$
Moreover, because the positive cone is determined by the unique trace,
$\kappa_0$ sends every nonzero positive $K_0$-class to a nonzero positive
$K_0$-class. Hence Theorem~\ref{thm:unital-classification} gives a unital monomorphism $\phi_{\alpha,\beta}
\colon
A_{\Theta_{\alpha,\beta}}
\longrightarrow
A_{\Theta_{\alpha,\beta}}$
such that $
(\phi_{\alpha,\beta})_{*0}=\kappa_0$ and
$(\phi_{\alpha,\beta})_{*1}=0.$ Define
\[
A_{\ex}^{\alpha,\beta}
:=
\varinjlim
\bigl(A_{\Theta_{\alpha,\beta}},\phi_{\alpha,\beta}\bigr).
\]
This is a simple unital protoral \(C^*\)-algebra. Since \(\kappa_0\) is an
idempotent projection with image \(H_{\alpha,\beta}\), we have
\[
K_0(A_{\ex}^{\alpha,\beta})
\cong
\varinjlim(K_0(A_{\Theta_{\alpha,\beta}}),\kappa_0)
\cong
H_{\alpha,\beta}
\cong
\Z^4.
\]
Using the ordered basis $1,\ e_{12},\ e_{34},\ e_{1234}$
of $H_{\alpha,\beta}$, the trace pairing is $\rho_{\alpha,\beta}^{\ex}(a,b,c,d)
=a+b\alpha+c\beta+d\alpha\beta.$
Thus
\[
K_0(A_{\ex}^{\alpha,\beta})^+
=
\{(a,b,c,d)\in\Z^4:
a+b\alpha+c\beta+d\alpha\beta>0\}\cup\{0\},
\]
and the order unit is
$[1_{A_{\ex}^{\alpha,\beta}}]=(1,0,0,0).$
Since the connecting map on \(K_1\) is zero, $K_1(A_{\ex}^{\alpha,\beta})=0$. The system is unital, so $T(A_{\ex}^{\alpha,\beta})=\{\tau_{\alpha,\beta}^{\ex}\},$
where $\tau_{\alpha,\beta}^{\ex}$ is the unique tracial state.

We now classify these examples by their Elliott invariants. Let $\alpha',\beta'\in(0,1)$
be irrational numbers such that $1,\alpha',\beta',\alpha'\beta'$
are linearly independent over $\Z$, and construct
$A_{\ex}^{\alpha',\beta'}$ in the same way. Then
\[
A_{\ex}^{\alpha,\beta}
\cong
A_{\ex}^{\alpha',\beta'}
\]
if and only if $\langle 1,\alpha,\beta,\alpha\beta\rangle_{\Z}
=
\langle 1,\alpha',\beta',\alpha'\beta'\rangle_{\Z}$
as subgroups of $\R$.

Indeed, the trace map identifies $K_0(A_{\ex}^{\alpha,\beta})$
order-isomorphically with the ordered subgroup $G_{\alpha,\beta}:=
\langle 1,\alpha,\beta,\alpha\beta\rangle_{\Z}\subset\R$
with order inherited from $\R$ and distinguished order unit $1$. Similarly, $K_0(A_{\ex}^{\alpha',\beta'})$
is identified with $G_{\alpha',\beta'}:=\langle 1,\alpha',\beta',\alpha'\beta'\rangle_{\Z}.$
Since both algebras have $K_1=0$ and unique tracial state, their unital
Elliott invariants are isomorphic exactly when
$G_{\alpha,\beta}=G_{\alpha',\beta'}$
as ordered subgroups of $\R$ carrying the same distinguished unit $1$.
By the classification theorem, this is equivalent to $A_{\ex}^{\alpha,\beta}
\cong
A_{\ex}^{\alpha',\beta'}.$

Equivalently, the same condition can be written in matrix form: there exists $G\in\GL_4(\Z)$
with $G(1,0,0,0)^t=(1,0,0,0)^t$
such that $(1,\alpha',\beta',\alpha'\beta')\,G=(1,\alpha,\beta,\alpha\beta).$

Finally, the connecting map $\phi_{\alpha,\beta}$ is not of any of
Cases~\textup{(1)}--\textup{(4)}. A toric-corner map has, up to the
$K$-theory isomorphism coming from the Morita/corner part, a $K$-theory
map of the form $\Lambda^*(M)$ for a full-column-rank integer matrix $M$.
After tensoring with $\mathbb Q$, such exterior maps are injective. The map
$\kappa_0$ above kills the rank-four subgroup
$L_{\alpha,\beta}\subset K_0(A_{\Theta_{\alpha,\beta}})$, and
$\kappa_1=0$. Hence this connecting map cannot arise from any of the
exterior-algebra maps appearing in Cases~\textup{(1)}, \textup{(3)}, or
\textup{(4)}, nor from a pure corner isomorphism in Case~\textup{(2)}.

\end{example}

\section{Spectral triples on noncommutative protori}
\label{sec:spectral-triples}

In this last section, we construct spectral triples on the protoral
$C^*$-algebras of Section~\ref{sec:protorus-invariants} by means of the
inductive-limit construction of spectral triples from \cite{FloricelGhorbanpour}. In the unital toric
case, namely Case~\textup{(1)} of Section~\ref{toric}, the connecting maps act
directly on Fourier labels, so one can use the 
inductive realization of \cite{FloricelGhorbanpour}  for compatible families of Dirac operators. In the
nonunital cases, namely Cases~\textup{(2)}, \textup{(3)}, and \textup{(4)}, the
connecting maps factor through full corners and are no longer unital. In that
regime we adapt the same Hilbert-space/operator construction using
trace-rescaled GNS maps, but the resulting objects are naturally locally
compact spectral triples rather than unital spectral triples.

\subsection{The unital case: Fourier-compatible spectral triples}
\label{subsec:unital-fourier-compatible}

We briefly recall the part of \cite{FloricelGhorbanpour} that we use. A
spectral triple $(A,H,D)$ on a unital $C^*$-algebra $A$ consists of a
faithful unital representation $\pi\colon A\to\mathcal B(H)$ and a
selfadjoint operator $D$ such that $D$ has compact resolvent and the
$*$-algebra \[A^\infty (D)
:=\{a\in A : [D,\pi(a)] \text{ is densely defined and extends to a bounded operator}\}\]
is dense in $A$.

An isometric morphism between two spectral triples $(A_1,H_1,D_1)$ and $(A_2,H_2,D_2)$, in the sense of
\cite[Definition~2.1]{FloricelGhorbanpour}, is a pair $(\phi,I)$, where
$\phi\colon A_1\to A_2$ is an injective unital $*$-homomorphism and
$I\colon H_1\to H_2$ is an isometry, such that $\phi(A_1^\infty (D_1))\subseteq A_2^\infty (D_2),$ 
$I\pi_1(a)=\pi_2(\phi(a))I,$ 
$I(\operatorname{Dom}(D_1))\subseteq \operatorname{Dom}(D_2)$, and $ID_1=D_2I$ on $\operatorname{Dom}(D_1)$. Given an inductive
system of spectral triples $\bigl\{(A_j,H_j,D_j),(\phi_{j,k},I_{j,k})\bigr\},$ 
 its inductive realization is the triple
\[ (A:=\varinjlim A_j,H:=\varinjlim H_j,D:=\varinjlim D_j),\] 
obtained by taking the Hilbert-space inductive limit and then closing the densely
defined operator $D(I_j\xi_j):=I_jD_j\xi_j$
on the core $\bigcup_j I_j(\operatorname{Dom}(D_j))$; see
\cite[Definition~2.3]{FloricelGhorbanpour}. Theorem~3.1 of \cite{FloricelGhorbanpour} gives equivalent
criteria for the compactness of the resolvent of the limit operator,  while
Theorem~3.2 and Corollary~3.2.1 show that the density of $A^\infty(D)$ in $A$ follows from uniform boundedness of the stagewise commutators $\bigl\{[D_k,\pi_k(\phi_{j,k}(a))]\bigr\}_{k\geq j}$
for $a\in A_j^\infty$.

Let now $\Theta\in M_m(\R)$ be skew-symmetric. Write
\[
A_\Theta^\infty
=
\left\{
\sum_{x\in\Z^m} a_xU^x:
(a_x)_{x\in\Z^m}\text{ is Schwartz on }\Z^m
\right\}
\]
for the smooth subalgebra for the canonical $\T^m$-action, and write $A_\Theta^{\mathrm{alg}}:=\operatorname{span}\{U^x:x\in\Z^m\}$
for the algebraic Fourier subalgebra. Under the twisted-group picture,
$A_\Theta\cong C^*(\Z^m,\sigma_\Theta)$. Since $\Z^m$ is amenable, the
left regular representation is faithful. The GNS representation associated to
the canonical trace $\tau_\Theta$ identifies $L^2(A_\Theta,\tau_\Theta)\cong \ell^2(\Z^m),$
and we denote by $\delta_x$ the basis vector corresponding to the Fourier
monomial $U^x$.

For $1\leq j\leq m$, let $\delta_j^\Theta$ be the canonical smooth
$*$-derivation $\delta_j^\Theta(U^x)=2\pi i\,x_jU^x,$ for
$x=(x_1,\ldots,x_m)\in\Z^m.$
As an unbounded operator on $L^2(A_\Theta,\tau_\Theta)$, the closure of
$\delta_j^\Theta$ is skew-adjoint.  We put $P_j^\Theta:=-i\delta_j^\Theta,$
so that $
P_j^\Theta(U^x)=2\pi x_jU^x.$

Fix a finite-dimensional Hilbert space \(S\).
Let $F\colon \Z^m\longrightarrow \operatorname{End}(S)_{\mathrm{sa}}$
be a selfadjoint matrix-valued function, and define  Fourier-multiplier operator
\[
D_F(\delta_x\otimes\xi)=\delta_x\otimes F(x)\xi,
\qquad x\in\Z^m,\quad \xi\in S,
\]
on $H_\Theta:=L^2(A_\Theta,\tau_\Theta)\otimes S$. 

The following proposition is a twisted
$\mathbb Z^m$-version of the standard construction of spectral triples from
proper translation-bounded length functions on discrete groups.

\begin{proposition}[Fourier-multiplier triples on a single torus]
\label{prop:single-torus-fourier-multiplier}
Assume that $F$ satisfies the following two conditions:

\begin{enumerate}
\item[\emph{(a)}]
\emph{ [finite spectral multiplicity in bounded intervals]:} for every
$R>0$, the vector space
\[
\bigoplus_{\substack{x\in\Z^m\\
\operatorname{Spec}(F(x))\cap[-R,R]\neq\varnothing}}
S
\]
is finite-dimensional;

\item[\emph{(b)}]
\emph{[bounded translation increments]:} for every $a\in\Z^m$, $\sup_{x\in\Z^m}\|F(x+a)-F(x)\|<\infty.$
\end{enumerate}
\smallskip

\noindent
Then $(A_\Theta,H_\Theta,D_F)$
is a spectral triple. 
\end{proposition}
\begin{proof}
Via the Fourier basis, we identify $H_\Theta=L^2(A_\Theta,\tau_\Theta)\otimes S
\cong
\bigoplus_{x\in\Z^m} (\delta_x\otimes S).$
With respect to this orthogonal decomposition, the operator $D_F$ is the
diagonal direct sum
\[
D_F=\bigoplus_{x\in\Z^m} F(x),
\]
with domain $\dom(D_F)=
\left\{
(\xi_x)_{x\in\Z^m}\in \bigoplus_{x\in\Z^m}S :
\sum_{x\in\Z^m}\|F(x)\xi_x\|^2<\infty
\right\}.$

Since each $F(x)$ is selfadjoint on the finite-dimensional space $S$, the
operator $D_F$ is selfadjoint. Indeed, for $\eta=(\eta_x)_{x\in\Z^m}\in
H_\Theta$, define $\xi_x:=(F(x)+i)^{-1}\eta_x.$
Because \(F(x)\) is selfadjoint, one has
$\|(F(x)+i)^{-1}\|\leq 1$ and
$\|F(x)(F(x)+i)^{-1}\|\leq 1.$ Hence $\sum_x \|\xi_x\|^2\leq \sum_x \|\eta_x\|^2$ and $\sum_x \|F(x)\xi_x\|^2\leq \sum_x \|\eta_x\|^2.$
Thus \(\xi\in\dom(D_F)\), and $(D_F+i)\xi=\eta.$
So $D_F+i$ is surjective; similarly $D_F-i$ is surjective. Therefore
$D_F$ is selfadjoint.

We now prove compact resolvent. For $R>0$, let
\[
X_R:=\{x\in\Z^m:\operatorname{Spec}(F(x))\cap[-R,R]\neq\varnothing\}.
\]
Because $S$ is finite-dimensional, condition \textup{(a)} is equivalent to the
finiteness of $X_R$. Let $E_R:=1_{[-R,R]}(D_F)$
be the spectral projection of $D_F$ for the interval $[-R,R]$. Since $D_F$
is diagonal, one has $E_R
=
\bigoplus_{x\in\Z^m} 1_{[-R,R]}(F(x)),$
where $1_{[-R,R]}(F(x))$ is the spectral projection of the finite-dimensional
selfadjoint operator $F(x)$. If $x\notin X_R$, then $1_{[-R,R]}(F(x))=0.$
Hence
\[
E_R
=
\bigoplus_{x\in X_R} 1_{[-R,R]}(F(x)).
\]
Since $X_R$ is finite and each summand acts on the finite-dimensional space
$S$, the projection $E_R$ has finite rank. This implies compact resolvent. Indeed,
\[
(1+D_F^2)^{-1/2}
=
(1+D_F^2)^{-1/2}E_R
+
(1+D_F^2)^{-1/2}(1-E_R).
\]
The first term is finite-rank because $E_R$ is finite-rank. On the orthogonal
complement of $E_R$, the spectrum of $|D_F|$ is contained in $(R,\infty)$,
so
\[
\|(1+D_F^2)^{-1/2}(1-E_R)\|
\leq
(1+R^2)^{-1/2}.
\]
As $R\to\infty$, this tends to $0$. Therefore $(1+D_F^2)^{-1/2}$ is the
norm limit of finite-rank operators, hence is compact.

Conversely, if $(1+D_F^2)^{-1/2}$ is compact, then for every $R>0$,
$ E_R=1_{[(1+R^2)^{-1/2},\,1]}\bigl((1+D_F^2)^{-1/2}\bigr).$
Since $(1+D_F^2)^{-1/2}$ is a compact selfadjoint operator, its spectral
projection for any interval of the form $[\varepsilon,1]$ with $\varepsilon>0$
has finite rank. Hence $E_R$ is finite-rank. Thus the two formulations are
equivalent.

It remains to verify bounded commutators. Let $\pi_\Theta$ be the left regular
representation of $A_\Theta$ on $H_\Theta$. For $a\in\Z^m$, left
multiplication by $U^a$ is a twisted shift: \[\pi_\Theta(U^a)(\delta_x\otimes\xi)
=
\sigma_\Theta(a,x)\,\delta_{a+x}\otimes\xi,\] where
$\sigma_\Theta(a,x)\in\T$. Hence $[D_F,\pi_\Theta(U^a)](\delta_x\otimes\xi)
=\sigma_\Theta(a,x)\,\delta_{a+x}\otimes \bigl(F(a+x)-F(x)\bigr)\xi.$
Therefore
\[
\|[D_F,\pi_\Theta(U^a)](\delta_x\otimes\xi)\|
\leq
\|F(a+x)-F(x)\|\,\|\xi\|.
\]
By condition \textup{(b)}, $M_a:=\sup_{x\in\Z^m}\|F(x+a)-F(x)\|<\infty,$
so $\|[D_F,\pi_\Theta(U^a)]\|\leq M_a.$
Thus $[D_F,\pi_\Theta(U^a)]$ extends boundedly for every $a\in\Z^m$. By
linearity, the same is true for every $b\in A_\Theta^{\mathrm{alg}}=\operatorname{span}\{U^a:a\in\Z^m\}.$
Hence
\[
A_\Theta^{\mathrm{alg}}\subseteq A_\Theta^\infty(D_F).
\]
Since $A_\Theta^{\mathrm{alg}}$ is dense in $A_\Theta$, the algebra
$A_\Theta^\infty(D_F)$ is dense in $A_\Theta$. Therefore $(A_\Theta,H_\Theta,D_F)$
is a spectral triple.
\end{proof}

\begin{observation}[Standard families in the unital regime]
\label{rem:unital-special-spectral-triples}

We highlight the scope of Proposition~\ref{prop:single-torus-fourier-multiplier} through a collection of classical examples and related constructions. Items \textup{(1)} and \textup{(2)} below give the fundamental examples, while Items \textup{(3)} and \textup{(4)} present further constructions derived from spectral triples arising from Proposition~\ref{prop:single-torus-fourier-multiplier}. In general, these latter two constructions are no longer diagonal Fourier multipliers, so conditions \textup{(a)} and \textup{(b)} cease to be the relevant assumptions. We describe the precise mechanism involved in each case.

\smallskip

\noindent
 
 (1) [\emph{Flat triples.}] Take $S=S_m$ to be a complex Clifford module with selfadjoint Clifford
generators $\gamma_1^{(m)},\ldots,\gamma_m^{(m)}$, and define $F_L(x):=2\pi\,\gamma^{(m)}(Lx),$ for some $L\in \GL(m,\R)$,
where $\gamma^{(m)}(v):=\sum_{a=1}^m v_a\gamma_a^{(m)}.$
Then
\[
D_{\Theta,L}
=\sum_{a=1}^m
\left(\sum_{j=1}^m L_{aj}P_j^\Theta\right)\otimes \gamma_a^{(m)}
\]
is the standard flat translation-invariant Dirac operator associated to the cotangent norm
$v\mapsto \|Lv\|$.

To check condition \textup{(a)} of Proposition~\ref{prop:single-torus-fourier-multiplier}, note that Clifford multiplication satisfies
$\gamma^{(m)}(v)^2=\|v\|^2\,1_S,$
hence $\operatorname{Spec}(F_L(x))\subseteq \{\pm 2\pi\|Lx\|\}.$
Therefore $\operatorname{Spec}(F_L(x))\cap[-R,R]\neq\varnothing$ if and only if
$\|Lx\|\leq \frac{R}{2\pi}.$

Since $L$ is invertible, there exists $c_L>0$ such that $\|Lx\|\geq c_L\|x\|$ for all
$x\in\Z^m$,
so $\|Lx\|\leq \frac{R}{2\pi}$ implies $\|x\|\leq \frac{R}{2\pi c_L}.$
Hence only finitely many $x\in\Z^m$ contribute, and \textup{(a)} holds.

To check condition \textup{(b)}, observe that $F_L(x+a)-F_L(x)=2\pi\,\gamma^{(m)}(La),$
which is independent of $x$. Thus
\[
\sup_{x\in\Z^m}\|F_L(x+a)-F_L(x)\|
=
2\pi\,\|\gamma^{(m)}(La)\|
=
2\pi\,\|La\|
<
\infty.
\]
So \textup{(b)} also holds.

These are the usual flat triples on smooth noncommutative tori; see
\cite[Section~2]{CDV1} and \cite[Section~3]{CDV2}.

\smallskip

\noindent

(2) [\emph{Length-type triples.}]
Take $S=\C$ and $F(x)=\ell(x)$, where $\ell\colon\Z^m\to[0,\infty)$
is a proper length function with bounded translation increments. Then
\[
D_\ell\delta_x=\ell(x)\delta_x.
\]
This is the twisted-group-algebra spectral triple associated to the
presentation $A_\Theta\cong C^*(\Z^m,\sigma_\Theta).$

Condition \textup{(a)} is exactly properness of $\ell$: since $S=\C$,
\[
\bigoplus_{\substack{x\in\Z^m\\
\operatorname{Spec}(F(x))\cap[-R,R]\neq\varnothing}}\C
=
\bigoplus_{\substack{x\in\Z^m\\
\ell(x)\le R}}\C,
\]
which is finite-dimensional if and only if the set $\{x\in\Z^m:\ell(x)\le R\}$
is finite.

Condition \textup{(b)} is the bounded translation increment hypothesis:
$\sup_{x\in\Z^m}|\ell(x+a)-\ell(x)|<\infty$ for all
$a\in\Z^m$.
In particular, if $\ell$ is a genuine group length in the usual sense, then
$|\ell(x+a)-\ell(x)|\le \ell(a),$
so \textup{(b)} is automatic.

For length-function Dirac operators on group $C^*$-algebras and their
twisted analogues, including noncommutative tori, see
\cite[Introduction and Section~1]{RieffelGroupCQMS}; see also
\cite{AntoniniGuidoIsolaRubin} for matrix-valued length-function triples.

\smallskip

\noindent

(3) [\emph{Bounded perturbations and inner fluctuations.}]
Let $(A_\Theta,H_\Theta,D_F)$ be a spectral triple obtained from
Proposition~\ref{prop:single-torus-fourier-multiplier}, and let $B=B^*\in\mathcal B(H_\Theta)$
be bounded. Then $D_F+B$, with domain \(\dom(D_F)\), is again selfadjoint by
the Kato--Rellich theorem. Since $B$ is bounded, $D_F+B$ has compact
resolvent whenever $D_F$ does. Moreover, for every $a\in A_\Theta^\infty(D_F),$
one has $[D_F+B,\pi_\Theta(a)]=[D_F,\pi_\Theta(a)]+[B,\pi_\Theta(a)],$
and the second term is bounded because $B$ and $\pi_\Theta(a)$ are bounded.
Hence $A_\Theta^\infty(D_F+B)=A_\Theta^\infty(D_F).$

In particular, one may consider inner fluctuations. The Connes one-forms
associated to $(A_\Theta,H_\Theta,D_F)$ are
\[
\Omega^1_{D_F}(A_\Theta^\infty(D_F))
:=
\left\{
\sum_j \pi_\Theta(a_j)[D_F,\pi_\Theta(b_j)]:
a_j,b_j\in A_\Theta^\infty(D_F)
\right\}
\subseteq \mathcal B(H_\Theta). \]
If $A=A^*\in \Omega^1_{D_F}(A_\Theta^\infty(D_F))$, then the fluctuated Dirac
operator
\[
D_F^A:=D_F+A
\]
is again a spectral triple on $A_\Theta$. This is the standard mechanism of
inner fluctuations in noncommutative geometry; see
\cite[Chapter~VI]{ConnesBook}.

If one wants to remain literally inside the class of Fourier multipliers from
Proposition~\ref{prop:single-torus-fourier-multiplier}, one may restrict to
diagonal bounded perturbations: let $G\colon \Z^m\to \operatorname{End}(S)_{\mathrm{sa}}$
be bounded, $\sup_{x\in\Z^m}\|G(x)\|<\infty,$
and suppose that $G$ has bounded translation increments. Then
\[
(F+G)(x):=F(x)+G(x)
\]
again satisfies \textup{(a)} and \textup{(b)}. Indeed, boundedness of $G$
implies that for every $R>0$, $\operatorname{Spec}(F(x)+G(x))\cap[-R,R]\neq\varnothing$
can occur only when
\[
\operatorname{Spec}(F(x))\cap[-(R+\|G\|_\infty),\,R+\|G\|_\infty]\neq\varnothing,
\]
so \textup{(a)} follows from the corresponding property for $F$. Condition
\textup{(b)} follows from
\[
\|(F+G)(x+a)-(F+G)(x)\|
\le
\|F(x+a)-F(x)\|+\|G(x+a)-G(x)\|.
\]

\smallskip

\noindent
(4) [\emph{Conformal or curved triples.}]
Conformal or curved Dirac operators are generally not direct Fourier
multipliers, so conditions \textup{(a)} and \textup{(b)} of
Proposition~\ref{prop:single-torus-fourier-multiplier} do not literally apply
to them. Rather, they should be viewed as deformations of spectral triples
already constructed by Fourier-multiplier or flat methods.

A standard conformal deformation starts with a positive invertible element $k=e^{h/2}\in A_\Theta^\infty$
and a flat or Fourier-multiplier spectral triple $(A_\Theta,H_\Theta,D_F).$
In the usual left-regular representation, the operator formally has the shape $D_{F,k}=\pi_\Theta(k)D_F\pi_\Theta(k)$
on the domain $\operatorname{Dom}(D_{F,k})=\pi_\Theta(k)^{-1}\operatorname{Dom}(D_F).$
Since $\pi_\Theta(k)$ is bounded, positive, and invertible, the operator
$D_{F,k}$ is selfadjoint whenever $D_F$ is selfadjoint.  Moreover, compact resolvent is
preserved, because the graph norm of $D_{F,k}$ is equivalent, via the bounded
invertible operator $\pi_\Theta(k)$, to the graph norm of $D_F$. However, because $k$ is generally noncentral, $\pi_\Theta(k)$ does not
commute with $\pi_\Theta(a)$. Thus bounded commutators for $D_{F,k}$ do not
follow formally from bounded commutators for $D_F$. For the usual flat Dirac
operators, they are obtained by the standard conformal pseudodifferential
analysis of noncommutative tori.

There is, however, a simple commutant-side variant which is especially well
adapted to inductive limits. Let $R_k\otimes 1_S$ denote right multiplication by $k$ on
$L^2(A_\Theta,\tau_\Theta)\otimes S$. Since $R_k$ belongs to the commutant
of the left regular representation, it commutes with $\pi_\Theta(a)$ for all
$a\in A_\Theta$. For
\[
D_{F,k}^{\mathrm{right}}:=(R_k\otimes 1_S)D_F(R_k\otimes 1_S),
\qquad
\operatorname{Dom}(D_{F,k}^{\mathrm{right}})
=(R_k\otimes 1_S)^{-1}\operatorname{Dom}(D_F),
\]
one has, first on the natural core and then by closure,
\[
[D_{F,k}^{\mathrm{right}},\pi_\Theta(a)]
=
(R_k\otimes 1_S)[D_F,\pi_\Theta(a)](R_k\otimes 1_S).
\]
Thus bounded commutators are inherited directly from the undeformed triple. Selfadjointness and compact resolvent also follow from the bounded invertible
graph-norm equivalence.

In dimension two, left conformal deformations are the starting point of the
conformal spectral geometry of the noncommutative two-torus; see
\cite{ConnesTretkoff,FathizadehKhalkhaliScalar}. In higher dimensions, related
Dirac-type operators arise from complex structures and Hermitian metrics on
higher-dimensional noncommutative tori. In particular, the associated
Dolbeault-type operator $\bar\partial+\bar\partial^*$
on $(0,\bullet)$-forms appears in the complex-geometric setting of
\cite{MathaiRosenberg2}.

If one wishes to remain strictly inside the diagonal Fourier-multiplier
framework of Proposition~\ref{prop:single-torus-fourier-multiplier}, then the
curved or conformal symbol itself must remain diagonal in the Fourier basis. In
that special case one simply checks conditions \textup{(a)} and \textup{(b)}
for the new symbol. In general, however, conformal and curved Dirac operators
should be regarded as standard deformations of the basic Fourier-multiplier
triples rather than as direct instances of the proposition.

\end{observation}

\begin{lemma}
\label{lem:toric-fourier-smooth-derivations}
Let $\Theta\in M_m(\R)$ and $\Psi\in M_n(\R)$ be nondegenerate
skew-symmetric matrices, and let $\varphi_{M,z}\colon A_\Theta\to A_\Psi$
be a unital toric map, with $M\in M_{n\times m}(\Z)$, and 
$z=(z_1,\ldots,z_m)\in\T^m.$
Let $\alpha^\Theta\colon \T^m\curvearrowright A_\Theta$ and $
\alpha^\Psi\colon \T^n\curvearrowright A_\Psi$
denote the canonical torus actions, and define the continuous group
homomorphism $\beta_M\colon \T^n\to \T^m,$\[
\beta_M(t)_j:=\prod_{r=1}^n t_r^{\,M_{rj}},
\qquad
1\le j\le m.
\]
Then: \emph{(1)}  For every $t\in\T^n$,
$\alpha_t^\Psi\circ\varphi_{M,z}=\varphi_{M,z}\circ \alpha_{\beta_M(t)}^\Theta.$
Consequently, $\varphi_{M,z}(A_\Theta^\infty)\subseteq A_\Psi^\infty.$
\smallskip

\noindent
 \emph{(2)}  For each $x\in\Z^m$, there exists a scalar $\omega_{M,z}(x)\in\T$
such that $\varphi_{M,z}(U^x)=\omega_{M,z}(x)\,V^{Mx}.$
Consequently, the GNS isometry
$I_\varphi\colon L^2(A_\Theta,\tau_\Theta)\to L^2(A_\Psi,\tau_\Psi),$
 $I_\varphi\Lambda_\Theta(a):=\Lambda_\Psi(\varphi_{M,z}(a)),$
sends the Fourier basis vector indexed by $x$ to the Fourier basis vector
indexed by $Mx$, up to the phase $\omega_{M,z}(x)$:
\[
I_\varphi(\delta_x^\Theta)=\omega_{M,z}(x)\,\delta_{Mx}^\Psi.
\]
\smallskip

\noindent
 \emph{(3)} For every $1\le r\le n$, we have
\begin{equation}\label{eq:toric-derivation-formula-merged}
\delta_r^\Psi\circ\varphi_{M,z}
=
\sum_{j=1}^m M_{rj}\,\varphi_{M,z}\circ\delta_j^\Theta
\qquad\text{on }A_\Theta^\infty.
\end{equation}
Equivalently, with $P_j=-i\delta_j$, $P_r^\Psi\circ\varphi_{M,z}=\sum_{j=1}^m M_{rj}\,\varphi_{M,z}\circ P_j^\Theta.$

If, in addition, one works in the twisted-group basis for
$A_{M^t\Psi M}$ (equivalently, after identifying $A_\Theta$ with
$A_{M^t\Psi M}$ via the generator identification used elsewhere), then the
phase in \textup{(2)} is exactly $z^x:=z_1^{x_1}\cdots z_m^{x_m}.$
\end{lemma}

\begin{proof}
(1) For $1\le j\le m$, the toric map is defined by $\varphi_{M,z}(U_j)=z_jV^{Me_j}.$
Let $t=(t_1,\ldots,t_n)\in\T^n$. Then
\[
\alpha_t^\Psi(\varphi_{M,z}(U_j))
=
\alpha_t^\Psi(z_jV^{Me_j})
=
z_j\Bigl(\prod_{r=1}^n t_r^{M_{rj}}\Bigr)V^{Me_j}
=
\varphi_{M,z}\bigl(\beta_M(t)_jU_j\bigr)
=
\varphi_{M,z}\bigl(\alpha_{\beta_M(t)}^\Theta(U_j)\bigr).
\]
Since the generators $U_1,\ldots,U_m$ generate $A_\Theta$, this proves $\alpha_t^\Psi\circ\varphi_{M,z}
=\varphi_{M,z}\circ \alpha_{\beta_M(t)}^\Theta.$

Now let $a\in A_\Theta^\infty$.  By definition of the smooth subalgebra, the orbit map $s\longmapsto \alpha_s^\Theta(a)$
is smooth on $\T^m$. Since $\beta_M\colon \T^n\to\T^m$ is a smooth group
homomorphism and $\varphi_{M,z}$ is continuous, the map
\[
t\longmapsto \alpha_t^\Psi(\varphi_{M,z}(a))
=
\varphi_{M,z}(\alpha_{\beta_M(t)}^\Theta(a))
\]
is smooth on $\T^n$. Hence $\varphi_{M,z}(a)\in A_\Psi^\infty$, proving
$\varphi_{M,z}(A_\Theta^\infty)\subseteq A_\Psi^\infty.$

(2) For $x=(x_1,\ldots,x_m)\in\Z^m$, we have $U^x=U_1^{x_1}\cdots U_m^{x_m},$
so
$\varphi_{M,z}(U^x)=\prod_{j=1}^m \bigl(z_jV^{Me_j}\bigr)^{x_j}.$ Write $a_j:=Me_j\in\Z^n.$
For ordered monomials in $A_\Psi$, one has $V^aV^b=c_\Psi(a,b)\,V^{a+b},$ for all 
$ a,b\in\Z^n,$
for a scalar $c_\Psi(a,b)\in\T$. Explicitly,
\[
c_\Psi(a,b)
=
\exp\!\Bigl(2\pi i\sum_{1\le p<q\le n} a_qb_p\,\Psi_{p,q}\Bigr).
\]
Thus the product of two ordered monomials is again an ordered monomial, up to
a scalar. Repeating this identity inductively shows that the product $\prod_{j=1}^m (V^{a_j})^{x_j}$
is a scalar multiple of the ordered monomial whose exponent is the sum of the
exponents:
\[
\sum_{j=1}^m x_j a_j
=
\sum_{j=1}^m x_jMe_j
=
Mx.
\]
Therefore there exists $\eta_M(x)\in\T$ such that $\prod_{j=1}^m (V^{Me_j})^{x_j}=\eta_M(x)\,V^{Mx}.$
Multiplying by the scalar factor $z^x:=z_1^{x_1}\cdots z_m^{x_m}$, we obtain
\[
\varphi_{M,z}(U^x)=\omega_{M,z}(x)\,V^{Mx}
\]
for some $\omega_{M,z}(x)\in\T$.

Since $\varphi_{M,z}$ is unital and both algebras have unique tracial states, we have
$\tau_\Psi\circ\varphi_{M,z}=\tau_\Theta.$
Hence the GNS map $I_\varphi\Lambda_\Theta(a)=\Lambda_\Psi(\varphi_{M,z}(a))$
is an isometry. Applying this to $a=U^x$ gives
$I_\varphi(\delta_x^\Theta)
=\Lambda_\Psi(\varphi_{M,z}(U^x))
=\omega_{M,z}(x)\Lambda_\Psi(V^{Mx})
=\omega_{M,z}(x)\delta_{Mx}^\Psi.$

(3)  For $1\le r\le n$, define the
one-parameter subgroup
\[
e_r(s):=(1,\ldots,1,e^{2\pi is},1,\ldots,1)\in\T^n,
\qquad s\in\R,
\]
where $e^{2\pi is}$ appears in the $r$-th coordinate. By definition of the
canonical derivations,
\[
\delta_r^\Psi(b)
=
\left.\frac{d}{ds}\right|_{s=0}\alpha_{e_r(s)}^\Psi(b),
\qquad
b\in A_\Psi^\infty.
\]
Moreover, $\beta_M(e_r(s))=\bigl(e^{2\pi iM_{r1}s},\ldots,e^{2\pi iM_{rm}s}\bigr)\in\T^m.$
For a smooth element $a\in A_\Theta^\infty$, consider the smooth orbit map
$f_a(t):=\alpha_t^\Theta(a),$ for all $t\in\T^m.$
Then, by the chain rule,
\[
\left.\frac{d}{ds}\right|_{s=0}
f_a\bigl(\beta_M(e_r(s))\bigr)
=
\sum_{j=1}^m M_{rj}\,\delta_j^\Theta(a).
\]
Indeed, this identity is immediate on Fourier monomials:
$
\alpha_{\beta_M(e_r(s))}^\Theta(U^x)
=
e^{2\pi i s\sum_{j=1}^mM_{rj}x_j}\,U^x,$ whose derivative at \(s=0\) is $2\pi i\Bigl(\sum_{j=1}^mM_{rj}x_j\Bigr)U^x
=
\sum_{j=1}^mM_{rj}\,\delta_j^\Theta(U^x).$
Since $A_\Theta^\infty$ is the space of smooth vectors for the torus action,
the same formula holds for all $a\in A_\Theta^\infty$.

Now differentiate the equivariance identity $\alpha_t^\Psi\circ\varphi_{M,z}=\varphi_{M,z}\circ \alpha_{\beta_M(t)}^\Theta$
along the path $t=e_r(s)$ at $s=0$. This gives
\[
\delta_r^\Psi(\varphi_{M,z}(a))
=
\varphi_{M,z}\!\left(
\left.\frac{d}{ds}\right|_{s=0}\alpha_{\beta_M(e_r(s))}^\Theta(a)
\right)
=
\sum_{j=1}^m M_{rj}\,\varphi_{M,z}(\delta_j^\Theta(a)),
\]
which is exactly \eqref{eq:toric-derivation-formula-merged}. Multiplying by
$-i$ gives the corresponding formula for $P=-i\delta$. \end{proof}

\begin{proposition}
\label{prop:FG-morphism-condition-criterion}
Let $\varphi=\varphi_{M,z}\colon A_\Theta\to A_\Psi$
be a unital toric map, with $M\in M_{n\times m}(\Z)$ and 
$z\in\T^m.$
Assume that $M$ has full column rank. Let $S_1$ and $S_2$ be finite-dimensional Hilbert spaces, and let $F\colon \Z^m\to \End(S_1)_{\mathrm{sa}}$ and 
$G\colon \Z^n\to \End(S_2)_{\mathrm{sa}}$
satisfy the two hypotheses of
Proposition~\ref{prop:single-torus-fourier-multiplier}, and let $D_F,$ $D_G$
be the corresponding Fourier-multiplier Dirac operators.

Write $A_F^\infty:=A_\Theta^\infty(D_F)$ and
$A_G^\infty:=A_\Psi^\infty(D_G)$, for simplicity. 
Define the graph norm on \(A_F^\infty\) by
$\|a\|_F:=\|a\|+\|[D_F,\pi_\Theta(a)]\|.$ Then the following are equivalent:

\smallskip

\noindent
 \emph{(i)} $\varphi(A_F^\infty)\subseteq A_G^\infty.$
 
 \smallskip

\noindent
 \emph{(ii)} The toric map is bounded from the \(D_F\)-graph algebra to the
\(D_G\)-graph algebra: there exists a constant \(C>0\) such that
\begin{equation}\label{eq:graph-bound-FG}
\|[D_G,\pi_\Psi(\varphi(a))]\|
\leq
C\bigl(\|a\|+\|[D_F,\pi_\Theta(a)]\|\bigr)
\end{equation}
for every $a\in A_F^\infty$.

 \smallskip

\noindent
 \emph{(iii)}  Choose a set $\mathcal R\subseteq \Z^n$ of representatives
for the quotient group $\Z^n/M\Z^m.$
For $r\in\mathcal R$, set $\mathcal H_r:=\ell^2(r+M\Z^m)\otimes S_2
\subseteq
\ell^2(\Z^n)\otimes S_2,$
and let $W_r\colon \ell^2(\Z^m)\otimes S_2\to \mathcal H_r$
be the unitary  $W_r(\delta_x\otimes \eta)=\delta_{r+Mx}\otimes \eta.$ Define
$D_{G,r}:=W_r^*(D_G|_{\mathcal H_r})W_r,$
so that $D_{G,r}(\delta_x\otimes\eta)=\delta_x\otimes G(r+Mx)\eta.$
Also define $\pi_{\varphi,r}(a):=W_r^*\bigl(\pi_\Psi(\varphi(a))|_{\mathcal H_r}\bigr)W_r,$
for all $a\in A_\Theta.$
Then, for every $a\in A_F^\infty$,  the commutators $[D_{G,r},\pi_{\varphi,r}(a)]$
extend to a bounded operator, and
\begin{equation}\label{eq:coset-uniform-condition}
\sup_{r\in\mathcal R}
\bigl\|
[D_{G,r},\pi_{\varphi,r}(a)]
\bigr\|
<\infty.
\end{equation}

Moreover, if $A_\Theta^{\mathrm{alg}}$ is a core for the closed derivation $a\longmapsto [D_F,\pi_\Theta(a)]$
with respect to the graph norm $\|\cdot\|_F$, then the above conditions are
equivalent to the following algebraic estimate: there is $C>0$ such that
\begin{equation}\label{eq:algebraic-graph-bound-FG}
\|[D_G,\pi_\Psi(\varphi(a))]\|
\leq
C\bigl(\|a\|+\|[D_F,\pi_\Theta(a)]\|\bigr)
\end{equation}
for every $a\in A_\Theta^{\mathrm{alg}}$.
\end{proposition}

\begin{proof}
We first prove the equivalence of \textup{(i)} and \textup{(ii)}.

Assume \textup{(i)}. Define $T\colon A_F^\infty\to \mathcal B(H_\Psi)$ by $T(a):=[D_G,\pi_\Psi(\varphi(a))].$ 
The space $A_F^\infty$, equipped with the graph norm $\|a\|_F=\|a\|+\|[D_F,\pi_\Theta(a)]\|,$
is a Banach space, because the commutator derivation
$a\mapsto [D_F,\pi_\Theta(a)]$ is closed.

We claim that $T$ has closed graph. Suppose $a_j\to a$ in $\|\cdot\|_F$
and $T(a_j)\to B$ in operator norm. Then $a_j\to a$ in $A_\Theta$, so $\pi_\Psi(\varphi(a_j))\to \pi_\Psi(\varphi(a))$
in operator norm. Let \(\xi\in\dom(D_G)\). Since $D_G\pi_\Psi(\varphi(a_j))\xi=T(a_j)\xi+\pi_\Psi(\varphi(a_j))D_G\xi,$
the right-hand side converges to $B\xi+\pi_\Psi(\varphi(a))D_G\xi.$
Also $\pi_\Psi(\varphi(a_j))\xi\to \pi_\Psi(\varphi(a))\xi.$
Since $D_G$ is closed, it follows that
$\pi_\Psi(\varphi(a))\xi\in\dom(D_G)$
and $D_G\pi_\Psi(\varphi(a))\xi=B\xi+\pi_\Psi(\varphi(a))D_G\xi.$
Thus
\[
[D_G,\pi_\Psi(\varphi(a))]\xi=B\xi
\]
on $\dom(D_G)$, and the commutator extends to the bounded operator $B$.
Hence $T(a)=B$. Therefore $T$ is closed. By the closed graph theorem,
$T$ is bounded, and this gives \eqref{eq:graph-bound-FG}.

Conversely, \eqref{eq:graph-bound-FG} immediately implies that $[D_G,\pi_\Psi(\varphi(a))]$
is bounded for every $a\in A_F^\infty$, hence $\varphi(a)\in A_G^\infty.$
Thus \textup{(ii)} implies \textup{(i)}.

We now prove the coset formulation. Put $L:=M\Z^m\subseteq \Z^n.$
Since $M$ has full column rank, the map
$\Z^m\ni x\mapsto Mx \in L$ 
is an isomorphism of abelian groups. Therefore 
\[
\ell^2(\Z^n)\otimes S_2
=
\bigoplus_{r\in\mathcal R}\mathcal H_r,
\qquad
\mathcal H_r=\ell^2(r+L)\otimes S_2.
\]

The diagonal operator $D_G$ preserves this decomposition. Moreover,
$\pi_\Psi(\varphi(a))$ preserves the decomposition for every
$a\in A_\Theta$. Indeed, $\varphi(A_\Theta)$ is generated by monomials
which shift Fourier labels by elements of the sublattice $L=M\Z^m$. Hence
left multiplication by any element of $\varphi(A_\Theta)$ preserves each
coset $r+L$.

Thus $D_G$ and $\pi_\Psi(\varphi(a))$ are block diagonal with respect to
the decomposition $\ell^2(\Z^n)\otimes S_2
=\bigoplus_{r\in\mathcal R}\mathcal H_r.$
On the block $\mathcal H_r$, the commutator is unitarily equivalent to $[D_{G,r},\pi_{\varphi,r}(a)].$
Therefore, whenever the block commutators extend boundedly and have uniformly
bounded norms, one has
\[
[D_G,\pi_\Psi(\varphi(a))]
=
\bigoplus_{r\in\mathcal R}
W_r[D_{G,r},\pi_{\varphi,r}(a)]W_r^*
\]
as a bounded block-diagonal operator. The uniform bound also gives the required domain invariance. Indeed, if
$\xi=(\xi_r)_r\in\operatorname{Dom}(D_G)$, then blockwise
\[
D_{G,r}\pi_{\varphi,r}(a)\xi_r
=
\pi_{\varphi,r}(a)D_{G,r}\xi_r
+
[D_{G,r},\pi_{\varphi,r}(a)]\xi_r,
\]
and the right-hand side is square-summable because
$\pi_{\varphi,r}(a)$ is uniformly bounded in $r$ and the block
commutators are uniformly bounded. Hence
$\pi_\Psi(\varphi(a))\xi\in\operatorname{Dom}(D_G)$.

Conversely, if $[D_G,\pi_\Psi(\varphi(a))]$
extends boundedly, then each block commutator extends boundedly and
$\bigl\|[D_{G,r},\pi_{\varphi,r}(a)]\bigr\|\leq\bigl\|[D_G,\pi_\Psi(\varphi(a))]\bigr\|$
for every $r$. Hence
\[
\|[D_G,\pi_\Psi(\varphi(a))]\|
=
\sup_{r\in\mathcal R}
\bigl\|[D_{G,r},\pi_{\varphi,r}(a)]\bigr\|.
\]
This proves the equivalence of \textup{(i)} and \textup{(iii)}.

Finally, assume that $A_\Theta^{\mathrm{alg}}$ is a graph-norm core for the
closed derivation $a\mapsto [D_F,\pi_\Theta(a)].$
If \eqref{eq:algebraic-graph-bound-FG} holds on $A_\Theta^{\mathrm{alg}}$,
let $a\in A_F^\infty$, and choose $a_j\in A_\Theta^{\mathrm{alg}}$ with
$a_j\to a$ in the graph norm. The estimate implies that $[D_G,\pi_\Psi(\varphi(a_j))]$
is Cauchy in operator norm. Let its limit be $B$. Since $\pi_\Psi(\varphi(a_j))\to\pi_\Psi(\varphi(a))$
in norm, closedness of $D_G$ gives $\varphi(a)\in A_G^\infty$ and
$[D_G,\pi_\Psi(\varphi(a))]=B.$
Thus \eqref{eq:graph-bound-FG} holds on all of $A_F^\infty$.  

Conversely, if \eqref{eq:graph-bound-FG} holds
on $A_F^\infty$, then it certainly holds on $A_\Theta^{\mathrm{alg}}$.
This proves the final assertion.
\end{proof}

\begin{remark}
\label{cor:FG-condition-easy-cases}
Keep the notation of Proposition~\ref{prop:FG-morphism-condition-criterion}.
\smallskip

\noindent
 
\textup{(1)} If $M\in\GL(m,\Z)$
so that $n=m$ and $M\Z^m=\Z^m$, then there is only one coset. Hence the inclusion $\varphi(A_F^\infty)\subseteq A_G^\infty$
holds if and only if the single transported commutator is bounded on
$A_F^\infty$, equivalently if and only if the graph estimate
\[
\|[D_G,\pi_\Psi(\varphi(a))]\|
\leq
C\bigl(\|a\|+\|[D_F,\pi_\Theta(a)]\|\bigr)
\]
holds.
\smallskip

\noindent
 
\textup{(2)} For flat triples, suppose $F(x)=2\pi\gamma_1(Lx)$ and $G(y)=2\pi\gamma_2(Hy),$
where $L\in\GL(m,\R)\), \(H\in\GL(n,\R)$, and
$\gamma_1,\gamma_2$ are Clifford representations. Then every unital toric
map $\varphi_{M,z}\colon A_\Theta\to A_\Psi$
satisfies $\varphi_{M,z}(A_F^\infty)\subseteq A_G^\infty.$

Indeed, put $T_\alpha^\Theta:=\sum_{j=1}^m L_{\alpha j}P_j^\Theta,$ for
$1\leq \alpha\leq m.$
Then $D_F=\sum_{\alpha=1}^m T_\alpha^\Theta\otimes\gamma_{1,\alpha}.$
Since the Clifford matrices $\gamma_{1,\alpha}$ are linearly independent,
there are bounded linear functionals $\lambda_\alpha\colon \End(S_1)\to\C$
such that $\lambda_\alpha(\gamma_{1,\beta})=\delta_{\alpha\beta}.$
Hence, if $a\in A_F^\infty$, then each commutator $[T_\alpha^\Theta,\pi_\Theta(a)]$
extends boundedly, since it is obtained from $[D_F,\pi_\Theta(a)]$
by applying $\id\otimes\lambda_\alpha$.

Similarly, put $S_\beta^\Psi:=\sum_{r=1}^n H_{\beta r}P_r^\Psi,$ for 
$1\leq \beta\leq n.$
Since $HM=(HML^{-1})L,$
the toric derivation formula $P_r^\Psi\circ\varphi_{M,z}=\sum_{j=1}^m M_{rj}\,\varphi_{M,z}\circ P_j^\Theta$
shows, first on algebraic Fourier polynomials and then by closedness of the
commutator derivations, that $[S_\beta^\Psi,\pi_\Psi(\varphi_{M,z}(a))]$
is a finite linear combination of the transported bounded commutators $[T_\alpha^\Theta,\pi_\Theta(a)].$
Thus each target directional commutator is bounded. Therefore $[D_G,\pi_\Psi(\varphi_{M,z}(a))]
=\sum_{\beta=1}^n[S_\beta^\Psi,\pi_\Psi(\varphi_{M,z}(a))]\otimes\gamma_{2,\beta}$
is bounded. Hence $\varphi_{M,z}(a)\in A_G^\infty.$
This proves $\varphi_{M,z}(A_F^\infty)\subseteq A_G^\infty.$
 \end{remark}

\begin{definition}[Compatible Fourier-multiplier structure]
\label{def:compatible-fourier-multiplier-data}
Let $A_{\Theta_1}\xrightarrow{\ \varphi_1\ }A_{\Theta_2}
\xrightarrow{\ \varphi_2\ }A_{\Theta_3}\xrightarrow{}\cdots$
be a unital toric inductive system, with $\varphi_n=\varphi_{M_n,z_n}$ monomial toric homomorphisms,  where $M_n\in M_{m_{n+1}\times m_n}(\Z)$ and $z_n\in\T ^{m_n}$.
A compatible Fourier-multiplier structure $(F_n,S_n,J_n)_{n\geq 1}$ consists of:
\smallskip

\noindent
 
 (i) finite-dimensional Hilbert spaces $S_n$;
\smallskip

\noindent
 
(ii)  isometries $J_n\colon S_n\to S_{n+1}$;

\smallskip

\noindent
 
(iii) selfadjoint matrix-valued functions $F_n\colon \Z^{m_n}\to\operatorname{End}(S_n)_{\mathrm{sa}},$
such that each $F_n$ satisfies the two hypotheses of
Proposition~\ref{prop:single-torus-fourier-multiplier}, and such that
\begin{equation}\label{eq:fourier-multiplier-compatibility}
F_{n+1}(M_nx)J_n=J_nF_n(x),
\qquad x\in\Z^{m_n}. \end{equation}
\end{definition}

\begin{proposition}
\label{prop:unital-fourier-multiplier-triples}
Let $A_{\Theta_1}\xrightarrow{\ \varphi_1\ }A_{\Theta_2}
\xrightarrow{\ \varphi_2\ }A_{\Theta_3}\xrightarrow{}\cdots$
be a unital toric inductive system, where each $\Theta_n$ is nondegenerate and $\varphi_n=\varphi_{M_n,z_n}$
is a monomial toric homomorphism for some $M_n\in M_{m_{n+1}\times m_n}(\Z)$ and $z_n\in\T^{m_n}.$ Let $(F_n,S_n,J_n)$ be a compatible Fourier-multiplier structure. 

For each $n$, put $A_n:=A_{\Theta_n},$ $H_n:=L^2(A_{\Theta_n},\tau_{\Theta_n})\otimes S_n,$ and $D_n:=D_{F_n}.$ Assume that, for every $n$,
\begin{equation}\label{eq:spec-alg}\varphi_n\bigl(A_{\Theta_n}^{\infty}(D_n)\bigr)
\subseteq
A_{\Theta_{n+1}}^{\infty}(D_{n+1}).\end{equation} Let $I_n:H_n\to H_{n+1}$
be defined on elementary vectors by $I_n(\Lambda_n(a)\otimes\xi)
:=
\Lambda_{n+1}(\varphi_n(a))\otimes J_n\xi,$ for all 
$a\in A_n$ and $\xi\in S_n.$ Then \begin{equation}\label{eq:inds}\bigl\{(A_n,H_n,D_n),(\varphi_n,I_n)\bigr\}_{n\geq1}\end{equation}
is an inductive system of spectral triples in the sense of
\cite[Definitions~2.1--2.3]{FloricelGhorbanpour}.
\end{proposition}

\begin{proof}
Since each $F_n$ satisfies the
two hypotheses of Proposition~\ref{prop:single-torus-fourier-multiplier}, $(A_{\Theta_n},H_n,D_n)$ is a spectral triple 
and $A_{n}^{\mathrm{alg}}
\subseteq
A_{n}^{\infty}(D_n).$

We next show that $I_n$ is an isometry. 
Let $a,b\in A_{n}$ and $\xi,\eta\in S_n$. Since $\tau_{\Theta_{n+1}}\circ\varphi_n=\tau_{\Theta_n}$, one has
\begin{eqnarray*}
&&\left\langle
I_n(\Lambda_n(a)\otimes\xi),
I_n(\Lambda_n(b)\otimes\eta)
\right\rangle
=
\left\langle
\Lambda_{n+1}(\varphi_n(a))\otimes J_n\xi,
\Lambda_{n+1}(\varphi_n(b))\otimes J_n\eta
\right\rangle  =      \\                                            
&&
\tau_{\Theta_{n+1}}\bigl(\varphi_n(b)^*\varphi_n(a)\bigr)
\langle J_n\xi,J_n\eta\rangle                                   
=
\tau_{\Theta_{n+1}}\bigl(\varphi_n(b^*a)\bigr)
\langle \xi,\eta\rangle                                           
=
\tau_{\Theta_n}(b^*a)\langle \xi,\eta\rangle                       
=
\left\langle
\Lambda_n(a)\otimes\xi,
\Lambda_n(b)\otimes\eta
\right\rangle .
\end{eqnarray*}
Hence $I_n$ extends to an isometry $I_n:H_n\to H_{n+1}.$

The representation intertwining is immediate from the definition. For
$c,a\in A_{n}$ and $\xi\in S_n$,
\begin{eqnarray*}
I_n\pi_n(c)(\Lambda_n(a)\otimes\xi)
&=&
I_n(\Lambda_n(ca)\otimes\xi)                                      
=
\Lambda_{n+1}(\varphi_n(ca))\otimes J_n\xi                         
=
\pi_{n+1}(\varphi_n(c))
\bigl(\Lambda_{n+1}(\varphi_n(a))\otimes J_n\xi\bigr)    \\         
&=&
\pi_{n+1}(\varphi_n(c))I_n(\Lambda_n(a)\otimes\xi).
\end{eqnarray*}
Thus
$I_n\pi_n(c)=\pi_{n+1}(\varphi_n(c))I_n,$ for all $c\in A_{n}.$

We now verify the Dirac intertwining. By
Lemma~\ref{lem:toric-fourier-smooth-derivations}, for each
$x\in\Z^{m_n}$ there exists a phase $\omega_n(x)\in\T$
such that $\varphi_n(U_n^x)=\omega_n(x)U_{n+1}^{M_nx}.$
Therefore $I_n(\delta_x^{(n)}\otimes\xi)=\omega_n(x)\,\delta_{M_nx}^{(n+1)}\otimes J_n\xi.$
Hence
$D_{n+1}I_n(\delta_x^{(n)}\otimes\xi)
=D_{n+1}
\bigl(
\omega_n(x)\delta_{M_nx}^{(n+1)}\otimes J_n\xi
\bigr)                                                         
=\omega_n(x)\delta_{M_nx}^{(n+1)}
\otimes F_{n+1}(M_nx)J_n\xi,$
whereas $I_nD_n(\delta_x^{(n)}\otimes\xi)=
I_n(\delta_x^{(n)}\otimes F_n(x)\xi)=
\omega_n(x)\delta_{M_nx}^{(n+1)}
\otimes J_nF_n(x)\xi.$
These two expressions are equal by the compatibility condition \eqref{eq:fourier-multiplier-compatibility}. Thus $D_{n+1}I_n=I_nD_n$
on the algebraic Fourier span $A_{\Theta_n}^{\mathrm{alg}}\otimes S_n.$

We claim that the equality $D_{n+1}I_n=I_nD_n$ extends to the full domain of $D_n$. Indeed, let $\zeta\in\dom(D_n),$ and
choose a sequence $\zeta_\ell\in A_{\Theta_n}^{\mathrm{alg}}\otimes S_n$
such that $\zeta_\ell\to \zeta$ and 
$D_n\zeta_\ell\to D_n\zeta.$ Such a sequence exists because the finite Fourier-support vectors form a core
for the diagonal operator $D_n$. Since $I_n$ is an isometry, $I_n\zeta_\ell\to I_n\zeta$ and $I_nD_n\zeta_\ell\to I_nD_n\zeta.$ But on the algebraic core we have $D_{n+1}I_n\zeta_\ell=I_nD_n\zeta_\ell.$
Since $D_{n+1}$ is closed, it follows that $I_n\zeta\in\dom(D_{n+1})$
and $D_{n+1}I_n\zeta=I_nD_n\zeta.$
Thus $I_n(\dom D_n)\subseteq \dom D_{n+1}$ and $D_{n+1}I_n=I_nD_n$
on $\dom(D_n)$.

Finally, assuming $\varphi_n\bigl(A_{n}^{\infty}(D_n)\bigr)\subseteq A_{n+1}^{\infty}(D_{n+1}),$ for all $n$, we conclude that the pair
$(\varphi_n,I_n)$ satisfies the requirements for an isometric morphism of spectral triples in the
sense of \cite[Definition~2.1]{FloricelGhorbanpour}: the homomorphism
$\varphi_n$ is a unital injective $*$-homomorphism, the isometry $I_n$
intertwines the representations, the spectral algebra is mapped into the next
spectral algebra (by the assumed inclusion above), and the Dirac operators
intertwine on domains. Hence $\bigl\{(A_n,H_n,D_n),(\varphi_n,I_n)\bigr\}_{n\geq1}$
is an inductive system of spectral triples in the sense of
\cite[Definition~2.2]{FloricelGhorbanpour}. 
\end{proof}
Let \[ (A_\ex:=\varinjlim A_n,H_\ex:=\varinjlim H_n,D_\ex:=\varinjlim D_n),\] denote the inductive realization of the inductive system \eqref{eq:inds}.

\begin{corollary}
\label{cor:limit-multiplier-criterion}
Assume the hypotheses of Proposition~\ref{prop:unital-fourier-multiplier-triples}.
Assume in addition that $S_n=S$ and $J_n=\id_S$
for every $n$.
Then the formula $F_\infty([x,n]):=F_n(x)$
defines a selfadjoint matrix-valued function $F_\infty\colon
G_\infty:=\varinjlim(\Z^{m_n},M_n)\to \End(S)_{\mathrm{sa}},
$ where we write $[x,n]\in G_\infty$ for the
class of $x\in\Z^{m_n}$.

If $F_\infty$ has (a)  finite spectral multiplicity in bounded intervals; and (b) bounded
translation increments, that is (a)
for every $R>0$, the vector space
\[
\bigoplus_{\substack{g\in G_\infty\\
\operatorname{Spec}(F_\infty(g))\cap[-R,R]\neq\varnothing}}
S
\]
is finite-dimensional, and (b) for every
$h\in G_\infty$, $\sup_{g\in G_\infty}
\|F_\infty(g+h)-F_\infty(g)\|<\infty,
$
then the inductive realization
$(A_\ex,H_\ex,D_\ex)$
is a spectral triple on the protoral $C^*$-algebra $A_\ex$.
\end{corollary}
\begin{proof}
First we check that $F_\infty$ is well-defined. For $k>n$, write $M_{k,n}:=M_{k-1}M_{k-2}\cdots M_n$ and $M_{n,n}:=\id_{\Z^{m_n}}.$
Since $S_n=S$ and $J_n=\id_S$, the compatibility relation
\eqref{eq:fourier-multiplier-compatibility} gives $F_{n+1}(M_nx)=F_n(x),$ for every $x\in\Z^{m_n}$.
By induction, $F_k(M_{k,n}x)=F_n(x)$,
for every $k\geq n$ and $x\in\Z^{m_n}.$
Thus if $[x,n]=[y,\ell]\in G_\infty,$
then for some $k\geq n,\ell$, $M_{k,n}x=M_{k,\ell}y.$
Hence $F_n(x)=F_k(M_{k,n}x)=F_k(M_{k,\ell}y)=F_\ell(y).$
Therefore $F_\infty([x,n]):=F_n(x)$
is independent of the representative.

We next identify the Hilbert-space inductive limit explicitly. Since
$\Theta_n$ is nondegenerate and $\Theta_n\equiv M_n^t\Theta_{n+1}M_n
\pmod{M_{m_n}(\Z)_{\mathrm{skew}}},$
each $M_n\colon\Z^{m_n}\to\Z^{m_{n+1}}$ has trivial kernel. Indeed, if
$M_nx=0$ for some nonzero $x\in\Z^{m_n}$, then for all
$y\in\Z^{m_n}$, $x^t\Theta_n y=x^t(M_n^t\Theta_{n+1}M_n)y+x^tKy=x^tKy\in\Z, $
where $K:=\Theta_n-M_n^t\Theta_{n+1}M_n
\in M_{m_n}(\Z)_{\mathrm{skew}},$
contradicting nondegeneracy of $\Theta_n$. 
Hence the maps $M_n$ are injective as group homomorphisms, so the direct-limit
group $G_\infty=\varinjlim(\Z^{m_n},M_n)$ may be represented by compatible lattice labels.

For each $n$ and $x\in\Z^{m_n}$, write
$\varphi_n(U_n^x)=\omega_n(x)\,U_{n+1}^{M_nx}$, with $\omega_n(x)\in\T.$
Here $U_n^x$ denotes the ordered Fourier monomial in $A_{\Theta_n}$.
Since $M_n$ is injective, we may choose functions $\chi_n\colon\Z^{m_n}\to\T$
recursively so that $\chi_{n+1}(M_nx)=\chi_n(x)\,\overline{\omega_n(x)}$, for all
$x\in\Z^{m_n}.$
Indeed, this defines $\chi_{n+1}$ on the sublattice
$M_n\Z^{m_n}$, and we extend it arbitrarily to all of
$\Z^{m_{n+1}}$. We choose $\chi_n(0)=1$.

Let $\jmath_n\colon H_n\to H_\ex$
denote the canonical Hilbert-space maps into the inductive limit. Define a
unitary $\mathcal U\colon H_\ex\longrightarrow \ell^2(G_\infty)\otimes S$
on the dense union of the stage Hilbert spaces by
\[
\mathcal U\bigl(\jmath_n(\delta_x^{(n)}\otimes\xi)\bigr)
=
\chi_n(x)\,\delta_{[x,n]}\otimes\xi,
\qquad
x\in\Z^{m_n},\ \xi\in S.
\]
This is well-defined. Indeed,
$\jmath_n(\delta_x^{(n)}\otimes\xi)=\jmath_{n+1}\bigl(I_n(\delta_x^{(n)}\otimes\xi)\bigr)
=\omega_n(x)\jmath_{n+1}(\delta_{M_nx}^{(n+1)}\otimes\xi),$ whereas
$\mathcal U\bigl(
\omega_n(x)\jmath_{n+1}(\delta_{M_nx}^{(n+1)}\otimes\xi)
\bigr)
=\omega_n(x)\chi_{n+1}(M_nx)
\delta_{[M_nx,n+1]}\otimes\xi=\chi_n(x)\delta_{[x,n]}\otimes\xi.$

Therefore, under this identification $H_\ex\cong \ell^2(G_\infty)\otimes S,$
the image of $\delta_x^{(n)}\otimes\xi$ is $\chi_n(x)\,\delta_{[x,n]}\otimes\xi.$ Moreover, under the same identification, the inductive-limit operator is the diagonal Fourier multiplier
\[
D_\ex(\delta_g\otimes\xi)
=
\delta_g\otimes F_\infty(g)\xi,
\qquad
g\in G_\infty,\ \xi\in S.
\]
Indeed, if $g=[x,n]$, then $D_\ex\jmath_n(\delta_x^{(n)}\otimes\xi)
=\jmath_n(\delta_x^{(n)}\otimes F_n(x)\xi),$
and this becomes $\delta_g\otimes F_\infty(g)\xi$
after applying $\mathcal U$. Since each $F_\infty(g)$ is selfadjoint on the
finite-dimensional Hilbert space $S$, the operator $D_\ex$ is selfadjoint
with domain
\[
\dom(D_\ex)
=
\left\{
(\xi_g)_{g\in G_\infty}\in\bigoplus_{g\in G_\infty}S:
\sum_{g\in G_\infty}\|F_\infty(g)\xi_g\|^2<\infty
\right\}.
\]

The finite-spectral-multiplicity hypothesis gives compact resolvent. For
$R>0$, set $X_R:=
\{g\in G_\infty:
\operatorname{Spec}(F_\infty(g))\cap[-R,R]\neq\varnothing\}.$
By assumption, $\bigoplus_{g\in X_R}S$
is finite-dimensional. Hence the spectral projection $1_{[-R,R]}(D_\ex)=\bigoplus_{g\in G_\infty}1_{[-R,R]}(F_\infty(g))$
has finite rank. Therefore $(1+D_\ex^2)^{-1/2}$ is compact by the same
diagonal-operator argument as in
Proposition~\ref{prop:single-torus-fourier-multiplier}.

It remains to identify a dense algebraic subalgebra and check bounded
commutators. Let $\iota_n\colon A_{\Theta_n}\to A_\ex$
be the canonical $C^*$-algebra maps. For $h=[x,n]\in G_\infty$, define
\[
W_h
:=
\overline{\chi_n(x)}\,\iota_n(U_n^x)
\in A_\ex.
\]
This is independent of the representative. Indeed, $\iota_n(U_n^x)
=\iota_{n+1}(\varphi_n(U_n^x))=\omega_n(x)\,\iota_{n+1}(U_{n+1}^{M_nx}),$
while $\overline{\chi_{n+1}(M_nx)}=
\overline{\chi_n(x)\overline{\omega_n(x)}}=
\overline{\chi_n(x)}\,\omega_n(x).$
Thus $\overline{\chi_n(x)}\,\iota_n(U_n^x)=\overline{\chi_{n+1}(M_nx)}\,\iota_{n+1}(U_{n+1}^{M_nx}).$

The algebraic inductive limit is therefore
\[
\mathcal A_{\mathrm{alg}}
=
\operatorname{span}\{W_h:h\in G_\infty\},
\]
and it is dense in $A_\ex$. We now compute the product and the resulting twisted shifts. At stage $k$,
write $U_k^aU_k^b=\sigma_k(a,b)U_k^{a+b}$, for $a,b\in\Z^{m_k}$,
where, in the ordered monomial convention,
\[
\sigma_k(a,b)
=
\exp\!\left(
2\pi i
\sum_{1\leq p<q\leq m_k}
a_qb_p\,(\Theta_k)_{p,q}
\right).
\]
Let $h=[x,n]$ and $g=[y,\ell]$. Choose $k\geq n,\ell$, and put $a:=M_{k,n}x,$
$b:=M_{k,\ell}y.$
Then $h=[a,k],$ $g=[b,k],$ $h+g=[a+b,k].$
Define
\[
\sigma_\infty(h,g)
:=
\overline{\chi_k(a)}\,
\overline{\chi_k(b)}\,
\chi_k(a+b)\,
\sigma_k(a,b).
\]
This scalar is independent of the chosen common stage $k$, because the
elements $W_h$ are already well-defined in the inductive limit. With this
notation, $W_hW_g=\sigma_\infty(h,g)\,W_{h+g}.$
Thus $\sigma_\infty\colon G_\infty\times G_\infty\to\T$ is the $2$-cocycle
determined by the normalized Fourier basis $\{W_h\}_{h\in G_\infty}$.

Under the above identification $H_\ex\cong\ell^2(G_\infty)\otimes S,$
left multiplication by $W_h$, denoted again by $W_h$, acts as the twisted
shift
\[
W_h(\delta_g\otimes\xi)
=\sigma_\infty(h,g)\,\delta_{h+g}\otimes\xi,
\qquad g,h\in G_\infty,\ \xi\in S.
\]
Therefore $[D_\ex,W_h](\delta_g\otimes\xi)=\sigma_\infty(h,g)\,\delta_{h+g}\otimes\bigl(F_\infty(h+g)-F_\infty(g)\bigr)\xi.$
By the bounded-translation-increment hypothesis, $\sup_{g\in G_\infty}\|F_\infty(h+g)-F_\infty(g)\|<\infty$
for each fixed $h\in G_\infty$. Hence $[D_\ex,W_h]$ extends to a bounded
operator for every $h\in G_\infty$. By linearity, $[D_\ex,a]$
is bounded for every $a\in\mathcal A_{\mathrm{alg}}$. Thus $
\mathcal A_{\mathrm{alg}}\subseteq A_{\ex}^{\infty}(D_{\ex})$.

Finally, the representation of $A_{\ex}$ on $H_{\ex}$ is faithful. Indeed, on each
finite stage it agrees, after the Hilbert-space identifications, with the
faithful GNS representation of $A_{\Theta_n}$. Hence the representation is
isometric on each subalgebra $\iota_n(A_{\Theta_n})$. Since the algebraic
inductive limit is the union of these subalgebras and its $C^*$-norm is the
inductive-limit norm, the representation is isometric on the algebraic
inductive limit. It therefore extends faithfully to the $C^*$-completion
$A_{\ex}$. Hence $(A_\ex,H_\ex,D_\ex)$
is a spectral triple on $A_\ex$.
\end{proof}

\begin{example}[The flat inductive-limit operator on the unital toric $N$-solenoid]
\label{ex:solenoid-spectral}
Let $A_{\ex}^{\theta,N}$ be the unital toric $N$-solenoid of
Example~\ref{ex:solenoid}. Thus $\theta_n=\frac{\theta}{N^{2(n-1)}},$
$A_n=A_{\theta_n},$
and the connecting maps are $\phi_n(U_n)=U_{n+1}^N$, 
$\phi_n(V_n)=V_{n+1}^N.$ The connecting matrix on Fourier labels is $M_n=NI_2.$

Let $S_2$ be a fixed complex Clifford module for $\R^2$, with Clifford
multiplication denoted by $\gamma(v)=v_1\gamma_1+v_2\gamma_2$, for all $v=(v_1,v_2)\in\R^2.$
For each $n\geq1$, define $F_n\colon \Z^2\to \End(S_2)_{\mathrm{sa}}$
by
\[
F_n(x):=2\pi\,\gamma\bigl(N^{-(n-1)}x\bigr),
\qquad x\in\Z^2.
\]
Equivalently, on the Fourier basis of $H_n:=L^2(A_n,\tau_{\theta_n})\otimes S_2,$
the corresponding operator is
\[
D_n(\delta_x^{(n)}\otimes\xi)
=
2\pi\,\delta_x^{(n)}\otimes
\gamma\bigl(N^{-(n-1)}x\bigr)\xi,
\qquad x\in\Z^2,\ \xi\in S_2.
\]
This is exactly the flat triple of
Observation~\ref{rem:unital-special-spectral-triples}\textup{(1)} with $\Theta=J(\theta_n)$ and $L_n=N^{-(n-1)}I_2$. In the notation of that observation, $D_n=D_{J(\theta_n),\,L_n}.$
The usual unrescaled flat operator at the same stage corresponds to
$L=I_2$, namely $D_{J(\theta_n),I_2}$. 
Notice that $D_n=N^{-(n-1)}D_{J(\theta_n),I_2},$ so
$D_n$ is the standard flat operator on $A_{\theta_n}$, rescaled by
$N^{-(n-1)}$. Since $L_n\in\GL(2,\R)$, the finite-spectral-multiplicity and
bounded-translation-increment hypotheses of
Proposition~\ref{prop:single-torus-fourier-multiplier} hold by
Observation~\ref{rem:unital-special-spectral-triples}\textup{(1)}. Hence each $(A_n,H_n,D_n)$
is a spectral triple.

The rescaling is precisely what makes the operators compatible with the
connecting maps. Indeed, $F_{n+1}(M_nx)
=F_{n+1}(Nx)
=2\pi\,\gamma\bigl(N^{-n}Nx\bigr)
=2\pi\,\gamma\bigl(N^{-(n-1)}x\bigr)=F_n(x).$ Without this rescaling, the unrescaled flat symbols would satisfy $2\pi\gamma(Nx)=N\,2\pi\gamma(x),$
so the Dirac operators would not intertwine exactly with the toric embeddings.

We next verify the spectral-algebra inclusion \eqref{eq:spec-alg}, that is
$\phi_n\bigl(A_n^\infty(D_n)\bigr)\subseteq A_{n+1}^\infty(D_{n+1}).$
The target Fourier lattice decomposes into finitely many cosets of
$N\Z^2$:
\[
\Z^2=\bigsqcup_{r\in\mathcal R}(r+N\Z^2),
\qquad
\mathcal R:=\{0,\ldots,N-1\}^2.
\]
For $r\in\mathcal R$, set
\[
\mathcal H_r:=\ell^2(r+N\Z^2)\otimes S_2
\subseteq
\ell^2(\Z^2)\otimes S_2.
\]
Both $D_{n+1}$ and $\pi_{n+1}(\phi_n(a))$ preserve this decomposition,
because $\phi_n(a)$ shifts Fourier labels only by elements of $N\Z^2$. Identify $\ell^2(\Z^2)\otimes S_2$ with $\mathcal H_r$ by
\[
W_r(\delta_x\otimes\xi)=\delta_{r+Nx}\otimes\xi.
\]
On this coset block, $W_r^*D_{n+1}W_r(\delta_x\otimes\xi)=\delta_x\otimes F_{n+1}(r+Nx)\xi.$ But $F_{n+1}(r+Nx)=2\pi\,\gamma\bigl(N^{-n}r+N^{-(n-1)}x\bigr)=F_n(x)+C_{r,n},$
where $C_{r,n}:=2\pi\,\gamma(N^{-n}r)$
is a fixed bounded endomorphism of $S_2$, independent of $x$.

On the same block, $W_r^*\pi_{n+1}(\phi_n(a))W_r$ is the left regular
representation of a gauge transform of $a$. More explicitly, one checks on
the canonical generators that $W_r^*\pi_{n+1}(\phi_n(U_n))W_r=\pi_n(U_n),$
and $W_r^*\pi_{n+1}(\phi_n(V_n))W_r=e^{2\pi iN r_1\theta_{n+1}}\pi_n(V_n).$
Indeed, the second identity follows from the left regular action formula $V_{n+1}^N\delta_y=e^{2\pi iN y_1\theta_{n+1}}\delta_{y+(0,N)}$
and the relation $N^2\theta_{n+1}=\theta_n$. Therefore
\[
W_r^*\pi_{n+1}(\phi_n(a))W_r
=
\pi_n(\beta_{r,n}(a)),
\]
where $\beta_{r,n}\in\Aut(A_n)$ is the gauge automorphism determined by $\beta_{r,n}(U_n)=U_n$, $\beta_{r,n}(V_n)=e^{2\pi iN r_1\theta_{n+1}}V_n.$

Gauge automorphisms are implemented on $L^2(A_n,\tau_{\theta_n})$ by diagonal
unitaries in the Fourier basis. More precisely, $\beta_{r,n}$ is implemented by the unitary $G_{r,n}$ determined on
Fourier basis vectors by $G_{r,n}\delta_{(p,q)}^{(n)}=e^{2\pi iN r_1\theta_{n+1}q}\delta_{(p,q)}^{(n)}.$
Since $D_n$ is diagonal in the same Fourier basis, $G_{r,n}\otimes 1_{S_2}$
commutes with $D_n$. Hence, if $a\in A_n^\infty(D_n)$, then $\beta_{r,n}(a)\in A_n^\infty(D_n),$
and
\[
[D_n,\pi_n(\beta_{r,n}(a))]
=
(G_{r,n}\otimes 1_{S_2})
[D_n,\pi_n(a)]
(G_{r,n}^*\otimes 1_{S_2})
\]
is bounded. On the $r$-th coset block we have $W_r^*D_{n+1}W_r=D_n+1\otimes C_{r,n}.$ Since $1\otimes C_{r,n}$ acts only on the spinor factor, it commutes with
$\pi_n(\beta_{r,n}(a))$. Consequently,
\[
[
W_r^*D_{n+1}W_r,\,
W_r^*\pi_{n+1}(\phi_n(a))W_r
]
=
[D_n,\pi_n(\beta_{r,n}(a))],
\]
which is bounded. Moreover, since $1\otimes C_{r,n}$ is bounded, the domain of
$W_r^*D_{n+1}W_r=D_n+1\otimes C_{r,n}$ is $\dom(D_n)$. As
$\beta_{r,n}(a)\in A_n^\infty(D_n)$, the operator
$\pi_n(\beta_{r,n}(a))$ preserves $\dom(D_n)$. Hence the transported
operator $W_r^*\pi_{n+1}(\phi_n(a))W_r$ preserves the domain of
$W_r^*D_{n+1}W_r$, and the displayed commutator extends boundedly.

There are only finitely many cosets $r\in\mathcal R$, so the block
commutators are uniformly bounded. Therefore $\phi_n(a)\in A_{n+1}^\infty(D_{n+1})$
for every $a\in A_n^\infty(D_n)$, proving \eqref{eq:spec-alg}. Thus Proposition~\ref{prop:unital-fourier-multiplier-triples} gives an
inductive system of spectral triples and hence an operator-level inductive
realization $D_{\mathrm{proflat}}$ on the Hilbert-space limit associated to the protoral algebra
$A_{\ex}^{\theta,N}$.

We now show that this limit operator does not have compact resolvent. The
limiting Fourier group is $\Gamma_N:=\Z[1/N]^2.$
The label $x\in\Z^2$ at stage $n$ corresponds to $\frac{x}{N^{n-1}}\in\Gamma_N.$
Thus, under the identification $L^2(A_{\ex}^{\theta,N},\tau_{\theta,N})\otimes S_2
\cong
\ell^2(\Gamma_N)\otimes S_2,$
the limit operator is the diagonal flat multiplier
\[
D_{\mathrm{proflat}}(\delta_g\otimes\xi)
=
2\pi\,\delta_g\otimes\gamma(g)\xi,
\qquad
g\in\Gamma_N,\ \xi\in S_2.
\]
Consequently $D_{\mathrm{proflat}}^2(\delta_g\otimes\xi)=4\pi^2\|g\|^2\,\delta_g\otimes\xi,$
and hence $(1+D_{\mathrm{proflat}}^2)^{-1/2}(\delta_g\otimes\xi)=(1+4\pi^2\|g\|^2)^{-1/2}\delta_g\otimes\xi.$

Now consider the sequence $g_k:=\left(\frac{1}{N^k},0\right)\in\Gamma_N.$
Then $\|g_k\|\to 0,$
so $(1+4\pi^2\|g_k\|^2)^{-1/2}\to 1.$
Choose a fixed unit vector $\xi\in S_2$. The vectors $\delta_{g_k}\otimes\xi$
are mutually orthonormal, while their images under
$(1+D_{\mathrm{proflat}}^2)^{-1/2}$ have norms tending to $1$. Therefore the
sequence $(1+D_{\mathrm{proflat}}^2)^{-1/2}(\delta_{g_k}\otimes\xi)$
has no norm-convergent subsequence. Hence $(1+D_{\mathrm{proflat}}^2)^{-1/2}$
is not compact.

Equivalently, every bounded spectral interval containing $0$ contains
infinitely many Fourier labels $g_k$, so the corresponding spectral
projection is infinite-dimensional. Thus the rescaled stagewise flat triples give a compatible inductive-limit
unbounded operator, but the resulting limit operator is not a spectral triple
on $A_{\ex}^{\theta,N}$, because it fails compact resolvent.
\end{example}

\begin{example}[A compact-resolvent length triple on the unital toric \(N\)-solenoid]
\label{ex:compact-resolvent-solenoid-dirac} Let $A_{\ex}^{\theta,N}=\varinjlim(A_n,\phi_n)$, with $A_n=A_{\theta/N^{2(n-1)}}$ and 
$\phi_n(U_n)=U_{n+1}^N,$ $\phi_n(V_n)=V_{n+1}^N$
 be as in Example~\ref{ex:solenoid-spectral}. Its
limiting Fourier group is $\Gamma_N=\Z[1/N]^2.$

For $g\in\Gamma_N$, define the $N$-denominator depth $h_N(g):=\min\{k\geq0 : N^kg\in\Z^2\},
$
and set
\[
\ell_N(g):=|g|+h_N(g),
\]
where $|\cdot|$ is the Euclidean norm on $\R^2$.  Then $\ell_N$ is a
proper length function on the discrete abelian group $\Gamma_N$.

We check that $\ell_N$ is proper. Let $R>0$. If $\ell_N(g)\leq R,$
then $h_N(g)\leq R$ and $|g|\leq R.$ Thus $g$ has denominator depth at most $\lfloor R\rfloor$, so $g\in N^{-\lfloor R\rfloor}\Z^2,$
and $g$ lies in the Euclidean ball of radius $R$. There are only finitely
many such points. Hence $\{g\in\Gamma_N:\ell_N(g)\leq R\}$
is finite.

Next we check bounded translation increments. Fix \(a\in\Gamma_N\). The
Euclidean part satisfies $\bigl||a+g|-|g|\bigr|\leq |a|.$
For the denominator-depth part, note that $h_N(a+g)\leq \max\{h_N(a),h_N(g)\},$
and, writing $g=(a+g)-a$, $h_N(g)\leq \max\{h_N(a+g),h_N(a)\}.$
These two inequalities imply $|h_N(a+g)-h_N(g)|\leq h_N(a).$
Therefore $|\ell_N(a+g)-\ell_N(g)|\leq|a|+h_N(a),$
uniformly in $g$. Hence $\ell_N$ has bounded translation increments.

For each stage $n$, define $\ell_n(x):=\ell_N\left(\frac{x}{N^{n-1}}\right),$ for 
$ x\in\Z^2.$ Let
\[
D_n^\ell\delta_x^{(n)}:=\ell_n(x)\delta_x^{(n)}
\]
on $H_n=L^2(A_n,\tau_{\theta_n}).$

We verify the two hypotheses of
Proposition~\ref{prop:single-torus-fourier-multiplier} at each finite stage.
Since the map $\Z^2\ni x\to  \frac{x}{N^{n-1}}\in\Gamma_N,$ is injective, properness of $\ell_N$ implies properness of $\ell_n$. Thus,
for each $R>0$, $\{x\in\Z^2:\ell_n(x)\leq R\}$ is finite. This is exactly finite spectral multiplicity in bounded intervals
for the scalar multiplier $F_n=\ell_n$.

Similarly, for $a\in\Z^2$,
$
|\ell_n(x+a)-\ell_n(x)|=\left|
\ell_N\left(\frac{x}{N^{n-1}}+\frac{a}{N^{n-1}}\right)
-
\ell_N\left(\frac{x}{N^{n-1}}\right)\right|.$
Since $\ell_N$ has bounded translation increments on $\Gamma_N$, the
right-hand side is uniformly bounded in $x$. Hence $\ell_n$ has bounded
translation increments on $\Z^2$. Therefore each $(A_n,H_n,D_n^\ell)$
is a spectral triple.

The finite-stage multipliers are compatible with the connecting maps: $\ell_{n+1}(Nx)
=\ell_N\left(\frac{Nx}{N^n}\right)
=\ell_N\left(\frac{x}{N^{n-1}}\right)=\ell_n(x),$
so $F_{n+1}(M_nx)=F_n(x), $ where $M_n=NI_2.$ 
We next check the spectral-algebra inclusion $\phi_n\bigl(A_n^\infty(D_n^\ell)\bigr)
\subseteq A_{n+1}^\infty(D_{n+1}^\ell).$

As in Example~\ref{ex:solenoid-spectral}, decompose $\Z^2=\bigsqcup_{r\in\mathcal R}(r+N\Z^2),$ where
$\mathcal R:=\{0,\ldots,N-1\}^2,$ and set $\mathcal H_r:=\ell^2(r+N\Z^2)\subseteq \ell^2(\Z^2).$
Let $W_r\colon \ell^2(\Z^2)\to \mathcal H_r$, $W_r\delta_x=\delta_{r+Nx}.$
Then $D_{n+1}^\ell$ and \(\pi_{n+1}(\phi_n(a))\) preserve the coset decomposition. On the $r$-th coset block, $W_r^*D_{n+1}^\ell W_r\,\delta_x=\ell_{n+1}(r+Nx)\delta_x.$ Since $\ell_{n+1}(r+Nx)=\ell_N\left(\frac{r+Nx}{N^n}\right)
=\ell_N\left(\frac{x}{N^{n-1}}+\frac{r}{N^n}\right),$
we may write $W_r^*D_{n+1}^\ell W_r=D_n^\ell+B_{r,n},$
where \(B_{r,n}\) is the diagonal bounded operator
\[
B_{r,n}\delta_x
=
\left[
\ell_N\left(\frac{x}{N^{n-1}}+\frac{r}{N^n}\right)
-
\ell_N\left(\frac{x}{N^{n-1}}\right)
\right]\delta_x .
\]
The boundedness of $B_{r,n}$ follows from the bounded translation increments
of $\ell_N$, applied to the fixed element $r/N^n\in\Gamma_N$.

As in Example~\ref{ex:solenoid-spectral}, on the same coset block one has $W_r^*\pi_{n+1}(\phi_n(a))W_r
=\pi_n(\beta_{r,n}(a)),$ where $\beta_{r,n}$ is a gauge automorphism of $A_n$. Gauge automorphisms
are implemented by diagonal unitaries in the Fourier basis, and these commute
with $D_n^\ell$. Hence, if $a\in A_n^\infty(D_n^\ell)$, then $\beta_{r,n}(a)\in A_n^\infty(D_n^\ell),$
and $[D_n^\ell,\pi_n(\beta_{r,n}(a))]$
is bounded.

Since $B_{r,n}$ is bounded, the commutator $[B_{r,n},\pi_n(\beta_{r,n}(a))]$
is automatically bounded. Therefore
\[
[
W_r^*D_{n+1}^\ell W_r,\,
W_r^*\pi_{n+1}(\phi_n(a))W_r
]
=
[D_n^\ell,\pi_n(\beta_{r,n}(a))]
+
[B_{r,n},\pi_n(\beta_{r,n}(a))]
\]
extends to a bounded operator. There are only finitely many cosets
$r\in\mathcal R$, so the block commutators are uniformly bounded. Thus $\phi_n(a)\in A_{n+1}^\infty(D_{n+1}^\ell)$
for every $a\in A_n^\infty(D_n^\ell)$.

The finite-stage multipliers therefore glue to the limit multiplier
\[
F_\infty=\ell_N\colon \Gamma_N\to\R,
\]
and $\ell_N$ has finite spectral multiplicity in bounded intervals and
bounded translation increments on \(\Gamma_N\). Hence
Corollary~\ref{cor:limit-multiplier-criterion} applies. The resulting
inductive-limit operator $D_{\ex}^{ \ell_N}$
defines a spectral triple on $A_{\ex}^{\theta,N}$. Under the identification $L^2(A_{\ex}^{\theta,N},\tau_{\theta,N})\cong \ell^2(\Gamma_N)$
it is the diagonal operator
\[
D_{\ex}^{ \ell_N}\delta_g=\ell_N(g)\delta_g,
\qquad g\in\Gamma_N.
\]

Thus the unital toric $N$-solenoid carries both a natural flat inductive-limit
operator, which fails to have compact resolvent, and a natural length-type
spectral triple, which does have compact resolvent.
\end{example}

\begin{observation}
This length-type triple is in the same general spirit as the length-function
spectral triples of \cite{FarsiLandryLarsenPackerSolenoids}. However, the
specific length $\ell_N(g)=|g|+h_N(g)$ is chosen only to be proper and to have
bounded translation increments. It is not bounded-doubling and is not finitely
summable. Thus the construction proves compact resolvent and bounded
commutators, but it does not give the stronger metric and Wiener-lemma
properties obtained in \cite{FarsiLandryLarsenPackerSolenoids}.

The construction above should be distinguished from the spectral triples of
\cite{FarsiLatremolierePackerStandardSolenoids}. The naive flat inductive-limit
operator associated to the standard torus Dirac operators fails to have compact
resolvent, because Euclidean length is not proper on
$\Gamma_N=\Z[1/N]^2$. The present example fixes this by adding the unbounded
denominator-depth term $h_N$, producing a proper length function on
$\Gamma_N$. Thus it is closer to the length-function construction of
\cite{FarsiLandryLarsenPackerSolenoids}. By contrast,
\cite{FarsiLatremolierePackerStandardSolenoids} studies how to obtain spectral
triples on noncommutative solenoids from the standard quantum-torus spectral
triples, using appropriate bounded perturbations and the spectral propinquity.
\end{observation}

\begin{example}[A compatible inner fluctuation on the unital toric \(N\)-solenoid]
\label{ex:solenoid-inner-fluctuation}

We now introduce a compatible inner fluctuation for the inductive system of spectral triples $(A_n,H_n,D_n^\ell)$ constructed in Example~\ref{ex:compact-resolvent-solenoid-dirac}, retaining the same notation. Thus $A_n=A_{\theta/N^{2(n-1)}},$
$H_n=L^2(A_n,\tau_{\theta_n}),$
and $D_n^\ell\delta_x^{(n)}=\ell_n(x)\delta_x^{(n)},$ $\ell_n(x)=\ell_N\left(\frac{x}{N^{n-1}}\right).$

Fix $h:=(1,0)\in \Gamma_N=\Z[1/N]^2,$
and let $W_h$ denote the corresponding Fourier unitary in the algebraic
twisted group algebra of the limit. At stage $n$, the same Fourier label $h$
is represented by the integer vector $h_n:=N^{n-1}(1,0)\in\Z^2.$
Let $u_n:=U_n^{N^{n-1}}\in A_{\theta_n}^{\mathrm{alg}}$
be the corresponding unitary. Then $\phi_n(u_n)=u_{n+1}.$ Fix a real scalar $\lambda\in\R$, and define
\[
B_n:=\lambda\,\pi_n(u_n)[D_n^\ell,\pi_n(u_n)^*]\in \mathcal B(H_n).
\]
This is a Connes one-form associated to the spectral triple
$(A_n,H_n,D_n^\ell)$, since $u_n,u_n^*\in A_{\theta_n}^{\infty}(D_n^\ell)$
and $B_n=\lambda\,\pi_n(u_n)[D_n^\ell,\pi_n(u_n)^*].$

We first check that $B_n$ is bounded and selfadjoint. Boundedness follows from
the bounded translation increments of $\ell_n$. Indeed, $\pi_n(u_n)^*$
shifts the Fourier label by $-h_n$, and $[D_n^\ell,\pi_n(u_n)^*]\delta_x^{(n)}$
is a scalar multiple of $\bigl(\ell_n(x-h_n)-\ell_n(x)\bigr)\delta_{x-h_n}^{(n)}.$
But $\ell_n(x-h_n)-\ell_n(x)=\ell_N\left(\frac{x}{N^{n-1}}-h\right)-\ell_N\left(\frac{x}{N^{n-1}}\right),$
and the bounded-translation-increment property of $\ell_N$ gives
\[
\sup_{x\in\Z^2}
\left|
\ell_n(x-h_n)-\ell_n(x)
\right|
\leq |h|+h_N(h)<\infty.
\]
Since $h=(1,0)$, one has $h_N(h)=0$, so this bound is independent of $n$.
Consequently, $\sup_n\|B_n\|<\infty.$

For selfadjointness, set $U_n:=\pi_n(u_n)$. Then $U_n[D_n^\ell,U_n^*]=U_nD_n^\ell U_n^*-D_n^\ell$
as a bounded operator. Hence $\bigl(U_n[D_n^\ell,U_n^*]\bigr)^*=(U_nD_n^\ell U_n^*-D_n^\ell)^*=U_nD_n^\ell U_n^*-D_n^\ell.$
Thus $U_n[D_n^\ell,U_n^*]$, and hence $B_n$, is selfadjoint.

The perturbations are compatible with the inductive system. Let $I_n:H_n\to H_{n+1}$ 
be the GNS isometry induced by $\phi_n$. Since $\phi_n(u_n)=u_{n+1},$ we have $I_n\pi_n(u_n)=\pi_{n+1}(u_{n+1})I_n.$
Also, by the construction of the length triples, $I_nD_n^\ell=D_{n+1}^\ell I_n$ on $\dom(D_n^\ell)$.
Hence $
I_n[D_n^\ell,\pi_n(u_n)^*]=[D_{n+1}^\ell,\pi_{n+1}(u_{n+1})^*]I_n$ on $\dom(D_n^\ell)$,
and therefore $I_nB_n=B_{n+1}I_n.$

Define the perturbed finite-stage operators
\[
D_n^\lambda:=D_n^\ell+B_n,
\qquad
\dom(D_n^\lambda):=\dom(D_n^\ell).
\]
Since $B_n$ is bounded and selfadjoint, $D_n^\lambda$ is selfadjoint by the
Kato--Rellich theorem, and it has compact resolvent because it is a bounded
selfadjoint perturbation of $D_n^\ell$. Moreover, $I_nD_n^\lambda=D_{n+1}^\lambda I_n.$
The spectral algebra is unchanged:
\[
A_{\theta_n}^{\infty}(D_n^\lambda)
=
A_{\theta_n}^{\infty}(D_n^\ell),
\]
because for every $a\in A_{\theta_n}$, $[D_n^\lambda,\pi_n(a)]=[D_n^\ell,\pi_n(a)]+[B_n,\pi_n(a)],$
and the second term is bounded since \(B_n\) and \(\pi_n(a)\) are bounded. Conversely,
$[D_n^\ell,\pi_n(a)]=[D_n^\lambda,\pi_n(a)]-[B_n,\pi_n(a)].$
Thus the spectral-algebra inclusion \eqref{eq:spec-alg} for the length triples
also holds for the perturbed triples. Consequently the triples $(A_{\theta_n},H_n,D_n^\lambda)$
form an inductive system of spectral triples. 

Let $D_{\ex}^{\ell_N}$ denote the compact-resolvent length operator on the limit algebra
$A_{\ex}^{\theta,N}$, so that $D_{\ex}^{\ell_N}\delta_g=\ell_N(g)\delta_g,$ for $ g\in\Gamma_N.$
The compatible perturbations $B_n$ define a bounded selfadjoint operator $B=\varinjlim B_n$
on the Hilbert-space limit. Explicitly,
\[
B
=
\lambda\,\pi(W_h)[D_{\ex}^{\ell_N},\pi(W_h)^*].
\]
Thus the inductive-limit perturbed operator is
\[
D^\lambda=D_{\ex}^{\ell_N}+B.
\]
Since $B$ is bounded and selfadjoint, $D^\lambda$ is a bounded selfadjoint
perturbation of $D_{\ex}^{\ell_N}$. Therefore $D^\lambda$ has compact
resolvent. Moreover, for every $a$ in the algebraic twisted group algebra $\operatorname{span}\{W_g:g\in\Gamma_N\},$
one has $[D^\lambda,\pi(a)]=[D_{\ex}^{\ell_N},\pi(a)]+[B,\pi(a)],$
and both terms are bounded. Hence $(A_{\ex}^{\theta,N},L^2(A_{\ex}^{\theta,N},\tau_{\theta,N}),D^\lambda)$
is a spectral triple on the unital toric $N$-solenoid. Its dense spectral
algebra contains $\operatorname{span}\{W_g:g\in\Gamma_N\}.$

Thus, for $\lambda\neq0$, the compact-resolvent length triple on the
solenoid admits nontrivial compatible inner fluctuations coming from
finite-stage Connes one-forms.
\end{example}

\begin{remark}[Diagonal bounded perturbations on the solenoid]
\label{rem:solenoid-diagonal-bounded-perturbations}
The preceding construction also admits purely diagonal bounded perturbations.
Let $q\colon \Gamma_N\to\R$ be a bounded function. Then $q$ automatically has
bounded translation increments, and set $q_n(x):=q\left(\frac{x}{N^{n-1}}\right),$ for all
$x\in\Z^2.$
Then $\ell_n+q_n$ is compatible: $(\ell_{n+1}+q_{n+1})(Nx)=(\ell_n+q_n)(x).$
Indeed, $
\ell_{n+1}(Nx)=\ell_n(x)$ and $q_{n+1}(Nx)=q\left(\frac{Nx}{N^n}\right)
=q\left(\frac{x}{N^{n-1}}\right)=q_n(x).$

Each finite-stage function $\ell_n+q_n$ has finite spectral multiplicity in
bounded intervals. This follows because $q_n$ is uniformly bounded and
$\ell_n$ is proper. More precisely, if $|\ell_n(x)+q_n(x)|\le R,$
then $\ell_n(x)\le R+\|q\|_\infty,$
and only finitely many $x\in\Z^2$ satisfy this inequality. Also
$\ell_n+q_n$ has bounded translation increments, since both $\ell_n$ and
$q_n$ do.

The spectral-algebra inclusion \eqref{eq:spec-alg} follows by the same coset
argument as for $\ell_n$. On a coset $r+N\Z^2$, the difference between the
target multiplier and the transported source multiplier is
$(\ell_N+q)\left(\frac{r}{N^n}+\frac{x}{N^{n-1}}\right)-(\ell_N+q)\left(\frac{x}{N^{n-1}}\right),$
which is bounded uniformly in $x$, because both $\ell_N$ and $q$ have
bounded translation increments. There are only finitely many cosets.

Therefore the finite-stage diagonal operators
\[
D_n^q\delta_x^{(n)}
=
(\ell_n(x)+q_n(x))\delta_x^{(n)}
\]
form an inductive system of spectral triples. The limit operator is
\[
D^q\delta_g=(\ell_N(g)+q(g))\delta_g,
\qquad g\in\Gamma_N.
\]
Since $q$ is bounded, $D^q$ is a bounded selfadjoint perturbation of
$D_{\ex}^{\ell_N}$. Hence it defines another compact-resolvent spectral triple
on $A_{\ex}^{\theta,N}$.
\end{remark}

\begin{example}[A compatible right-conformal deformation on the unital toric \(N\)-solenoid]
\label{ex:solenoid-right-conformal}
Keep the notation of Example~\ref{ex:compact-resolvent-solenoid-dirac}. Thus $A_{\ex}^{\theta,N}=\varinjlim(A_{\theta_n},\phi_n),$ where
$\theta_n=\frac{\theta}{N^{2(n-1)}},$
and $\phi_n(U_n)=U_{n+1}^N,$
$\phi_n(V_n)=V_{n+1}^N.$ Let $D_n^\ell\delta_x^{(n)}=\ell_n(x)\delta_x^{(n)}$
be the finite-stage length operator constructed in
Example~\ref{ex:compact-resolvent-solenoid-dirac}. Its inductive-limit operator
is denoted $D_{\ex}^{\ell_N}.$

Fix $\varepsilon\in\mathbb R$ with $0<|\varepsilon|<1/2$, 
and set $k_1:=1+\varepsilon(U_1+U_1^*)\in A_{\theta_1}^{\mathrm{alg}}.$
Then $k_1=k_1^*$, and $k_1$ is positive and invertible, since $\operatorname{Spec}(U_1+U_1^*)\subseteq[-2,2],$
so $\operatorname{Spec}(k_1)\subseteq [1-2|\varepsilon|,1+2|\varepsilon|]
\subset (0,\infty).$
For \(n\geq1\), define
\[
k_n:=\phi_{1,n}(k_1)
=
1+\varepsilon\left(U_n^{N^{n-1}}+U_n^{-N^{n-1}}\right)
\in A_{\theta_n}^{\mathrm{alg}},
\]
where $\phi_{1,n}:=\phi_{n-1}\circ\cdots\circ\phi_1$ and $\phi_{1,1}=\id$. Then $\phi_n(k_n)=k_{n+1}.$
Each $k_n$ is positive and invertible.

Let $R_{k_n}:H_n\to H_n$ denote right multiplication by $k_n$, that is $R_{k_n}\Lambda_n(a):=\Lambda_n(ak_n),$ for all $a\in A_{\theta_n}$.
The operator $R_{k_n}$ is bounded, selfadjoint, positive, and invertible,
with $\|R_{k_n}\|\leq \|k_n\|$ and inverse
$R_{k_n}^{-1}=R_{k_n^{-1}}$. Moreover, $R_{k_n}$ commutes with the left regular
representation of $A_{\theta_n}$: $R_{k_n}\pi_n(a)\Lambda_n(b)=\Lambda_n(abk_n)=\pi_n(a)R_{k_n}\Lambda_n(b).$

We define the right-conformally deformed Dirac operator
\[
D_n^{\mathrm{conf}}
:=
R_{k_n}D_n^\ell R_{k_n},
\]
with domain $\dom(D_n^{\mathrm{conf}}):=R_{k_n}^{-1}\dom(D_n^\ell).$
It is clear that $D_n^{\mathrm{conf}}$ is selfadjoint. Moreover, $D_n^{\mathrm{conf}}$ has compact resolvent. To see this, note that the map $R_{k_n}\colon \dom(D_n^{\mathrm{conf}})\to \dom(D_n^\ell)$
is a bounded linear bijection. Moreover, the graph norm of
$D_n^{\mathrm{conf}}$ is equivalent, via $R_{k_n}$, to the graph norm of
$D_n^\ell$. Since the embedding $\dom(D_n^\ell)\hookrightarrow H_n$
is compact, the embedding $\dom(D_n^{\mathrm{conf}})\hookrightarrow H_n$
is also compact. Hence $(1+(D_n^{\mathrm{conf}})^2)^{-1/2}$
is compact.

The spectral algebra is unchanged. Let $a\in A_{\theta_n}$. Since
$R_{k_n}$ commutes with $\pi_n(a)$, on the natural common core one has $[D_n^{\mathrm{conf}},\pi_n(a)]=R_{k_n}[D_n^\ell,\pi_n(a)]R_{k_n}.$
Therefore $[D_n^{\mathrm{conf}},\pi_n(a)]$ extends to a bounded operator if and only if $[D_n^\ell,\pi_n(a)]$
extends to a bounded operator. Indeed, one implication follows from boundedness of
$R_{k_n}$, and the other follows by multiplying by the bounded inverse $R_{k_n}^{-1}$. Thus $A_{\theta_n}^{\infty}(D_n^{\mathrm{conf}})
=A_{\theta_n}^{\infty}(D_n^\ell).$
In particular, $A_{\theta_n}^{\mathrm{alg}}\subseteq A_{\theta_n}^{\infty}(D_n^{\mathrm{conf}}),$
so $(A_{\theta_n},H_n,D_n^{\mathrm{conf}})$
is a spectral triple.

The conformal factors are compatible with the GNS embeddings. Let $I_n:H_n\to H_{n+1}$
be the GNS isometry induced by $\phi_n$. Since $\phi_n(k_n)=k_{n+1},$
for $a\in A_{\theta_n}$ we have $I_nR_{k_n}\Lambda_n(a)=
I_n\Lambda_n(ak_n)
=\Lambda_{n+1}(\phi_n(a)\phi_n(k_n))=\Lambda_{n+1}(\phi_n(a)k_{n+1})
=R_{k_{n+1}}I_n\Lambda_n(a).$
Hence $I_nR_{k_n}=R_{k_{n+1}}I_n.$ We also have $I_nD_n^\ell=D_{n+1}^\ell I_n.$

We now check the domain-level intertwining for the conformally deformed
operators. Let $\xi\in\dom(D_n^{\mathrm{conf}}).$
Then $R_{k_n}\xi\in\dom(D_n^\ell).$
Using the compatibility $I_nR_{k_n}=R_{k_{n+1}}I_n$, we get $R_{k_{n+1}}I_n\xi=I_nR_{k_n}\xi\in \dom(D_{n+1}^\ell).$
Thus $I_n\xi\in R_{k_{n+1}}^{-1}\dom(D_{n+1}^\ell)=\dom(D_{n+1}^{\mathrm{conf}}).$
Moreover, $
I_nD_n^{\mathrm{conf}}\xi=I_nR_{k_n}D_n^\ell R_{k_n}\xi=R_{k_{n+1}}I_nD_n^\ell R_{k_n}\xi=
R_{k_{n+1}}D_{n+1}^\ell I_nR_{k_n}\xi=
R_{k_{n+1}}D_{n+1}^\ell R_{k_{n+1}}I_n\xi=
D_{n+1}^{\mathrm{conf}}I_n\xi,$
so $I_nD_n^{\mathrm{conf}}
=
D_{n+1}^{\mathrm{conf}}I_n$
on $\dom(D_n^{\mathrm{conf}})$.

Since $A_{\theta_n}^{\infty}(D_n^{\mathrm{conf}})=A_{\theta_n}^{\infty}(D_n^\ell)$
and since the length triples satisfy the spectral-algebra
inclusion, the conformally deformed triples also satisfy it. Hence the triples $(A_{\theta_n},H_n,D_n^{\mathrm{conf}})$
form an inductive system of spectral triples.

Let $h:=(1,0)\in\Gamma_N,$
and let $W_h$ denote the Fourier unitary of the limiting twisted group
algebra. Define $k_\infty:=1+\varepsilon(W_h+W_h^*)\in A_{\ex}^{\theta,N}.$
Then $k_\infty$ is the inductive limit of the $k_n$'s, and it is positive
and invertible. Under the identification $L^2(A_{\ex}^{\theta,N},\tau)\cong \ell^2(\Gamma_N),$
the inductive-limit conformal factor is right multiplication by $k_\infty$,
denoted $R_{k_\infty}.$
The inductive-limit operator is
\[
D^{\mathrm{conf}}
=
R_{k_\infty}D_{\ex}^{\ell_N}R_{k_\infty}.
\]

It has compact resolvent by the same graph-norm argument as above, since $R_{k_\infty}$ is bounded, positive, and invertible. Furthermore, for every
$a$ in the algebraic twisted group algebra $\operatorname{span}\{W_g:g\in\Gamma_N\},$
one has $[D^{\mathrm{conf}},\pi(a)]=R_{k_\infty}[D_\ex^{\ell_N},\pi(a)]R_{k_\infty},$
which is bounded. Therefore $(A_{\ex}^{\theta,N},L^2(A_{\ex}^{\theta,N},\tau),D^{\mathrm{conf}})$
is a compact-resolvent spectral triple on the unital toric $N$-solenoid.

This example is a genuine inductive-limit conformal construction: the
finite-stage conformal factors $k_n$ are noncentral, compatible under the
connecting maps, and the conformally deformed Dirac operators intertwine
exactly under the GNS isometries.
\end{example}

\subsection{The nonunital cases: Morita-compatible locally compact triples}
\label{subsec:nonunital-morita-compatible}

We now turn to Cases~\textup{(2)}, \textup{(3)}, and \textup{(4)}. The
inductive-limit machinery of \cite{FloricelGhorbanpour} used in the previous
subsection is formulated for unital spectral triples and unital connecting maps,
so it does not apply verbatim to the present nonunital systems. Nevertheless,
the same Hilbert-space and operator constructions still make sense after replacing the ordinary GNS embeddings by trace-rescaled GNS isometries.

Recall that a locally compact, or nonunital, spectral triple on a $C^*$-algebra
$A$ consists of a faithful nondegenerate representation $\pi\colon A\to\mathcal B(H),$
a selfadjoint operator $D$ on \(H\), and a dense $*$-subalgebra $\mathcal A\subseteq A$
such that $[D,\pi(a)]$ is densely defined and extends to a bounded operator for every $a\in\mathcal A$,
and
\[
\pi(a)(1+D^2)^{-1/2}\in\mathcal K(H)
\qquad
(a\in\mathcal A).
\]
Equivalently, since $\mathcal A$ is dense in $A$, the local compactness
condition then extends to every $a\in A$. This is the standard nonunital
version of the spectral triple condition; see
\cite{Rennie, CareyGayralRennieSukochevLocallyCompact, CareyGayralPhillipsRennieSukochevNonunitalSF}.

\begin{definition}[Morita-compatible Dirac structure]
\label{def:morita-compatible-dirac-data}
Let $B_1\xrightarrow{\ \phi_1\ }B_2\xrightarrow{\ \phi_2\ }B_3\xrightarrow{}\cdots$
be a protoral inductive system of Cases~\textup{(2)}--\textup{(4)}, with $B_n=M_{r_n}(A_{\Theta_n}).$
A Morita-compatible Dirac structure for this system consists of
the following data.
\smallskip

\noindent
\textup{(i)} For each $n$, a smooth factorization of the connecting map $\phi_n$ of the form
\[
\begin{tikzcd}[column sep=4.2em, row sep=3.6em]
B_n
  \arrow[r, "\psi_n"]
  \arrow[rrr, bend right=18, swap, "\phi_n"]
&
C_n
  \arrow[r, "\alpha_n"]
&
p_nB_{n+1}p_n
  \arrow[r, "\iota_{p_n}"]
&
B_{n+1}\end{tikzcd}
\]
where \begin{enumerate}
\item[\textup{($i_a$)}] $C_n$ is an
intermediate matrix-amplified noncommutative torus;
\item[($i_b$)] $\psi_n$ is either the identity or a unital toric map;
\item[($i_c$)] $p_n\in B_{n+1}^\infty$
is a smooth full projection;
\item[\textup{($i_d$)}] $\alpha_n\colon C_n\to p_nB_{n+1}p_n$ is a
$C^*$-isomorphism whose restriction gives a Fréchet $*$-isomorphism $\alpha_n\colon C_n^\infty\xrightarrow{\cong}p_nB_{n+1}^\infty p_n$; and
\item[\textup{($i_e$)}] $\iota_{p_n}$ is the corner inclusion.
\end{enumerate}
\smallskip
\noindent
\textup{(ii)} Hilbert spaces $H_n:=L^2(B_n,\tau_n)\otimes S_n$, $H_n^C:=L^2(C_n,\omega_n)\otimes S_n^C$, and $H_{n+1}^{(p_n)}
:=
L^2(p_nB_{n+1}p_n,\tau_{p_n})\otimes S_{n+1}^{(p_n)}$, where $S_n$,  $S_n^C$, and $S_{n+1}^{(p_n)}$ are finite dimensional Hilbert spaces, $\tau_n$ is the normalized trace on $B_n$,  $\omega_n$ is the normalized trace on $C_n$, and $\tau_{p_n}(x)
:=
\frac{\tau_{n+1}(x)}{\tau_{n+1}(p_n)}$
is the normalized trace on the corner $p_nB_{n+1}p_n$.

\smallskip
\noindent
\textup{(iii)} Selfadjoint
operators
$D_n$
on $H_n$, $D_n^C$ on $H_n^C$, and $D_{n+1}^{(p_n)}$ on $H_{n+1}^{(p_n)}$.

\smallskip
\noindent
\textup{(iv)} Isometries $I_{\psi_n}:H_n\to H_n^C$, $I_{\alpha_n}\colon H_n^C\to H_{n+1}^{(p_n)}$, and  $I_{\iota_{p_n}}:H_{n+1}^{(p_n)}\to H_{n+1}$ that are compatible with given structure, in the sense that they
\begin{enumerate}
\item[($iv_a$)] intertwine the corresponding selfadjoint operators, i.e.,  $I_{\psi_n}(\dom (D_n))\subseteq \dom (D_n^C)$ and $D_n^C I_{\psi_n}=I_{\psi_n}D_n$ on $\dom (D_n)$, $I_{\alpha_n}(\dom (D_n^C))\subseteq \dom (D_{n+1}^{(p_n)})$ and 
$D_{n+1}^{(p_n)}I_{\alpha_n}=I_{\alpha_n}D_n^C$ on $\dom (D_n^C)$, and $I_{\iota_{p_n}}(\dom (D_{n+1}^{(p_n)}))
\subseteq
\dom (D_{n+1})$
and $D_{n+1}I_{\iota_{p_n}}
=
I_{\iota_{p_n}}D_{n+1}^{(p_n)}$
on $\dom (D_{n+1}^{(p_n)})$;

\item[($iv_b$)] intertwine the corresponding faithful representations $\pi_n\colon B_n\to\mathcal B(H_n)$, 
$\pi_n^C\colon C_n\to\mathcal B(H_n^C)$, and
$\pi_{n+1}^{(p_n)}\colon p_nB_{n+1}p_n
   \to\mathcal B(H_{n+1}^{(p_n)})$ given by left multiplications, i.e.,
$I_{\psi_n}\pi_n(b)
=
\pi_n^C(\psi_n(b))I_{\psi_n}$, for $b\in B_n$, $I_{\alpha_n}\pi_n^C(c)
=
\pi_{n+1}^{(p_n)}(\alpha_n(c))I_{\alpha_n}$,
for $c\in C_n$,
and $I_{\iota_{p_n}}\pi_{n+1}^{(p_n)}(x)
=
\pi_{n+1}(\iota_{p_n}(x))I_{\iota_{p_n}}$, for
$x\in p_nB_{n+1}p_n$.

\item[($iv_c$)] their composition $I_n
:=
I_{\iota_{p_n}}\circ I_{\alpha_n}\circ I_{\psi_n}
:
H_n\to H_{n+1}$
is exactly the trace-rescaled GNS isometry associated to $\phi_n=\iota_{p_n}\circ\alpha_n\circ\psi_n,$ that is
$I_n(\Lambda_n(a)\otimes\xi)
=
t_n^{-1/2}\Lambda_{n+1}(\phi_n(a))\otimes J_n\xi,$ for all $a\in B_n$ and $\xi\in S_n$, where $t_n=\tau_{n+1}(\phi_n(1_{B_n}))$ and $
J_n\colon S_n\to S_{n+1}$
is a chosen spinor isometry.
\end{enumerate} 

\end{definition}

Given a Morita-compatible Dirac structure on a protoral inductive system $B_1\xrightarrow{\ \phi_1\ }B_2\xrightarrow{\ \phi_2\ }B_3\xrightarrow{}\cdots$ built from Cases~\textup{(2)}, \textup{(3)}, and
\textup{(4)}, the trace-rescaled GNS maps $I_n\colon H_n\to H_{n+1}$
satisfy $I_n(\dom (D_n))\subseteq \dom (D_{n+1})$,
$D_{n+1}I_n=I_nD_n$, and $I_n\pi_n(b)=\pi_{n+1}(\phi_n(b))I_n$, ($b\in B_n$), for every $n$. 

Let now $H_{\ex}:=\varinjlim(H_n,I_n)$ be the Hilbert-space inductive limit, with canonical isometries $I_{n,\infty}:H_n\to H_\ex$. As in the unital case, let $\mathcal D_0:=\bigcup_{n\geq1}I_{n,\infty}(\dom (D_n))\subseteq H_\ex.$
For $\xi\in\dom (D_n)$, set \[D_0I_{n,\infty}\xi:=I_{n,\infty}D_n\xi.\]
The operator $D_0$ is symmetric and essentially selfadjoint, and set
\[
D_\ex:=\overline{D_0}.
\]
We notice that  for each $m$, the closed subspace
$I_{m,\infty}H_m\subseteq H_{\ex}$ reduces $D_{\ex}$, and $\dom(D_{\ex})\cap I_{m,\infty}H_m=I_{m,\infty}\dom(D_m).$
Indeed, for $\eta\in H_m$,
$(D_{\ex}\pm i)^{-1}I_{m,\infty}\eta=I_{m,\infty}(D_m\pm i)^{-1}\eta,$
which follows from the defining relation $(D_0\pm i)I_{m,\infty}\xi=I_{m,\infty}(D_m\pm i)\xi.$
Thus the resolvents preserve $I_{m,\infty}H_m$, and the asserted domain
identity follows.

We next construct a faithful nondegenerate
representation $\pi\colon A_{\ex}\to\mathcal B(H_{\ex}).$ For $b\in B_n$ and a
vector represented at stage $m$, say $I_{m,\infty}\eta$, choose
$k\ge n,m$. Define
\[
\pi(\iota_n(b))I_{m,\infty}\eta:=I_{k,\infty}\pi_k(\phi_{n,k}(b))I_{m,k}\eta,
\]
where $I_{m,k}:=I_{k-1}\cdots I_m$. The representation-intertwining relation
shows that this definition is independent of the choice of \(k\) and of the
representatives. It therefore gives a representation of the algebraic inductive
limit, which extends by continuity to a representation $\pi\colon A_{\ex}\to\mathcal B(H_{\ex}).$
Since the connecting maps are injective and the finite-stage left regular
representations are faithful, $\pi$ is faithful. It is nondegenerate because
the projections $\pi(\iota_n(1_{B_n}))$
increase strongly to the identity on the dense subspace
$\bigcup_n I_{n,\infty}H_n$.

\begin{theorem}
\label{thm:protorus-spectral-general}
Let $B_1\xrightarrow{\ \phi_1\ }B_2\xrightarrow{\ \phi_2\ }B_3\xrightarrow{}\cdots$
be a protoral inductive system built from Cases~\textup{(2)}, \textup{(3)}, and
\textup{(4)}, and assume that it is equipped with Morita-compatible Dirac structure
in the sense of Definition~\ref{def:morita-compatible-dirac-data}. 
\smallskip

\noindent
 
 \emph{(i)}  Let $a\in B_n^\infty$, and put $a_m:=\phi_{n,m}(a)\in B_m^\infty,$ for all $m\ge n.$
Then the commutator $[D_{\ex},\pi(\iota_n(a))]$
extends to a bounded operator on $H_{\ex}$ if and only if the stagewise
commutators $[D_m,\pi_m(a_m)]$
extend boundedly and satisfy
\[
\sup_{m\ge n}
\left\|
\overline{[D_m,\pi_m(a_m)]}
\right\|<\infty.
\]
In that case, on the dense subspace $\bigcup_{m\ge n}I_{m,\infty}H_m\subseteq H_{\ex},$
one has $[D_{\ex},\pi(\iota_n(a))]I_{m,\infty}=I_{m,\infty}\,\overline{[D_m,\pi_m(a_m)]}.$

 \smallskip

\noindent
 
 \emph{(ii)} Set $e_n:=\iota_n(1_{B_n})\in A_{\ex},$ 
$P_n:=\pi(e_n)\in\mathcal B(H_{\ex}).$ Let $\mathcal A\subseteq A_{\ex}$ be a dense $*$-subalgebra
 such that each $a\in\mathcal A$ is supported in some finite stage, i.e. there exists $n$
such that $a=e_nae_n$. Assume also that $e_n\in\mathcal A$ for every $n$. Then the local
compactness condition
\[
\pi(a)(1+D_{\ex}^2)^{-1/2}\in\mathcal K(H_{\ex}),
\qquad a\in\mathcal A,
\]
holds if and only if $P_n(1+D_{\ex}^2)^{-1/2}\in\mathcal K(H_{\ex})$ for every $n\geq1$.

Consequently, if the equivalent conditions in \textup{(i)} hold for every
$a\in\mathcal A$, and if the local compactness criterion in \textup{(ii)}
holds, then $(\mathcal A,H_{\ex},D_{\ex})$
is a locally compact spectral triple on $A_{\ex}$.

\end{theorem}

\begin{proof}
We prove \textup{(i)}. Fix \(a\in B_n^\infty\), and put $T_m^0:=[D_m,\pi_m(a_m)]$
on $\dom (D_m)$. Suppose first that each $T_m^0$ extends to a bounded
operator $T_m$, and that $\sup_{m\ge n}\|T_m\|<\infty$. The representation
and Dirac intertwining give $I_m\pi_m(a_m)=\pi_{m+1}(a_{m+1})I_m$
and $D_{m+1}I_m=I_mD_m.$
Therefore, on $\dom (D_m)$, $T_{m+1}^0I_m=I_mT_m^0.$
By boundedness, this extends to $T_{m+1}I_m=I_mT_m.$
Thus the formula
\[
T I_{m,\infty}\xi:=I_{m,\infty}T_m\xi,
\qquad \xi\in H_m,
\]
defines a bounded operator $T$ on $H_{\ex}$. On the core
$\bigcup_{m\ge n}I_{m,\infty}(\dom (D_m))$, one computes
\[
\begin{aligned}
[D_{\ex},\pi(\iota_n(a))]I_{m,\infty}\xi
&=
D_{\ex}I_{m,\infty}\pi_m(a_m)\xi
-
\pi(\iota_n(a))I_{m,\infty}D_m\xi  \\
&=
I_{m,\infty}D_m\pi_m(a_m)\xi
-
I_{m,\infty}\pi_m(a_m)D_m\xi \\
&=
I_{m,\infty}[D_m,\pi_m(a_m)]\xi \\
&=
T I_{m,\infty}\xi.
\end{aligned}
\]
It remains only to note that $\pi(\iota_n(a))$ preserves
$\dom(D_{\ex})$. Let $\zeta\in\dom(D_{\ex})$, and choose
$\zeta_j$ in the core $\bigcup_{m\ge n}I_{m,\infty}\dom(D_m)$
such that $\zeta_j\to\zeta$ and $D_{\ex}\zeta_j\to D_{\ex}\zeta.$
On the core we have $D_{\ex}\pi(\iota_n(a))\zeta_j=\pi(\iota_n(a))D_{\ex}\zeta_j+T\zeta_j.$
The right-hand side converges to $\pi(\iota_n(a))D_{\ex}\zeta+T\zeta.$
Since $\pi(\iota_n(a))\zeta_j\to\pi(\iota_n(a))\zeta$ and
$D_{\ex}$ is closed, it follows that $\pi(\iota_n(a))\zeta\in\dom(D_{\ex})$
and $[D_{\ex},\pi(\iota_n(a))]\zeta=T\zeta.$ Hence $[D_{\ex},\pi(\iota_n(a))]$ extends to $T$.

Conversely, suppose that $[D_{\ex},\pi(\iota_n(a))]$ extends to a bounded
operator $T$. Fix $m\ge n$, and let $\xi\in\dom(D_m)$. Since $I_{m,\infty}\xi\in\dom(D_{\ex})$
and since the global commutator is bounded, $\pi(\iota_n(a))I_{m,\infty}\xi=I_{m,\infty}\pi_m(a_m)\xi$
belongs to $\dom(D_{\ex})$. By the domain identity $\dom(D_{\ex})\cap I_{m,\infty}H_m=I_{m,\infty}\dom(D_m),$
we get $\pi_m(a_m)\xi\in\dom(D_m).$
Moreover, $I_{m,\infty}[D_m,\pi_m(a_m)]\xi=T I_{m,\infty}\xi.$
Since $I_{m,\infty}$ is an isometry, the stagewise commutator extends
boundedly and \[\|\overline{[D_m,\pi_m(a_m)]}\|\leq \|T\|.\]
Thus the uniform boundedness condition holds. This proves \textup{(i)}.

We prove \textup{(ii)}. First assume that local compactness holds on
$\mathcal A$. Since $e_n=\iota_n(1_{B_n})\in\mathcal A,$
we get $P_n(1+D_{\ex}^2)^{-1/2}=\pi(e_n)(1+D_{\ex}^2)^{-1/2}\in\mathcal K(H_{\ex})$
for every $n$.

Conversely, assume $P_n(1+D_{\ex}^2)^{-1/2}\in\mathcal K(H_{\ex})$
for every $n$. Let $a\in\mathcal A$. Then $a=e_na=ae_n$ for some $n\ge1$,
so $\pi(a)=\pi(a)P_n.$ Therefore
\[
\pi(a)(1+D_{\ex}^2)^{-1/2}
=
\pi(a)P_n(1+D_{\ex}^2)^{-1/2}.
\]
The right-hand side is compact. This proves the local-compactness criterion.
If the commutator criterion in \textup{(i)} also holds for all elements of
$\mathcal A$, then the locally compact spectral-triple axioms
hold on the dense algebra $\mathcal A$. Hence
$(\mathcal A,H_{\ex},D_{\ex})$ is a locally compact spectral triple over
$A_{\ex}$.\end{proof}

\begin{example}[A trace-GNS locally compact triple on the pure-corner stable protorus]
\label{ex:stable-corner-trace-gns-spectral}
Let $A_{\ex}^{\Theta}=\varinjlim(B_n,\phi_n)\cong A_\Theta\otimes\K$
be the pure-corner stable protorus of Example~\ref{ex:stable-corner}, where
\(\Theta\in M_d(\R)\) is a fixed nondegenerate skew-symmetric matrix. Thus $m_n:=2^{n-1},$
$B_n:=B_n^\Theta=M_{m_n}(A_\Theta),$
and $\phi_n(a):=\phi_n^\Theta(a)=
\begin{pmatrix}
a&0\\
0&0
\end{pmatrix}
\in
B_{n+1}.$
Let $\tau_n:=\tau_n^\Theta=\frac1{m_n}\Tr_{m_n}\otimes\tau_\Theta$
be the normalized trace on $B_n$. Then $\tau_{n+1}\circ\phi_n=\frac12\,\tau_n,$
so the trace-rescaled GNS isometry is
\[
I_n\Lambda_n(a)
=
\sqrt2\,\Lambda_{n+1}
\begin{pmatrix}
a&0\\
0&0
\end{pmatrix}.
\]
Let $S_d$ be a complex Clifford module with selfadjoint generators $\gamma_1^{(d)},\ldots,\gamma_d^{(d)}$
satisfying $\gamma_i^{(d)}\gamma_j^{(d)}+\gamma_j^{(d)}\gamma_i^{(d)}=2\delta_{ij}1_{S_d}.$
Set $H_\Theta:=L^2(A_\Theta,\tau_\Theta)\otimes S_d,$ and let $D_\Theta^{\mathrm{std}}:=\sum_{j=1}^d P_j^\Theta\otimes\gamma_j^{(d)}
$ be the standard flat Dirac operator on  $H_\Theta$. Here
$P_j^\Theta=-i\delta_j^\Theta$, where $\delta_j^\Theta(U^x)=2\pi i\,x_jU^x.$
Equivalently,
\[
D_\Theta^{\mathrm{std}}(\delta_x\otimes\xi)=2\pi\,\delta_x\otimes\gamma^{(d)}(x)\xi,
\qquad
\gamma^{(d)}(x)=\sum_{j=1}^d x_j\gamma_j^{(d)}.
\]
This is the flat operator $D_{\Theta,I_d}$ from
Observation~\ref{rem:unital-special-spectral-triples}\textup{(1)}.

We now describe the trace-GNS Hilbert spaces explicitly. Let $\mathcal H_n:=L^2(B_n,\tau_n)\otimes S_d\otimes\C^2.$
Using the matrix units $E_{ij}^{(n)}$ of $M_{m_n}(\C)$, identify
\[
L^2(B_n,\tau_n)
\cong
L^2(A_\Theta,\tau_\Theta)
\otimes
\C^{m_n}
\otimes
\overline{\C^{m_n}}
\]
by sending the orthonormal vector $\sqrt{m_n}\,\Lambda_n(E_{ij}^{(n)}\otimes U^x)$
to $\delta_x\otimes e_i\otimes\overline{e_j}.$ Here $\overline{\C^{m_n}}$ denotes the conjugate Hilbert space.
Under this identification, the trace-rescaled GNS isometry $I_n$ is simply the
standard inclusion
\[
L^2(A_\Theta,\tau_\Theta)\otimes\C^{m_n}\otimes\overline{\C^{m_n}}
\hookrightarrow
L^2(A_\Theta,\tau_\Theta)\otimes\C^{2m_n}\otimes\overline{\C^{2m_n}}
\]
onto the upper-left block.

After the harmless tensor rearrangement
$\mathcal H_n \cong H_\Theta\otimes \C^{m_n}\otimes\overline{\C^{m_n}}\otimes\C^2$, let $N_n$ be the diagonal number operator on $\overline{\C^{m_n}}$, $N_n\overline{e_j}=j\,\overline{e_j},$ for 
$1\leq j\leq m_n$. Let $\sigma_1,\sigma_2$ be selfadjoint Pauli matrices  on $\C^2$ satisfying $\sigma_1^2=\sigma_2^2=1$ and
$\sigma_1\sigma_2+\sigma_2\sigma_1=0.$ We then define
\[
\mathcal D_n:=
D_\Theta^{\mathrm{std}}
\otimes1_{\C^{m_n}}
\otimes1_{\overline{\C^{m_n}}}
\otimes\sigma_1+1_{H_\Theta}
\otimes1_{\C^{m_n}}
\otimes N_n
\otimes\sigma_2.
\]

Equivalently, on basis vectors,
\[
\mathcal D_n
(\delta_x\otimes\xi\otimes e_i\otimes\overline{e_j}\otimes\zeta)
=
\delta_x\otimes e_i\otimes\overline{e_j}
\otimes
\left(
2\pi\,\gamma^{(d)}(x)\otimes\sigma_1
+
j\,1_{S_d}\otimes\sigma_2
\right)(\xi\otimes\zeta),
\]
for $\xi\in S_d$ and $\zeta\in\C^2$.

Each $\mathcal D_n$ is selfadjoint and has compact resolvent. Indeed, the two
summands anticommute because $\sigma_1\sigma_2+\sigma_2\sigma_1=0$, and
therefore $\mathcal D_n^2=(D_\Theta^{\mathrm{std}})^2\otimes1\otimes1\otimes1+1\otimes1\otimes N_n^2\otimes1.$
Since $D_\Theta^{\mathrm{std}}$ has compact resolvent and $N_n$ acts on a
finite-dimensional space, $\mathcal D_n$ has compact resolvent.

The data define a Morita-compatible Dirac structure in the sense of
Definition~\ref{def:morita-compatible-dirac-data}. Indeed, the algebraic
factorization is the pure-corner factorization
\[
B_n
\xrightarrow{\ \psi_n=\id\ }
C_n:=B_n
\xrightarrow{\ \alpha_n\ }
p_nB_{n+1}p_n
\xrightarrow{\ \iota_{p_n}\ }
B_{n+1},
\]
where
$p_n=
\begin{pmatrix}
1_{B_n}&0\\
0&0
\end{pmatrix}
\in B_{n+1}=M_2(B_n),$
and $\alpha_n(a)=
\begin{pmatrix}
a&0\\
0&0
\end{pmatrix}.$
For the operator data, take $H_n=L^2(B_n,\tau_n)\otimes(S_d\otimes\C^2)=\mathcal H_n$ and $D_n=\mathcal D_n$. The intermediate data are
$H_n^C=\mathcal H_n,$
$D_n^C=\mathcal D_n$, and $I_{\psi_n}=\id.$
The corner Hilbert space $H_{n+1}^{(p_n)}$ is identified with the upper-left
corner subspace of $\mathcal H_{n+1}$ through the trace-rescaled inclusion $\Lambda_{p_n}(x)\mapsto
\tau_{n+1}(p_n)^{-1/2}\Lambda_{n+1}(x)=\sqrt2\,\Lambda_{n+1}(x)$, and $D_{n+1}^{(p_n)}$ is the
restriction of $\mathcal D_{n+1}$ to that subspace. 

Since $N_{n+1}|_{\overline{\C^{m_n}}}=N_n$
on the upper-left copy $\overline{\C^{m_n}}\subseteq\overline{\C^{2m_n}},$
and since $D_\Theta^{\mathrm{std}}$ is unchanged from stage to stage, we have $I_n(\dom\mathcal D_n)\subseteq\dom\mathcal D_{n+1}$ and
$\mathcal D_{n+1}I_n=I_n\mathcal D_n.$
Thus the exact operator-intertwining hypothesis of
Theorem~\ref{thm:protorus-spectral-general} is satisfied.

We next check the uniform commutator condition from
Theorem~\ref{thm:protorus-spectral-general}\textup{(i)}. Let
\[\mathcal A_{\mathrm{sm}}
:=
\bigcup_{n\geq1}\iota_n(B_n^\infty)
\subseteq A_{\ex}^{\Theta}.\]
This is dense in $A_{\ex}^{\Theta}\cong A_\Theta\otimes\K$. For
$a=[a_{ij}]_{i,j=1}^{m_n}\in B_n^\infty=M_{m_n}(A_\Theta^\infty)$, let $a_m:=\phi_{n,m}(a)\in B_m^\infty$ for all $m\geq n$. Thus $a_m$ is the same finite matrix placed in the upper-left corner of
$M_{m_m}(A_\Theta)$. The operator $\pi_m(a_m)$ acts by left multiplication
on the left matrix index and on the $A_\Theta$-factor, and it does not act on
the right Hilbert--Schmidt index $\overline{\C^{m_m}}$. Hence
$\pi_m(a_m)$ commutes with the $N_m$-term.

The only contribution to the commutator comes from the flat $A_\Theta$-Dirac
term. More precisely, $[\mathcal D_m,\pi_m(a_m)]$
is the finite matrix of bounded operators $\bigl([D_\Theta^{\mathrm{std}},\pi_\Theta(a_{ij})]\bigr)_{i,j=1}^{m_n}$
placed in the upper-left corner of the left matrix index, tensored with
$\sigma_1$. Each entry is bounded because $a_{ij}\in A_\Theta^\infty$.
Moreover, the norm of this finite matrix commutator is independent of $m\geq n$, since passing to later stages
only embeds the same finite matrix into a larger upper-left corner. Therefore $\sup_{m\geq n}
\left\|
\overline{[\mathcal D_m,\pi_m(a_m)]}
\right\|
<\infty.$
By Theorem~\ref{thm:protorus-spectral-general}\textup{(i)}, the limit
commutators are bounded on $\mathcal A_{\mathrm{sm}}$.

Let $\mathcal H_{\ex}:=\varinjlim(\mathcal H_n,I_n)$
be the trace-rescaled GNS Hilbert-space limit. From the explicit matrix
identification above, we have
\[
\mathcal H_{\ex}
\cong
L^2(A_\Theta,\tau_\Theta)\otimes S_d
\otimes
\ell^2(\N)
\otimes
\overline{\ell^2(\N)}
\otimes
\C^2.
\]
Equivalently, $\mathcal H_{\ex}\cong H_\Theta\otimes\operatorname{HS}(\ell^2(\N))\otimes\C^2.$
The representation of $A_\Theta\otimes\K$ is the trace-GNS representation
\[
\pi_{\mathrm{GNS}}(a\otimes T)
=
(\pi_\Theta(a)\otimes1_{S_d})
\otimes
T
\otimes
1_{\overline{\ell^2(\N)}}
\otimes1_{\C^2}.
\]

Let $N$ denote the number operator on $\overline{\ell^2(\N)}$, $N\overline{e_j}=j\,\overline{e_j}.$
The inductive-limit operator is
\[
\mathcal D_{\ex}
=
D_\Theta^{\mathrm{std}}
\otimes1_{\ell^2(\N)}
\otimes1_{\overline{\ell^2(\N)}}
\otimes\sigma_1
+
1_{H_\Theta}
\otimes1_{\ell^2(\N)}
\otimes N
\otimes\sigma_2.
\]
It is selfadjoint by Theorem~\ref{thm:protorus-spectral-general}. It remains to check local compactness. Let $e_n:=\iota_n(1_{B_n})\in A_{\ex}^{\Theta}$
and $P_n:=\pi_{\mathrm{GNS}}(e_n).$
Under the above identification, $P_n=1_{H_\Theta}\otimes P_{m_n}^{\mathrm{left}}\otimes1_{\overline{\ell^2(\N)}}\otimes1_{\C^2},$
where $P_{m_n}^{\mathrm{left}}$ is the rank-$m_n$ projection onto
$\operatorname{span}\{e_1,\ldots,e_{m_n}\}\subseteq\ell^2(\N)$.

Since the two summands of $\mathcal D_{\ex}$ anticommute, one has
$\mathcal D_{\ex}^2=(D_\Theta^{\mathrm{std}})^2\otimes1\otimes1\otimes1+1\otimes1\otimes N^2\otimes1.$
After the tensor rearrangement $\mathcal H_{\ex}
\cong
\ell^2(\N)\otimes
\bigl(H_\Theta\otimes\overline{\ell^2(\N)}\otimes\C^2\bigr),$
we have

\[
P_n(1+\mathcal D_{\ex}^2)^{-1/2}
=
P_{m_n}^{\mathrm{left}}
\otimes
\left(
1+
(D_\Theta^{\mathrm{std}})^2\otimes1\otimes1
+
1\otimes N^2\otimes1
\right)^{-1/2}.
\]

The first factor $P_{m_n}^{\mathrm{left}}$ is finite-rank. The second factor is compact on $H_\Theta\otimes\overline{\ell^2(\N)}\otimes\C^2.$
Indeed, $D_\Theta^{\mathrm{std}}$ has compact resolvent and $N$ has compact
resolvent, so the operator \[\left(
1+
(D_\Theta^{\mathrm{std}})^2\otimes1\otimes1
+
1\otimes N^2\otimes1
\right)^{-1/2}\]
has eigenvalues tending to zero with finite multiplicities. Hence $P_n(1+\mathcal D_{\ex}^2)^{-1/2}\in\mathcal K(\mathcal H_{\ex})$
for every $n$. By Theorem~\ref{thm:protorus-spectral-general}\textup{(ii)},
the local compactness condition holds for all elements of
$\mathcal A_{\mathrm{sm}}$. Consequently $(\mathcal A_{\mathrm{sm}},\mathcal H_{\ex},\mathcal D_{\ex})$
is a locally compact spectral triple on $A_{\ex}^{\Theta}\cong A_\Theta\otimes\K$
in the trace-rescaled GNS representation.
\end{example}

\begin{example}[A locally compact weighted-length triple for the dimension-changing model]
\label{ex:dimension-changing-weighted-length-spectral}
 Let $\boldsymbol\theta=(\theta_j)_{j\geq1}$ be a sequence of irrational real numbers and
let $A_{\ex}:=A_{\ex}^{\boldsymbol\theta}$ be the dimension-changing toric-corner
model of Example~\ref{ex:dimension-changing}. Thus $\Theta_n:=\Theta_n^{\boldsymbol\theta}=J(\theta_1)\oplus\cdots\oplus J(\theta_n)\in M_{2n}(\R),$
and $B_n:=B_n^{\boldsymbol\theta}=M_{2^{n-1}}(A_{\Theta_n}).$
The connecting map has the form $\phi_n:=\phi_n^{\boldsymbol\theta}
=\iota_{p_n}\circ\alpha_n\circ\psi_n,$
where $\psi_n:=\id_{M_{2^{n-1}}}\otimes\varphi_n\colon B_n\longrightarrow
C_n:=M_{2^{n-1}}(A_{\Theta_{n+1}})$
is induced by the coordinate inclusion $M_n\colon \Z^{2n}\hookrightarrow \Z^{2n+2},$ $\alpha_n\colon C_n\xrightarrow{\cong}p_nB_{n+1}p_n$
is the upper-left corner identification, and $\iota_{p_n}\colon p_nB_{n+1}p_n\hookrightarrow B_{n+1}$
is the corner inclusion. Here, after identifying $B_{n+1}
=M_2(C_n),$
the projection is $p_n=
\begin{pmatrix}
1_{C_n}&0\\
0&0
\end{pmatrix}.$
Thus $p_n$ is a constant matrix projection.

Let \[G_\infty
:=
\varinjlim(\Z^{2n},M_n)
\cong
\Z^{(\infty)}
=
\bigoplus_{j\geq1}\Z e_j\]
be the limiting Fourier-label group. We first note that the standard
flat choice $L_n=I_{2n}$
does not give a locally compact spectral triple on the nonunital limit. Indeed,
let $D_n^{\mathrm{flat}}$ be the amplified standard flat operator on $L^2(B_n,\tau_n)\otimes S_{2n}$,
so that, on matrix units $E_{\alpha\beta}^{(n)}$ and Fourier monomials
$U_n^x$,
\[
D_n^{\mathrm{flat}}
(E_{\alpha\beta}^{(n)}\otimes U_n^x\otimes\xi)
=
2\pi\,E_{\alpha\beta}^{(n)}\otimes U_n^x
\otimes \gamma^{(2n)}(x)\xi.
\]
With compatible Clifford-module isometries 
\[
J_n\colon S_{2n}\to S_{2n+2},
\qquad
J_n\gamma^{(2n)}(\xi)
=
\gamma^{(2n+2)}(\iota_n\xi)J_n,
\]
where $\iota_n\colon\R^{2n}\hookrightarrow\R^{2n+2}$ is the coordinate
inclusion, these flat operators intertwine through the toric step. Since
$p_n$ is constant, the corner covariantization term vanishes, and the corner
step also intertwines. Hence the flat operators assemble to a natural
trace-GNS inductive-limit operator. However, local compactness fails. Indeed, let $e_1:=\iota_1(1_{B_1})\in A_{\ex}^{\boldsymbol\theta}$ and $P_1:=\pi(e_1)$.
For each $m\geq1$, let $e_{2m}\in\Z^{2m}$ be the $2m$-th standard basis
vector. Choose a unit vector $\xi_m\in S_{2m}$, and set $v_m:=\sqrt{2^{m-1}}
\Lambda_m(E_{11}^{(m)}\otimes U_m^{e_{2m}})
\otimes \xi_m
\in L^2(B_m,\tau_m)\otimes S_{2m}.$
The vector $v_m$ is normalized because the normalized matrix trace on
$M_{2^{m-1}}(A_{\Theta_m})$ gives $\|\Lambda_m(E_{11}^{(m)}\otimes U_m^{e_{2m}})\|^2=2^{-(m-1)}$. Moreover, the image of $1_{B_1}$ at stage $m$ is the constant matrix
projection $E_{11}^{(m)}$. Hence $I_{m,\infty}v_m\in P_1H_{\ex}$, where
$P_1=\pi(e_1)$. The vectors
$I_{m,\infty}v_m$ are mutually orthonormal in the Hilbert-space limit, since
their Fourier labels are distinct. Also $(D_m^{\mathrm{flat}})^2v_m=4\pi^2v_m.$
Therefore
\[
P_1(1+(D_{\ex}^{\mathrm{flat}})^2)^{-1/2}I_{m,\infty}v_m
=
(1+4\pi^2)^{-1/2}I_{m,\infty}v_m.
\]
Thus $P_1(1+(D_{\ex}^{\mathrm{flat}})^2)^{-1/2}$
is not compact. Consequently the standard flat dimension-changing operator has
bounded commutators on the natural smooth algebra, but it does not satisfy the
local compactness condition.

We now replace the flat metric by a weighted length which is proper on the
limiting Fourier group and add a number weight on the right Hilbert--Schmidt
matrix index. The finite-stage cut-downs will control the left matrix index.

Choose a sequence of positive weights $w_1,w_2,w_3,\ldots$
with $w_j\to\infty.$
For example, one may take $w_j=j$. Define
\[
\ell_\infty\left(\sum_{j=1}^\infty x_je_j\right)
:=
\sum_{j=1}^\infty w_j|x_j|,
\]
where only finitely many $x_j$'s are nonzero. Then $\ell_\infty$ is a
proper length function on $G_\infty$. Indeed, if $\ell_\infty(g)\leq R,$
then $x_j=0$ whenever $w_j>R$. Since only finitely many $j$'s satisfy
$w_j\leq R$, and for each such $j$ the coordinate $x_j$ has only finitely
many possible values, the ball $\{g\in G_\infty:\ell_\infty(g)\leq R\}$
is finite. Moreover, $\ell_\infty$ has bounded translation increments as, for fixed
$h\in G_\infty$, $|\ell_\infty(g+h)-\ell_\infty(g)|\leq\ell_\infty(h),$ for all
$g\in G_\infty.$

For each $n$, let $\ell_n\colon \Z^{2n}\to[0,\infty)$
be the restriction of $\ell_\infty$ to the first $2n$ coordinates:
\[
\ell_n(x_1,\ldots,x_{2n})
=
\sum_{j=1}^{2n}w_j|x_j|.
\]
Then
$\ell_{n+1}(M_nx)=\ell_n(x),$ for all $x\in\Z^{2n}.$

Set $m_n:=2^{n-1},$ so
$B_n=M_{m_n}(A_{\Theta_n})$, and 
let \(E_{\alpha\beta}^{(n)}\), \(1\leq\alpha,\beta\leq m_n\), denote the
standard matrix units in $M_{m_n}$. Let $\lambda_1,\lambda_2,\lambda_3,\ldots$
be another sequence of positive numbers with $\lambda_\beta\to\infty$ (for instance, one may take $\lambda_\beta=\beta$).
The normalized trace on $B_n$ is
$\tau_n=\frac{1}{m_n}\Tr_{m_n}\otimes\tau_{\Theta_n}.$
For $x\in\Z^{2n}$, let $U_n^x$ denote the corresponding Fourier monomial in
$A_{\Theta_n}$. The vectors
\[
\varepsilon_{\alpha\beta,x}^{(n)}
:=
\sqrt{m_n}\,
\Lambda_n(E_{\alpha\beta}^{(n)}\otimes U_n^x),
\qquad
1\leq\alpha,\beta\leq m_n,\quad x\in\Z^{2n},
\]
form an orthonormal basis of $L^2(B_n,\tau_n)$.

Define an unbounded diagonal operator $D_n$ on $H_n:=L^2(B_n,\tau_n)$
by
\[
D_n\varepsilon_{\alpha\beta,x}^{(n)}
:=
\bigl(\ell_n(x)+\lambda_\beta\bigr)
\varepsilon_{\alpha\beta,x}^{(n)}.
\]
with $\dom(D_n)
=
\left\{
\sum c_{\alpha\beta,x}\varepsilon_{\alpha\beta,x}^{(n)}:
\sum
(\ell_n(x)+\lambda_\beta)^2|c_{\alpha\beta,x}|^2<\infty
\right\}.$
Since $D_n$ is diagonal with real eigenvalues, it is selfadjoint. It has
compact resolvent: for $R>0$, the set of basis vectors satisfying $\ell_n(x)+\lambda_\beta\leq R$
is finite, because $m_n$ is finite and $\ell_n$ is proper on \(\Z^{2n}\).

We next check bounded commutators at the finite stages. Let
$a\in B_n^\infty:=M_{m_n}(A_{\Theta_n}^\infty)$. Then $a$ has a rapidly decaying Fourier expansion
\[
a=
\sum_{\alpha,\beta=1}^{m_n}
\sum_{y\in\Z^{2n}}
a_{\alpha\beta}(y)\,
E_{\alpha\beta}^{(n)}\otimes U_n^y.
\]
Left multiplication by a monomial $E_{ij}^{(n)}\otimes U_n^y$
changes the row index $\alpha$ to $i$, leaves the column index $\beta$
unchanged, and shifts the Fourier label by $y$, up to a scalar cocycle. Thus
the corresponding commutator with $D_n$ has multiplier
$(\ell_n(x+y)+\lambda_\beta)-(\ell_n(x)+\lambda_\beta)=\ell_n(x+y)-\ell_n(x).$
Since
$|\ell_n(x+y)-\ell_n(x)|\leq \ell_n(y),$
the commutator with each monomial is bounded. For a general smooth element
$a$, the rapid decay of the coefficients gives
\[
\sum_{\alpha,\beta=1}^{m_n}
\sum_{y\in\Z^{2n}}
|a_{\alpha\beta}(y)|\,\ell_n(y)<\infty,
\]
so $[D_n,\pi_n(a)]$ extends boundedly. Hence $B_n^\infty\subseteq B_n^\infty(D_n).$

We now verify that these data are Morita-compatible in the sense of
Definition~\ref{def:morita-compatible-dirac-data}. 
Recall that $C_n=M_{m_n}(A_{\Theta_{n+1}})$.
The toric part of the connecting map is $\psi_n:=\id_{M_{m_n}}\otimes\varphi_n\colon B_n\to C_n.$
On $C_n$, use the corresponding diagonal operator $D_n^C$ defined by
\[
D_n^C\varepsilon_{\alpha\beta,y}^{C,n}
=
\bigl(\ell_{n+1}(y)+\lambda_\beta\bigr)
\varepsilon_{\alpha\beta,y}^{C,n},
\qquad
y\in\Z^{2n+2}.
\]
The GNS isometry for $\psi_n$ sends $\varepsilon_{\alpha\beta,x}^{(n)}$
to the corresponding vector with Fourier label $M_nx$, up to a phase. Since $\ell_{n+1}(M_nx)=\ell_n(x),$
and since the column index $\beta$ is unchanged, we get $D_n^C I_{\psi_n}=I_{\psi_n}D_n.$

The corner step is the upper-left matrix inclusion. Under the identification $B_{n+1}=M_2(C_n),$
let $p_n=
\begin{pmatrix}
1_{C_n}&0\\
0&0
\end{pmatrix}$ as before. Let $D_{n+1}^{(p_n)}$ be the restriction of $D_{n+1}$ to the upper-left
corner subspace. Since $p_n$ is a constant matrix projection, this subspace
reduces $D_{n+1}$. The trace-scaling constant is $t_n=\tau_{n+1}(p_n)=\frac12.$ The trace-rescaled corner-inclusion isometry sends normalized basis vectors to the same normalized basis vectors in the next stage:
\[
I_n\varepsilon_{\alpha\beta,x}^{(n)}
=
\omega_n(x)\,
\varepsilon_{\alpha\beta,M_nx}^{(n+1)}
\]
for some phase $\omega_n(x)\in\T$. Therefore $D_{n+1}I_n=I_nD_n$
on $\dom(D_n)$. Thus the operators $D_n$ define a Morita-compatible Dirac structure in the
sense of Definition~\ref{def:morita-compatible-dirac-data}.
Hence they assemble to a selfadjoint
operator $D_{\ex}$
on the Hilbert-space inductive limit $H_{\ex}:=\varinjlim(H_n,I_n).$

The limit Hilbert space has a concrete description. The embeddings preserve the matrix indices $\alpha,\beta$ and send the
Fourier label $x\in\Z^{2n}$ to its class in
$G_\infty=\Z^{(\infty)}$, up to the harmless scalar phases coming from the
toric monomial maps. After absorbing these phases into the Fourier basis of the
Hilbert-space limit, we obtain $H_{\ex}\cong \ell^2(\N\times\N\times G_\infty)$,
and under this identification
\[
D_{\ex}\delta_{\alpha,\beta,g}
=
\bigl(\ell_\infty(g)+\lambda_\beta\bigr)
\delta_{\alpha,\beta,g},
\qquad
\alpha,\beta\in\N,\quad g\in G_\infty.
\]

Let $\mathcal A_{\mathrm{sm}}:=\bigcup_{n\geq1}\iota_n(B_n^\infty)
\subseteq A_{\ex}^{\boldsymbol\theta}.$
We claim that $[D_{\ex},\pi(a)]$
is bounded for every $a\in\mathcal A_{\mathrm{sm}}$. It suffices to take
$a\in\iota_n(B_n^\infty)$. Write $a$ at stage $n$ as above, with Fourier
coefficients $a_{\alpha\beta}(y)$. In the limit representation, left
multiplication by a monomial matrix coefficient changes only finitely many row
indices, leaves the column index $\beta$ unchanged, and shifts $g$ by the
embedded vector $y\in G_\infty$. Thus the commutator multiplier is $\ell_\infty(g+y)-\ell_\infty(g)$,
whose absolute value is at most $\ell_\infty(y)=\ell_n(y)$. The rapid decay of
the Fourier coefficients again gives a finite bound. Hence the commutator is
bounded.

We now prove local compactness. Let $e_n:=\iota_n(1_{B_n})\in A_{\ex}$ and $P_n:=\pi(e_n).$
In the concrete representation above, $P_n$ is the projection onto the
subspace spanned by $\delta_{\alpha,\beta,g}$ with
$1\leq\alpha\leq m_n,$ $\beta\in\N,$ and $g\in G_\infty.$
Indeed, $e_n$ is the constant matrix projection onto the first $m_n$ rows.

For $R>0$, the range of $P_n\,1_{[0,R]}(D_{\ex})$
is spanned by those basis vectors with $1\leq\alpha\leq m_n$ and $\ell_\infty(g)+\lambda_\beta\leq R.$
There are only finitely many such vectors: the index $\alpha$ ranges over a finite set,
the index $\beta$ ranges over a finite set because $\lambda_\beta\to\infty$, and
$g$ ranges over a finite set because $\ell_\infty$ is proper. Hence $P_n\,1_{[0,R]}(D_{\ex})$
has finite rank for every $R>0$. Therefore $P_n(1+D_{\ex}^2)^{-1/2}$
is compact: it is the norm limit, as $R\to\infty$, of the finite-rank
operators $P_n(1+D_{\ex}^2)^{-1/2}1_{[0,R]}(D_{\ex}),$
and the complementary norm is bounded by $(1+R^2)^{-1/2}$.

Finally, if $a\in\iota_n(B_n^\infty)$, then $a=e_nae_n,$
so $\pi(a)=\pi(a)P_n.$
Thus $\pi(a)(1+D_{\ex}^2)^{-1/2}=\pi(a)P_n(1+D_{\ex}^2)^{-1/2}$
is compact. Hence $(\mathcal A_{\mathrm{sm}},H_{\ex},D_{\ex})$
is a locally compact spectral triple on $A_{\ex}^{\boldsymbol\theta}.$

Notice that $D_{\ex}$ does not have compact resolvent. Indeed, for fixed
$\beta=1$ and $g=0$, the vectors $\delta_{\alpha,1,0},$ where 
$ \alpha\in\N,$
are mutually orthonormal eigenvectors with the same eigenvalue $\lambda_1$.
Thus $(1+D_{\ex}^2)^{-1/2}$ is not compact. 
\end{example}

The construction below is inspired by the filtration method of
Christensen--Ivan for AF algebras \cite{ChristensenIvanAF}.  In the AF case one
uses an increasing finite-dimensional filtration and defines a Dirac operator by
assigning increasing eigenvalues to the orthogonal differences of the
filtration.  Here the algebra is not AF, but the nonunital toric system carries a
natural finite flag of support projections at each stage.  We apply the same
filtration idea blockwise to the Peirce spaces determined by this flag.  The
resulting operators are filtration-adapted rather than geometric
Heisenberg-module Dirac operators.

\begin{example}
\label{ex:AX7-flag-filtration-spectral}
Let $A_{\ex}^{N,\theta_0}=\varinjlim(A_{\theta_n},\phi_n)$
be the same-dimensional nonunital noncorner family of
Example~\ref{ex:AX7-noncorner}. Thus $\theta_{n+1}=\frac{\theta_n}{N+\theta_n},$ for all $n\ge 0$, and the connecting map factors as
\[
A_{\theta_n}
\xrightarrow{\ \psi_n\ }
A_{\eta_n}
\xrightarrow{\ \chi_n\ }
q_nA_{\theta_{n+1}}q_n
\xrightarrow{\ \iota_{q_n}\ }
A_{\theta_{n+1}}.
\]
Here $\psi_n$ is the unital toric map induced by $M_N=
\begin{pmatrix}
1&0\\
0&N
\end{pmatrix},$
and $q_n\in A_{\theta_{n+1}}^\infty$ is a smooth Rieffel projection with $\tau_{\theta_{n+1}}(q_n)=1-\theta_{n+1}.$ We choose the Rieffel corner isomorphism so that it restricts to a Fréchet
$*$-isomorphism $
\chi_n^\infty\colon A_{\eta_n}^\infty
\xrightarrow{\cong}
q_nA_{\theta_{n+1}}^\infty q_n.$ This is possible by the smooth Morita-equivalence picture for
noncommutative tori; see \cite{MR2350069}.

We now pass to the nonunital flag. For $m\ge n$, write $\phi_{n,m}:=\phi_{m-1}\circ\cdots\circ\phi_n$ and
$\phi_{n,n}:=\id_{A_{\theta_n}}$, and define the support projections $p_{n,m}:=\phi_{n,m}(1_{A_{\theta_n}})\in A_{\theta_m}.$
For fixed $m$, these projections form an increasing smooth flag:
$p_{0,m}\le p_{1,m}\le\cdots\le p_{m,m}=1.$
Set $r_{0,m}:=p_{0,m},$ and 
$r_{i,m}:=p_{i,m}-p_{i-1,m},$ for $1\le i\le m.$ Then $1=\sum_{i=0}^m r_{i,m},$
and the $r_{i,m}$'s are pairwise orthogonal smooth projections.

Let $H_m:=L^2(A_{\theta_m},\tau_{\theta_m}).$
The Peirce decomposition associated to the flag is
\[
H_m
=
\bigoplus_{i,j=0}^m H_{i,j}^{(m)},\;\text{where}
\quad
H_{i,j}^{(m)}
:=
\overline{\Lambda_m(r_{i,m}A_{\theta_m}r_{j,m})}.
\]
This decomposition is orthogonal. Indeed, if $x\in r_{i,m}A_{\theta_m}r_{j,m}$ and 
$y\in r_{k,m}A_{\theta_m}r_{\ell,m},$
then $\langle \Lambda_m(x),\Lambda_m(y)\rangle=\tau_{\theta_m}(y^*x).$ Now $y^*x
\in r_{\ell,m}A_{\theta_m}r_{k,m}r_{i,m}A_{\theta_m}r_{j,m}.$
Thus $y^*x=0$ unless $k=i$. If $k=i$, then $y^*x\in r_{\ell,m}A_{\theta_m}r_{j,m},$
and traciality gives $\tau_{\theta_m}(y^*x)=\tau_{\theta_m}(r_{j,m}r_{\ell,m}y^*x),$
which vanishes unless $j=\ell$. Hence the Peirce blocks are mutually
orthogonal. Since $1=\sum_i r_{i,m}$, these blocks also exhaust $H_m$.

The trace-rescaled GNS embedding $I_m:H_m\to H_{m+1}$
sends $H_{i,j}^{(m)}$ isometrically into $H_{i,j}^{(m+1)}$ for all $0\le i,j\le m$.  Indeed, since
$\phi_m(p_{i,m})=p_{i,m+1},$ we have $\phi_m(r_{i,m})=r_{i,m+1}$ for all $0\leq i\leq m$, and hence $\phi_m(r_{i,m}A_{\theta_m}r_{j,m})
\subseteq
r_{i,m+1}A_{\theta_{m+1}}r_{j,m+1}.$
Therefore, for each pair $i,j\ge0$, we obtain a Hilbert-space inductive limit
\[
H_{i,j}^{(\infty)}
:=
\varinjlim_{m\ge \max\{i,j\}} H_{i,j}^{(m)}.
\]

We shall construct selfadjoint compact-resolvent operators $D_{i,j}^{(\infty)}$
on $H_{i,j}^{(\infty)}$ such that each finite-stage subspace \(H_{i,j}^{(m)}\) reduces
$D_{i,j}^{(\infty)}$. We then define
\[
D_{i,j}^{(m)}
:=
D_{i,j}^{(\infty)}\big|_{H_{i,j}^{(m)}}.
\]
Let
\[
\mathcal A_{\mathrm{flag}}
:=
*\text{-alg}\left(\bigcup_{n\geq0}\iota_n(A_{\theta_n}^{\mathrm{alg}})\right)
\subseteq A_{\ex}^{N,\theta_0}.
\]
This is a dense $*$-subalgebra. We use the generated algebra rather than the
plain union because the Rieffel corner identifications are smooth but need not
send Fourier polynomials to Fourier polynomials.  Since all connecting maps are
smooth, every element of $\mathcal A_{\mathrm{flag}}$ is represented by a
smooth element at some finite stage.

Choose a countable $*$-closed set $\mathcal W=\{a_1,a_2,a_3,\ldots\}\subseteq \mathcal A_{\mathrm{flag}}$
whose complex linear span is $\mathcal A_{\mathrm{flag}}$. For each $a_s$,
choose a stage $n(s)$ and a smooth representative $a_s^{(n(s))}\in A_{\theta_{n(s)}}^\infty$
such that $a_s=\iota_{n(s)}(a_s^{(n(s))}).$
For $m\geq n(s)$, put $a_{s,m}:=\phi_{n(s),m}(a_s^{(n(s))})\in A_{\theta_m}^\infty.$
Then $a_{s,m}=p_{n(s),m}a_{s,m}p_{n(s),m}.$

For $m\geq n(s)$, $0\leq i,k\leq n(s)$, and $0\leq j\leq m$, the block $r_{i,m}a_{s,m}r_{k,m}$
acts by left multiplication as a bounded operator $L_{i,k,j}^{(m)}(a_s):H_{k,j}^{(m)}
\longrightarrow H_{i,j}^{(m)}.$
These block operators are compatible with the GNS embeddings and therefore
define bounded limit operators $L_{i,k,j}^{(\infty)}(a_s):H_{k,j}^{(\infty)}\longrightarrow
H_{i,j}^{(\infty)}.$ Moreover, each $L_{i,k,j}^{(\infty)}(a_s)$ is a compression of left
multiplication by $a_s$, so $\|L_{i,k,j}^{(\infty)}(a_s)\|\leq \|a_s\|.$

We now construct filtrations simultaneously for all Peirce blocks. For each
pair $i,j\geq0$, set $m_0(i,j):=\max\{i,j\}.$
Let $K_{i,j}^{(m)}:=H_{i,j}^{(m)}\ominus H_{i,j}^{(m-1)}$ for
$m>m_0(i,j)$,
and set $K_{i,j}^{(m_0(i,j))}:=H_{i,j}^{(m_0(i,j))}.$
Then
\[
H_{i,j}^{(\infty)}
=
\bigoplus_{m\geq m_0(i,j)}K_{i,j}^{(m)}.
\]

For each $i,j$, choose a countable dense set of finite-stage vectors in
$H_{i,j}^{(\infty)}$, meaning vectors supported in finite direct sums of the
spaces $K_{i,j}^{(m)}$. We construct finite-dimensional subspaces
\[
V_{i,j,0}\subseteq V_{i,j,1}\subseteq V_{i,j,2}\subseteq\cdots
\subseteq H_{i,j}^{(\infty)}
\]
simultaneously for all $i,j$, with dense union, satisfying the following two
properties:
\smallskip

\noindent
 (1) Each $V_{i,j,\ell}$ is compatible with the decomposition into
$K$-summands:
\[
V_{i,j,\ell}
=
\bigoplus_{m\geq m_0(i,j)}
\bigl(V_{i,j,\ell}\cap K_{i,j}^{(m)}\bigr).
\]
In particular, every finite-stage subspace $H_{i,j}^{(m)}$ reduces the
diagonal operator defined below.
\smallskip

\noindent
 (2) For every $s$, every right index $j$, and all relevant $i,k$,
there is a constant $R_s$, independent of $j$, such that $L_{i,k,j}^{(\infty)}(a_s)V_{k,j,\ell}\subseteq V_{i,j,\ell+R_s}$
for every $\ell$. The same finite-propagation condition is required for the
adjoint block operator.

The construction is recursive. At step $\ell+1$, for each pair $i,j$, we
enlarge $V_{i,j,\ell}$ by adding the first $\ell+1$ chosen dense vectors in
$H_{i,j}^{(\infty)}$, and also adding $L_{i,k,j}^{(\infty)}(a_s)V_{k,j,\ell}$
and $\bigl(L_{k,i,j}^{(\infty)}(a_s)\bigr)^*V_{k,j,\ell}$
for all $s\leq \ell+1$ and all relevant $k$. We then add all orthogonal
$K_{i,j}^{(m)}$-components of the vectors introduced.

This remains finite-dimensional at each step. Indeed, only finitely many
vectors and finitely many operators are used at step $\ell+1$. Moreover, a
finite-stage vector is sent by a finite-stage left multiplication operator to a
finite-stage vector, so taking $K$-components produces only finitely many new
components. The recursion also gives the finite-propagation property with
$R_s$ independent of the right index $j$: for example, one may take
$R_s=s+1$. If $\ell\geq s$, then the recursive construction at step
$\ell+1$ adds $L_{i,k,j}^{(\infty)}(a_s)V_{k,j,\ell}$
to $V_{i,j,\ell+1}$. If $\ell<s$, then
$V_{k,j,\ell}\subseteq V_{k,j,s}$, and the construction adds its image by
step $s+1$. Hence $L_{i,k,j}^{(\infty)}(a_s)V_{k,j,\ell}\subseteq V_{i,j,\ell+R_s}$
for all $\ell$, with $R_s$ independent of $j$. The same argument applies
to the adjoint block operators.

Let $Q_{i,j,\ell}$
be the orthogonal projection onto $V_{i,j,\ell}\ominus V_{i,j,\ell-1},$ where $V_{i,j,-1}:=\{0\}.$ Define
\[
D_{i,j}^{(\infty)}
:=
\sum_{\ell=0}^{\infty}\ell\, Q_{i,j,\ell}
\]
with domain $\dom(D_{i,j}^{(\infty)})=\left\{
\xi\in H_{i,j}^{(\infty)}:
\sum_{\ell=0}^{\infty}\ell^2\|Q_{i,j,\ell}\xi\|^2<\infty
\right\}.$
Then $D_{i,j}^{(\infty)}$ is selfadjoint. Since each $Q_{i,j,\ell}$ has
finite rank and the eigenvalues $\ell$ tend to infinity,
$D_{i,j}^{(\infty)}$ has compact resolvent. 

Since each $V_{i,j,\ell}$ is compatible with the decomposition $H_{i,j}^{(\infty)}=\bigoplus_m K_{i,j}^{(m)},$
every finite-stage subspace $H_{i,j}^{(m)}$ reduces
$D_{i,j}^{(\infty)}$. Hence
\[
D_{i,j}^{(m)}
:=
D_{i,j}^{(\infty)}\big|_{H_{i,j}^{(m)}}
\]
is selfadjoint and satisfies
$D_{i,j}^{(m+1)}I_m=I_mD_{i,j}^{(m)}$
on $\dom(D_{i,j}^{(m)})$. 

We next record the commutator estimate for the block operators. Suppose $T:H_{k,j}^{(\infty)}\to H_{i,j}^{(\infty)}$
has propagation at most $R$ with respect to the filtrations, and suppose that
$T^*$ also has propagation at most $R$. Write $T_{pq}:=Q_{i,j,p}TQ_{k,j,q}.$
The propagation assumptions imply $T_{pq}=0$
whenever $|p-q|>R.$ Set $T_r:=\sum_q Q_{i,j,q+r}TQ_{k,j,q}.$
Then, on the common algebraic core, $D_{i,j}^{(\infty)}T-TD_{k,j}^{(\infty)}=\sum_{|r|\leq R} r\,T_r.$

Each $T_r$ is bounded and satisfies $\|T_r\|\leq \|T\|.$
Therefore 
\[
\left\|
D_{i,j}^{(\infty)}T-TD_{k,j}^{(\infty)}
\right\|
\leq
\sum_{|r|\leq R}|r|\,\|T\|
=
R(R+1)\|T\|.
\]
Applying this to $T=L_{i,k,j}^{(\infty)}(a_s)$ gives
$\left\|
D_{i,j}^{(\infty)}L_{i,k,j}^{(\infty)}(a_s)-L_{i,k,j}^{(\infty)}(a_s)D_{k,j}^{(\infty)}
\right\|
\leq
R_s(R_s+1)\|a_s\|,$
which is independent of $m$ and $j$. Restricting to finite stages
gives
\[
\sup_{m\geq n(s)}
\left\|
D_{i,j}^{(m)}\pi_m(r_{i,m}a_{s,m}r_{k,m})
-
\pi_m(r_{i,m}a_{s,m}r_{k,m})D_{k,j}^{(m)}
\right\|
<\infty
\]
for all relevant $i,k,j$.

Now choose a sequence of positive numbers $\lambda_0,\lambda_1,\lambda_2,\ldots$
with $\lambda_j\to\infty.$
At finite stage $m$, define $D_m:=\bigoplus_{i,j=0}^m\left(D_{i,j}^{(m)}+\lambda_j\right)$
on $H_m=\bigoplus_{i,j=0}^m H_{i,j}^{(m)}.$ This is a finite direct sum of selfadjoint compact-resolvent operators, hence
$D_m$ is selfadjoint with compact resolvent. The right-block weight $\lambda_j$ is scalar on each Peirce block and is
preserved by the connecting maps. Since $I_m$ sends $H_{i,j}^{(m)}$ into $H_{i,j}^{(m+1)}$, we get $I_m(\dom (D_m))\subseteq\dom (D_{m+1})$ and 
$D_{m+1}I_m=I_mD_m$ on $\dom(D_m)$.
Thus the operators $D_m$ define a Morita-compatible Dirac structure in the sense of Definition~\ref{def:morita-compatible-dirac-data}. For the intermediate Morita factors, one may
take $D_{n+1}^{(q_n)}$ to be the restriction of $D_{n+1}$ to the closed
corner subspace corresponding to $q_nA_{\theta_{n+1}}q_n,$
and take $D_n^C$ to be its transport through the smooth corner isomorphism
$\chi_n$. Since $I_n=I_{\iota_{q_n}}\circ I_{\chi_n}\circ I_{\psi_n}$
and $D_{n+1}I_n=I_nD_n$, the intertwinings required in
Definition~\ref{def:morita-compatible-dirac-data} hold after this choice.

Let $H_{\ex}:=\varinjlim(H_m,I_m)$
and let $D_{\ex}:=\overline{\varinjlim D_m}.$
We claim that $
[D_{\ex},\pi(a)]\in\mathcal B(H_{\ex})$
for every $a\in\mathcal A_{\mathrm{flag}}$. It suffices to prove this for
$a=a_s$, since $\mathcal A_{\mathrm{flag}}$ is the complex linear span of
the $a_s$'s.

Fix $s$, put $n=n(s)$, and write $a_m:=a_{s,m}=\phi_{n,m}(a_s^{(n)})\in A_{\theta_m},$ for all
$m\geq n.$ Since $a_m=p_{n,m}a_mp_{n,m}$
and $p_{n,m}=\sum_{i=0}^{n}r_{i,m},$
we have $a_m=\sum_{i,k=0}^{n} r_{i,m}a_mr_{k,m}.$
Thus, on the Peirce decomposition, left multiplication by $a_m$ has nonzero
entries only between left indices $0,\ldots,n$, while it preserves the right
index $j$.

On $H_{k,j}^{(m)}$, the operator $D_m$ is $D_{k,j}^{(m)}+\lambda_j,$
and on $H_{i,j}^{(m)}$, it is $D_{i,j}^{(m)}+\lambda_j.$
Therefore the scalar $\lambda_j$ cancels in the commutator:
\[
\begin{aligned}
&
\bigl(D_{i,j}^{(m)}+\lambda_j\bigr)
\pi_m(r_{i,m}a_mr_{k,m})
-
\pi_m(r_{i,m}a_mr_{k,m})
\bigl(D_{k,j}^{(m)}+\lambda_j\bigr)
\\
&\qquad =
D_{i,j}^{(m)}\pi_m(r_{i,m}a_mr_{k,m})
-
\pi_m(r_{i,m}a_mr_{k,m})D_{k,j}^{(m)}.
\end{aligned}
\]
By the block estimate above, these block commutators are uniformly bounded in
$m$ and $j$. Since only finitely many pairs $0\leq i,k\leq n$ occur, we
get $\sup_{m\geq n}
\left\|
[D_m,\pi_m(a_m)]
\right\|
<\infty.$
The inductive-limit commutator criterion of Theorem~\ref{thm:protorus-spectral-general} then implies that
$[D_{\ex},\pi(a_s)]$
extends boundedly. By linearity, $[D_{\ex},\pi(a)]$ is bounded for every
$a\in\mathcal A_{\mathrm{flag}}$.

We now prove local compactness. Let $e_n:=\iota_n(1_{A_{\theta_n}})$ and
$P_n:=\pi(e_n)$.
By Theorem~\ref{thm:protorus-spectral-general}, it suffices to prove $P_n(1+D_{\ex}^2)^{-1/2}\in\mathcal K(H_{\ex})$
for every $n$. Indeed, if $a\in\mathcal A_{\mathrm{flag}}$ is represented
at stage $n$, then $a=e_nae_n,$
so $\pi(a)=\pi(a)P_n.$
Thus $\pi(a)(1+D_{\ex}^2)^{-1/2}
=
\pi(a)P_n(1+D_{\ex}^2)^{-1/2},$ which is compact once the compactness for $P_n$ is known.

In the limit Peirce decomposition, $P_nH_{\ex}=\bigoplus_{\substack{0\leq i\leq n\\ j\geq0}}H_{i,j}^{(\infty)}.$
Thus $P_n$ cuts down the left Peirce index $i$ to the finite set
$\{0,\ldots,n\}$, but it does not cut down the right index $j$. This is why
the right-block weights $\lambda_j\to\infty$ are needed.

Let $E_R:=1_{[0,R]}(D_{\ex})$
be the spectral projection of $D_{\ex}$. Since $D_{\ex}=\bigoplus_{i,j\geq0}
\left(D_{i,j}^{(\infty)}+\lambda_j\right),$
and all summands are nonnegative, the projection $P_nE_R$ is supported only
on blocks satisfying $0\leq i\leq n$ and $\lambda_j\leq R.$
There are only finitely many such $j$'s because $\lambda_j\to\infty$. For
each fixed pair $(i,j)$, the operator $D_{i,j}^{(\infty)}$ has compact
resolvent, so the spectral projection of $D_{i,j}^{(\infty)}+\lambda_j$
for a bounded interval is finite-rank. Hence $P_nE_R$ is a finite direct sum
of finite-rank projections, and therefore $P_nE_R$ is finite-rank.

Now write
\[
P_n(1+D_{\ex}^2)^{-1/2}
=
P_n(1+D_{\ex}^2)^{-1/2}E_R+
P_n(1+D_{\ex}^2)^{-1/2}(1-E_R).
\]
The first term is finite-rank because $P_nE_R$ is finite-rank. The second term
has norm bounded by $(1+R^2)^{-1/2}.$
Letting $R\to\infty$, we see that $P_n(1+D_{\ex}^2)^{-1/2}$
is compact. Therefore, for every $a\in\mathcal A_{\mathrm{flag}}$, $\pi(a)(1+D_{\ex}^2)^{-1/2}\in\mathcal K(H_{\ex}).$

Consequently $(\mathcal A_{\mathrm{flag}},H_{\ex},D_{\ex})$
is a locally compact spectral triple on $A_{\ex}^{N,\theta_0}$.

\end{example}


\end{document}